\pdfoutput=1
\documentclass[a4paper,11pt]{article}	
\input{mystyle.sty}		

\title{Hitchin grafting representations I: Geometry}
\date{\today}
\author{Pierre-Louis Blayac, Ursula Hamenstädt, and Théo Marty}

\begin{document}

	\maketitle

\begin{abstract}
We give a geometric interpretation of Fock--Goncharov positivity and show that bending deformations of Fuchsian representations stabilize a uniform Finsler quasi-convex disk in the symmetric space $\mathrm{PSL}_d(\mathbb R)/\mathrm{PSO}(d)$.
     \end{abstract}

 { 
    \hypersetup{linkcolor=black}
    \renewcommand{\baselinestretch}{0.85}
    \normalsize
\setcounter{tocdepth}{2}
    \tableofcontents
    \renewcommand{\baselinestretch}{1.0}
    \normalsize
 }

\section*{Introduction}
\addcontentsline{toc}{section}{Introduction}

The \emph{Teichm\"uller space} ${\mathcal T}(S)$ of a closed oriented surface $S$ 
of genus $g\geq 2$ is the space of \emph{marked} hyperbolic structures on $S$, homeomorphic to $\R^{6g-6}$. 
Equivalently, it can be described as a distinguished component of the space of conjugacy classes of homomorphisms $\pi_1(S)\to \PSL_2(\mathbb{R})$, with target the group $\PSL_2(\mathbb{R})$ of orientation preserving isometries of the hyperbolic plane. 
It was discovered by Hitchin \cite{Hitchin} that 
an analog of the Teichm\"uller space also exists for conjugacy classes of 
representations of $\pi_1(S)$ into simple split real Lie groups of higher rank, which is also homeomorphic to a Euclidean space.

The so-called \emph{Hitchin component} ${\rm Hit}(S)$ for the target group $\PSL_d(\mathbb{R})$ $(d\geq 3)$
is the component of the \emph{character variety} containing conjugacy classes of 
so-called \emph{Fuchsian representations}, namely 
discrete representations 
which factor through an irreducible embedding $\PSL_2(\mathbb{R})\to \PSL_d(\mathbb{R})$. 
Hitchin~\cite{Hitchin} showed that the Hitchin component is homeomorphic to $\R^m$ for 
$m=(2g-2){\rm \dim}(\PSL_d(\R))$, and later Labourie~\cite{Lab06} and Fock--Goncharov~\cite{FG06} independently proved that all representations in the Hitchin component are faithful with discrete image. Thus each point in the Hitchin component defines a locally 
symmetric manifold with fundamental group isomorphic to $\pi_1(S)$. 
The mapping class group ${\rm Mod}(S)$ of $S$ of isotopy classes of diffeomorphisms of $S$ 
acts properly by precomposition on the Hitchin component. 
%


In this article we study the geometry of a distinguished class of 
these locally symmetric manifolds,
obtained by deforming Fuchsian representations via the famous \emph{bending procedure} described below, which is in our opinion the simplest kind of deformation.
We will see that the geometry of such manifolds is governed by the geometry of an explicit piecewise totally geodesic embedded subsurface.
This subsurface will be constructed by grafting flat cylinders into a hyperbolic surface.
In a sequel to this work \cite{BHMM25}, this will be used to investigate the so-called pressure metric on the Hitchin component defined by Bridgeman--Canary--Labourie--Sambarino \cite{BCLS15}. 

Given a hyperbolic metric $X$ on $S$ and a simple closed geodesic $\gamma\subset X$, one can define two types of deformations of $X$. 
The first consists in shrinking the length of $\gamma$ (this depends on the choice of 
a pair of pant in $S$ on each side of $\gamma$).
Another deformation consists in shearing the metric 
along $\gamma$, that is, rotating  along $\gamma$ the two components of $A\setminus \gamma$ where $A\subset S$ is an annulus containing $\gamma$ as its core curve.
This is the type of deformation we study here, adapted to the higher rank setting.
Beyond their simplicity, shearing deformations have many interesting features: they are used to define the famous Fenchel--Nielsen coordinates on the Teichmüller space; they define a symplectic flow on the Teichmüller space; they can be generalised into \emph{earthquakes} (shearing along measured geodesic laminations), which give another coordinate system for the Teichmüller space.

There is a more group theoretic interpretation of shearing.
Namely, 
if $\gamma$ is separating, then $\pi_1(S)$ splits into an amalgamated product and shearing can be thought of as a partial conjugation (conjugate only one factor of the product) of the representation of $\pi_1(S)$ 
into $\PSL_2(\R)$ defining $X$,
under an element in the infinite cyclic centralizer of $\gamma$ in $\PSL_2(\R)$.
If $\gamma$ is non-separating, then $\pi_1(S)$ is an HNN-extension and there is an analogous interpretation of shearing. 
This interpretation of shearing immediately extends to the Hitchin component, where 
now the (identity component of the) centralizer of $\gamma\in\PSL_d(\R)$ 
is conjugate to the group of matrices $\exp(z)$ where $z$ is a trace-free diagonal matrix.
The resulting deformations of Fuchsian representations are called
\emph{bending} or \emph{bulging} deformations,
and we call \emph{bending parameter} the matrix $z$ used to perform the bending.
Such deformations have frequently been looked at in the literature, see for example 
\cite{Go86, FK16, AZ23}, and they correspond to the \emph{rational} deformations of Fock and Goncharov \cite{FG06}. 
Bending can also be carried out along a simple geodesic multicurve $\gamma$, and then the bending transformation consists 
of a $k$-tuple of transformations where $k$ is the number of components of $\gamma$.
Bending techniques have been used in other geometric contexts as well, for example to study 
Kleinian groups (see, e.g. \cite{ThurstonNotes}),  pseudohyperbolic geometry \cite{Mess90}, or projective geometry \cite{JM86}.


Our geometric study involves, instead of the usual Riemannian metric, a $\PSL_d(\R)$-invariant Finsler
metric $\mf F$ on the symmetric space $\X=\PSL_d(\R)/{\rm PO}(d)$.
Consider the convex cone $\overline{\mathfrak a}^+$ of trace-free diagonal matrices with entries $x_1\geq \dots\geq x_d$ in descending order.
It can be seen as a subset of the tangent bundle $T\X$ (in the tangent space of a fixed basepoint), where it is a fundamental set for the action of $\PSL_d(\R)$.
Then any positive linear functional $\alpha_0$ on 
$\overline{\mathfrak a}^+$ satisfying symmetry and convexity assumptions (see Notation~\ref{nota:alpha}), for instance
\begin{equation}\label{eq:ex of alpha0}\alpha_0(x)=(d-1)x_1+(d-3)x_2+\dots + (1-d)x_d,\end{equation}
determines a $\PSL_d(\R)$-invariant Finsler metric $\mf F$ on $\X$.
We denote by $d^{\mf F}$ the distance function induced by ${\mf F}$; this function is bi-Lipschitz equivalent to the 
distance function of the usual Riemannian metric. 

Following Thurston (see \cite{T97} for details), define the \emph{abstract grafing} of a hyperbolic 
surface $X$ along a simple geodesic multi-curve $\gamma^*=\gamma_1\cup \cdots \cup \gamma_k$ with 
\emph{grafting heights} $L_1,\dots,L_k$ to be the surface obtained from $X$ by cutting $X$ open along
the geodesics $\gamma_i$ and inserting a flat cylinder of height $L_i$. Thus the abstract grafted surface
can be thought of as a geometric structure on $X$ which is piecewise flat or of constant curvature $-1$.
It will be convenient to allow that the flat metric on the cylinders is defined by a non-eulidean norm on $\mathbb{R}^2$. 

The following is our main result. It shows that in a very precise sense, abstract grafted surfaces serve 
as geometric models for Hitchin representations obtained from Fuchsian representations by bending. 
In its formulation, 
the \emph{cylinder height} of a bending representation at a component $\gamma_i$ of the bending locus $\gamma$
is the Finsler distance between $\gamma_i$ and its image under the bending transformation.
The grafted surface mentioned below is defined in Definition~\ref{def:abstractgraf}, and illustrated in Figure~\ref{fig-hitchin-rep}.

   \begin{maintheorem}\label{main1}
     For every $\sigma>0$, 
     there exists $C_{\sigma}>0$ such that
     the following holds.
    
     Consider a closed hyperbolic surface $S$, a multicurve $\gamma^*\subset S$ whose components have length at most $\sigma$, 
     and a bending parameter $z$ such that all cylinder heights of the resulting bending representation are
     bounded 
     from below by 
     some number $L>0$. 
     
     Then there exists an abstract grafted surface $S_z$, with universal covering $\tilde S_z$, and  
     a $\pi_1(S)$-equivariant embedding $\tilde Q_z:\tilde S_z\to (\X,d^{\mf F})$ which is quasi-isometric 
     with multiplicative constant $(1+C_\sigma/(L+1))$ and additive constant $C_{\sigma}$.
     Moreover, the image $\tilde S_z^\iota=\tilde Q_z(\tilde S_z)$ is $C_\sigma$-Finsler-quasiconvex in the sense that for all $x,y\in\tilde Q_z(\tilde S_z)$, there is \emph{at least one} Finsler geodesic from $x$ to $y$ at distance at most $C_\sigma$ from $\tilde S_z^\iota$.
    \end{maintheorem}

The result motivates us to mostly use the term \emph{grafting representations} instead of talking about bending.
Theorems~\ref{thm:quasiisom} and \ref{thm:quasiisom lengths} are more precise versions of this result.

In Section 12 of~\cite{KL18}, Kapovich and Leeb proved that for any Hitchin representation, all orbit maps are Finsler quasi-convex.
The main point of our result is that constants are independent of the representation, a result which cannot
be obtained from Kapovich--Leeb's approach.

The results of Kapovich--Leeb rely on a version of a Morse lemma 
in higher rank symmetric spaces (Theorem 1.3 of~\cite{KLPMorse}, see also \cite{KL18}).  
The proof of Theorem~\ref{main1} is independent of the results in~\cite{KLPMorse, KL18}, but also embarks from a 
(different) Morse-type lemma for Finsler metrics. 

\begin{maintheorem}\label{morseintro}
For every $C>0$ there exists a number $C^\prime >0$ with the following property.
Let $c:[a,b]\to (\X,d^{\mf F})$ be a map such that 
\[d^{\mf F}(c(s),c(u))+d^{\mf F}(c(u),c(t))\leq d^{\mf F}(c(s),s(t))+C\]
for all $a\leq s\leq u\leq t\leq b$.
Then there exists a Finsler geodesic connecting 
$c(a)$ to $c(b)$ at Hausdorff distance at most $C^\prime$ to $c$.
\end{maintheorem}

Theorem \ref{morseintro} does not hold for the Riemannian metric. We refer to 
Section \ref{sec:Morse} for more information. 

A key ingredient in the
proof of Theorem \ref{main1} is 
the construction of paths, with the properties stated in Theorem \ref{morseintro} 
for Hitchin representations obtained by bending a Fuchsian representation.
This construction relies
on positivity in the sense of Fock--Goncharov \cite{FG06} and can be viewed as a geometric
interpretation of positivity. It gives an alternative approach towards
the understanding of Hitchin representations via uniformly Morse paths in the sense of 
\cite{KLPMorse}, with geometric control of a different nature.

Theorem \ref{main1} can be used to study the degeneration of a Hitchin grafting ray, letting the grafting parameter go to infinity.
For instance, one can study the limit in the compactification of the Hitchin component introduced by Parreau \cite{Pa12,Pa00}.
A representation is identified with the projectivization of its $\bar {\mathfrak a}^+$-valued length function (the Jordan projection), living in a compact infinite-dimensional projective space.
A limit point of this compactification
can be seen, via a rescaling procedure, as the length function of an action of $\pi_1(S)$ on an affine building, which is an asymptotic cone of $\X$.
In \cite{BHMM25}, we examine other bordifications of the Hitchin component such as Loftin's \cite{LoftinNeck}.


The limit in Parreau's compactification of certain representations obtained via bending has already been studied by Parreau in her thesis (Section V.5 of \cite{Pa00}).
More precisely, Parreau assumes $d=3$ and deforms a Fuchsian representation by pinching a curve and bending along it.
She obtains an explicit formula for the limit length function.
Our results can be used to reprove this, and also to get formulas in other cases, for instance when no pinching is performed, and with $d\geq 3$.
This is the content Theorem~\ref{main2} stated below, after we mention a few other works.

Note that Parreau also defined in her context the grafted surface studied here (see Figure V.4 of \cite{Pa00}), although she did not prove quasi-convexity.
In \cite{Pa22,LTW22,CReid2023}, limits in Parreau's compactification of other kinds of deformations are studied.
In the first two papers, the authors construct an explicit invariant Finsler-convex subset for the limiting action on an affine building.
(Finsler-convex means any two points of the subset can be connected by \emph{at least one} Finsler geodesic; Parreau uses the terminology ``weakly convex''.)
Our results also give such a natural invariant convex set.

For simplicity, suppose $\gamma^*\subset S$ is an oriented simple closed curve and $S_0$ is one of the component of $S-\gamma^*$.
Following Parreau, one can define positive and negative intersection numbers $\iota^\pm(\gamma,\gamma^*)$ of a closed oriented curve $\gamma\subset S$ with $\gamma^*$.

\begin{maintheorem}\label{main2}
Let $\rho_t$ be a Hitchin grafting ray with grafting locus $\gamma^*$ and grafting parameter $tz$, where $z\in \mf a$ with nonzero cylinder height.
Then the following properties hold true.
\begin{enumerate}
\item $(\rho_t)_t$ converge, as $t$ goes to infinity, to the point of Parreau's compactification given by the projectivization of the $\mf a^+$-valued length function
$$\gamma \mapsto \iota^+(\gamma,\gamma^*)\kappa(z) + \iota^-(\gamma,\gamma^*)\kappa(-z),$$
where $\kappa$ is the Cartan projection.
\item For the corresponding $\pi_1(S)$-action on an asymptotic cone $B$ of $\X$, there is an invariant Finsler-convex subset $T\subset B$ isometric to the Bass--Serre tree of the graph of groups decomposition defined by $\gamma^*$, with edges decorated with $\kappa(z)$ or $\kappa(-z)$.  
\item For a suitable choice of basepoint, the quotient manifolds
$\rho_t(\pi_1(S))\backslash \X$ converge in the pointed 
Gromov--Hausdorff topology to the Fuchsian manifold defined by
the bordered hyperbolic surface $S_0$.
\end{enumerate}
\end{maintheorem}

Compare the formula for the length function with Proposition~V.5.9 of \cite{Pa00}.

\vspace{10pt}
\noindent 
{\bf Organization of the article and outline of the proofs.}
All main results in this article build on the concept of \emph{admissible paths} in 
the Lie group $\PSL_d(\R)$. Such a path can be thought of as a (continuous) 
path which projects to a piecewise geodesic in the symmetric space $\X$. 
Each geodesic piece corresponds either to a geodesic arc in a totally geodesic
embedded $\mathbb{H}^2$, 
or to an arc in a maximal flat, and these two types can be read off from the path in 
$\PSL_d(\R)$. Furthermore, the construction is done in such a way that it 
encodes positivity properties of the path in the sense of \cite{FG06}. 
This positivity then guarantees quantitative non-backtracking.

Admissible paths in $\PSL_d(\R)$ and $\X$ and their analogs for abstract grafted surfaces
are introduced in Section \ref{sec:admissible-paths}. Their natural appearance for
grafting deformations of Fuchsian representations is via so-called 
\emph{characteristic surfaces}, introduced in Section \ref{sec-HG}. 

In Section \ref{sec:Morse}, we introduce a very general class of invariant Finsler metrics
on $\X$ defined by \emph{polyhedral norms} in a Cartan subalgebra. We study geometric properties 
of these norms on ${\mathbb {R}}^{d-1}$ in Section \ref{sec:MorseRd}. The 
remainder of Section \ref{sec:Morse} is devoted to the
proof of Theorem \ref{morseintro}.

In Section \ref{sec:fockgoncharov}, the relation of admissible paths in $\PSL_d(\R)$ 
to positivity in the sense of \cite{FG06} is established. These results are applied
in Section \ref{QI-proof} to prove Theorem \ref{main1} and Theorem \ref{main2}.

\vspace{10pt}
{\bf Acknowledgement:} This project started as a working seminar in Fall 2021, during the pandemic,
held in person at the Max Planck Institute for Mathematics in Bonn.
We thank the MPI for the hospitality and financial support, and we thank
Gianluca Faraco, Elia Fioravanti, Frieder J\"ackel, Yannick Krifka, Andrea Egidio Monti, Laura Monk and Yongquan Zhang for 
many enjoyable discussions and good company.
P.-L.B. is grateful to Dick Canary, Fanny Kassel and Ralph Spatzier for helpful discussions, and 
U.H. and P.-L.B. thank Andr\'es Sambarino for 
helpful discussions.
All authors thank Beatrice Pozzetti for pointing out an error in an earlier version of this article.

\section{Lie groups and symmetric spaces}\label{sec:lietheoryI}

            
    This section collects some basic facts on Lie groups and symmetric spaces and introduces conventions and notations
    used later on.
    
    Consider the simple Lie group $G=\PSL_d(\R)$ and a representation $\tau:\PSL_2(\R)\to G$, that is, a locally injective Lie group homomorphism. 
    Many (but not all) of our results work for other semisimple Lie groups, and their proofs are easier to write using the abstract language of semisimple Lie groups, which we recall below.


    Recall that the Lie algebra $\mathfrak{sl}(2,\mathbb{R})$ of $\PSL_2(\R)$  
    is the Lie algebra of trace free real $(2,2)$-matrices.
    Denote by $\mathfrak g=T_{\id}G$ the Lie algebra of $G$, by $\mathfrak a\subset\mathfrak g$ a Cartan subalgebra 
    (maximal abelian subalgebra when $G$ is split) that contains $d\tau\hsl$ and by $\mathfrak a^+\subset\mathfrak a$ an 
    open Weyl cone whose closure $\overline{\mathfrak {a}^+}$ contains $d\tau\hsl$ (the definition of Weyl cone is recalled later in this section). 
    
    We require the representation $\tau$ to be \emph{regular}, that is, $d\tau\hsl$ belongs to the interior 
    $\mathfrak a^+$ of the Weyl chamber, or equivalently, there is a unique Cartan subspace $\mathfrak a$ containing $d\tau\hsl$.

    \medskip\noindent 
  {\bf Maximal compact subgroup and symmetric space}
  
    Let $K\subset G$ be a maximal compact subgroup which contains $\tau(\PSO(2))$ and whose Lie algebra $\mathfrak k$ is orthogonal (for the Killing form) to $\mathfrak a$.
    
    The symmetric space of $G$ is $\X=G/K$ with basepoint $\basepoint=[\id]=K\in G/K$
. Denote by $\pi_\X\colon G\to\X$ the projection.
    The space $\X$ is endowed with a nonpositively curved $G$-invariant Riemannian metric whose induced norm is denoted by $\Vert\cdot\Vert$, and whose distance function is denoted by $d_\X$. 
    The metric is normalised so that $\Vert d\tau\hsl\Vert=2$.

    \medskip\noindent 
  {\bf Maximal Flats}
  
    The subspace $\exp(\mathfrak a)\cdot\basepoint$ is a totally geodesic embedded 
    Euclidean subspace of $\X$ of maximal dimension. This flat 
    will often be identified with the abelian subgroup $A=\exp(\mf a)$ which acts simply transitively on it.
    Each maximal Euclidean subspace of $\X$ can be represented as $g\cdot A$ for some $g\in G$. These maximal Euclidean subspaces are called \emph{maximal flats}.

    \medskip\noindent 
  {\bf Root systems}
  
    Let $\mc R\subset\mathfrak a^*$ be the set of restricted roots of $G$, that is,
    the set of non-zero linear one-forms $\alpha$ on $\mathfrak a$ such that 
    $$\mathfrak g_\alpha:=\{X\in \mathfrak g \mid  [a,X]=\alpha(a)X \ \forall a\in \mathfrak a\}\neq 0.$$
    Recall that $G$ being split means that
    \[\mathfrak g=\mathfrak a\oplus\bigoplus_{\alpha\in\mc R}\mathfrak g_\alpha.\]

    The kernels of the restricted roots are the \emph{walls} of $\mf a$, and the \emph{open Weyl cones} are the connected components of the complements of the walls.
    By our regularity assumption on~$\tau$ there is a unique open Weyl cone containing $d\tau\hsl$.
    
    Let $\mc R^+:=\{\alpha\in\mc R:\alpha(d\tau\hsl)>0\}$ be the set of positive roots, and $\mc R^-=-\mc R^+$ the set of negative roots.
    Let $\Delta\subset\mc R^+$ be the set of simple roots (the positive roots that are not sums of several other positive roots).

    \medskip\noindent 
  {\bf Minimal parabolic subgroups and flag variety}
  
    The normalizer $P:=N_G(\bigoplus_{\alpha\in\mc R^+}\mathfrak g_\alpha)\subset G$ 
    for the adjoint representation is a minimal parabolic subgroup.
    Its Lie algebra is $\mathfrak p:=\mathfrak a\oplus\bigoplus_{\alpha\in\mc R^+}\mathfrak g_\alpha$.
    The \emph{opposite} parabolic subgroup is $P^-:=N_G(\bigoplus_{\alpha\in\mc R^-}\mathfrak g_\alpha)$.
       
    The \emph{flag variety} $\mc F:=G/P$ is compact.
    In fact, $K$ acts transitively on it, with finite point stabilizer.

    A notable subgroup of $P$ is $U=\exp\left(\bigoplus_{\alpha\in\mc R^+}\mathfrak g_\alpha\right)$.
    The dynamics coming from the geometry of the symmetric space and other homogeneous spaces of $G$ has some contraction properties that are recorded in the following algebraic fact:
    for any $u=\exp(\sum_{\alpha\in\mc R^+}X_\alpha)$ in $U$, for any  sequence $(a_n)_n\subset \mf a^+$ that diverges from the walls (i.e.\ $\alpha(a_n)\to+\infty$ for any~$\alpha\in \mc R_+$), we have 
    \begin{equation}\label{eq:basic contraction}\exp({-a_n})\cdot u\cdot \exp({a_n})=\exp\left(\sum_{\alpha\in \mc R_+} e^{-\alpha(a_n)}X_\alpha\right)\underset{n\to\infty}{\longrightarrow}\id.\end{equation}
    If $(a_n)_n$ does not diverge from all the walls but only some of them, then there is still a more complicated weaker contraction property.

\medskip\noindent 
  {\bf Maps induced by $\tau$}
  
    Let $T\subset\PSL_2(\R)$ be the subgroup of upper triangular matrices.
    The \emph{ideal boundary} 
    $\partial_\infty\H^2=\PSL_2(\R)/T$ of the hyperbolic plane $\H^2=\PSL_2(\R)/{\rm PSO}(2)$ 
    is naturally homeomorphic to the circle $\R\cup\{\infty\}$ under the map 
    $t\mapsto \left[\begin{psmallmatrix}1&t\\0&1\end{psmallmatrix}\cdot \begin{psmallmatrix}1&0\\1&1\end{psmallmatrix}\right]$ and $\infty\mapsto [\id]=T\in\PSL_2(\R)/T$.
    
    One can check that $\tau(T)\subset P$ and that $\tau$ induces an embedding 
    \[\partial\tau:\partial_\infty\H^2\hookrightarrow \mc F=G/P.\]

  \medskip\noindent 
  {\bf Transversality in the flag variety}
      
      Two flags $\xi,\eta\in\mc F$ are said to be \emph{transverse} if there exists $g\in G$ such that 
      $g\xi=\partial\tau(0)$ and $g\eta=\partial\tau(\infty)$; in this case we write $\xi\pitchfork \eta$.
      The set of transverse pairs of flags is an open dense subset of $\mc F^2$. 
      
      The set of flags transverse to $\partial\tau(0)$ is 
      \begin{equation}\label{eq:param transv}\partial\tau(0)^\pitchfork:=P^-\cdot \partial\tau(\infty) = \exp\left(\bigoplus_{\alpha\in\mc R^-}\mathfrak g_\alpha\right)\cdot\partial\tau(\infty).\end{equation}
      Similarly, for any flag $\xi$ denote by $\xi^\pitchfork$ the set of flags transverse to $\xi$.
      It is an open dense subset of $\mc F$. Note that by our convention, $P^-$ equals the stabilizer of 
      $\partial \tau(0)$ in $\mc F$. 

  Any two transverse flags are contained in the boundary of a unique maximal flat. The 
 maximal flat asymptotic to the transverse flags $\partial\tau(0)$ and $\partial\tau(\infty)$ equals  
      $A\cdot\basepoint\subset\X$.

      More generally, the flat between (that is, asymptotic to) 
      any two transverse flags $(\xi,\eta)=g(\partial\tau(0),\partial\tau(\infty))$ is
      $$F(\xi,\eta):=gA\cdot\basepoint= gA\subset\X.$$

    \medskip\noindent 
  {\bf Jordan and Cartan projections, and loxodromic elements}

  For any $g\in G$, the \emph{Cartan projection} is the unique element $\kappa(g)\in \overline{\mf a^+}$ such that 
  $g\in K\exp(\kappa(g))K$. Putting $A^+=\exp(\overline{\mathfrak{a}^+})$, 
  it is characterized by the fact that $\exp(\kappa(g))\basepoint$ is the unique intersection point of 
  $A^+\basepoint$ with the $K$-orbit of $g\basepoint$.
  Note that $d(\basepoint,g\basepoint)=||\kappa(g)||$ (here $d$ is the distance of the symmetric metric on $\X$).

Similarly, the $G$-orbit of any vector $v\in TX$ intersects $\overline{\mf a^+}\subset \mathfrak{p}$ in 
precisely one point $\kappa(v)$ which is called the \emph{Cartan projection} of $v$.

  For the \emph{Jordan projection} $\lambda(g)\in \overline{\mf a^+}$ we choose the following unnatural but convenient definition (see Remark 5.31 of~\cite{Benoist2016})
  $$\lambda(g):=\lim_{n\to\infty}\frac1n \kappa(g^n).$$

  The element $g\in G$ is called \emph{loxodromic} if $\lambda(g)$ 
  is contained in the interior $\mf a^+$ of 
  $\overline{\mf a^+}$, which is equivalent to saying that $g$ has an attracting/repelling fixed pair of transverse flags $(g^-,g^+)$. 
  Then $g$ acts as a translation on the flat $F(g^-,g^+)$.

  That the representation $\tau$ is regular means that the image $\tau(g)$ of any loxodromic $g\in \PSL_2(\R)$ is loxodromic in $G$.
      
    \medskip\noindent 
  {\bf Weyl Chambers and special directions}

      Since the symmetric space $\X$ is nonpositively curved, it admits a \emph{visual boundary} $\partial_\infty\X$, which is 
      naturally identified with the set of unit speed infinite geodesic rays starting at the basepoint $\basepoint$.
      
      By the normalization of the metric on $\X$, the representation $\tau$ induces an isometric embedding 
      $\H^2 \hookrightarrow \X$. 
      The isometric embeddings 
      $A\hookrightarrow \X,\, \H^2\hookrightarrow \X $ extend to embeddings of the visual boundaries 
      $\partial_\infty A \hookrightarrow\partial_\infty \X,\, \partial_\infty\H^2 \hookrightarrow\partial_\infty\X$.
      
      For any $g\in G$, we identify  $\xi=g\partial\tau(\infty)\in\mc F$ with a compact subset of the visual boundary $\partial_\infty\X$, 
      called a (closed) Weyl Chamber:
      $$\xi=g\partial_\infty A^+=g\cdot\{\lim_{t\to\infty}
      \exp({tv})\basepoint:v\in \overline{{\mathfrak a}^+}\}\subset g\partial_\infty A\subset\partial_\infty\X.$$
      It is the boundary at infinity of the \emph{Weyl Cone} $gA^+$ based at $g\basepoint$.
      The facets of $\xi=g\partial\tau(\infty)\in\mc F$ are the subsets of the form
      $$g \cdot \partial_\infty \left( A^+\cap \bigcap_{\alpha\in S}\ker\alpha \right)\subset\xi,$$
      where $S$ is a subset of $\Delta$; they are boundaries at infinity of facets of the Weyl Cone $gA^+$.

      Every $G$-orbit in $\partial_\infty\X$ intersects exactly \emph{once} every Weyl Chamber.
      In particular, to every Weyl Chamber $\xi\in\mc F$ and every point $p$ in the standard Weyl Chamber $\partial\tau(\infty)$ one can associate a point of $\xi$, which is the intersection point of $\xi$ with $G\cdot p$.
      The embedding $\partial_\infty \H^2\to \partial_\infty \X$ determines a special point of $\partial\tau(\infty)$.
      Its orbit under $G$ determines a special point in every Weyl Chamber.

\medskip\noindent 
  {\bf The Weyl group}
    
Let us recall the definition of the Weyl group, denoted by $\weyl$. 
It is the intersection of the maximal compact subgroup $K$ of $G$ with the stabilizer of 
$\mf a$, in restriction to $\mf a$ (so quotiented out by the fixator of $\mf a$ in $G$).
It can also be described as the group of orthogonal transformations of $\mf a$ generated by the reflections along the walls $(\ker\alpha)_{\alpha\in \Delta^+}$ 
of the closed Weyl Chamber $\overline{\mf a^+}$.
The Weyl group is finite, and any $\weyl$-orbit in $\mf a$ intersects $\mf a^+$ exactly once. 


\medskip\noindent 
  {\bf A Finsler metric coming from a linear functional on $\mathfrak a$}

\begin{notation}\label{nota:alpha}
   We fix a linear functional $\alpha_0$ on $\mf a$ which is positive on 
   $\overline{\mf a^+}$ and such that $\alpha_0(gv)< \alpha_0(v)$ for all 
   $v\in {\mf a^+}$ and $g\in\weyl$.

   We assume that $\alpha_0$ is symmetric in the sense that if $g$ is the transformation in 
   the Weyl group that maps $\mf a^+$ to its opposite $-\mf a^+$ then $\alpha_0(gv)=-\alpha_0(v)$ for any $v\in \mf a$.

   Let us denote by $\alpha_0^\myhash $ the vector in $\mf a$ such that $\alpha_0(v)=\langle v,\alpha_0^\myhash \rangle$ 
   for any $v\in\mf a$, where $\langle \cdot,\cdot\rangle $ is the inner product on $\mf a$ defined by the Riemannian metric on $\X$.
   Then the assumption above on $\alpha_0$ is equivalent to asking that $\alpha_0^\myhash \in\mf a^+$.
   We also denote by $\alpha_0^\myhash \in\partial_\infty\mf a^+$ the point at infinity to which the 
   ray spanned by $\alpha_0^\myhash $ limits.
\end{notation}

An example of a linear functional satisfying the above conditions is given in Equation~\ref{eq:ex of alpha0}.

Let $\mathfrak{k}\subset \mathfrak g$ be the Lie algebra of the maximal compact subgroup $K$ of $G$.
The orthogonal complement $\mathfrak{p}$ of $\mathfrak{k}$ 
with respect to the Killing form is naturally isomorphic to $T_\basepoint\X$, and it contains $\mathfrak{a}$. 
For any vector $v\in T\X$ we set
\begin{equation}\label{eq:mfF}\mf F(v)=\alpha_0(\kappa(v))\end{equation}
where as before, $\kappa(v)$ is the Cartan projection of $v$.

\begin{proposition}
[{Lemmas 5.9-10 of~\cite{KL18}}]\thlabel{ex:admissible finsler}
The following hold.
\begin{enumerate}
\item 
$\mf F$ defines a $G$-invariant Finsler metric on $\X$.
\item The unparameterized Riemannian geodesics of $\X$ are also geodesics for $\mf F$.
\item 
The translation length for $\mf F$ of any element $g\in G$ acting on $\X$ is given by $\ell^{\mf F}(g):=\alpha_0(\lambda(g))$ where $\lambda(g)\in\mf a^+$ is the Jordan projection.
\end{enumerate}
\end{proposition}

In the sequel we always normalize the functional $\alpha_0$ in such a way that 
the embedding $\H^2\to \X$ which is isometric for the symmetric metric also is
isometric for the Finsler metric $\mf F$.

\medskip\noindent 
  {\bf Busemann functions}

The Busemann functions, or horofunctions, are generalizations of distance functions on $\X$: they record relative distances to a point at infinity.
Geometrically, their level sets, called horospheres, are limits of spheres whose centers go to infinity.
Since there are several kinds of metrics on $\X$, there are also several kinds of horofunctions.

Because of the contraction property of the subgroup $U\subset G$ explained in \eqref{eq:basic contraction}, horospheres centered at a point $p$ in the interior of the Weyl chamber $\partial \tau(\infty)$ will always be $U$-invariant, since for any $u\in U$, for any $(x_n)_n$ converging to $p$ we have $d(x_n,ux_n)\to0$ for any $G$-invariant metric $d$.
In fact, every distance sphere of $\X$ (for any $G$-invariant metric) is foliated by orbits of the stabilizer of the center, and $U$-orbits can be described as limits of these leaves when the center tends to a point in the interior of $\partial\tau(\infty)$.

The leaves foliating the spheres centered at a given point can be parameterized by vectors of $\mf a^+$, which one may think of as vector-valued distances.
Namely, the $G$-orbit of any pair $(x,y)\in \X^2$ intersects $\{\basepoint\}\times A^+$ exactly once, and the intersection is denoted $(\basepoint,\exp(\kappa(x,y)))$,
where $\kappa(x,y)\in\overline{\mf a^+}$ is thought of as a vector-valued distance from $x$ to~$y$.
The orbits of the stabilizer of $x$ are the level sets of $\kappa(x,\cdot)$.
The Riemannian distance from $x$ to~$y$ can be expressed as $||\kappa(x,y)||$, and
the Finsler distance as defined in (\ref{eq:mfF}) can be expressed as
\begin{equation}\label{finslerdist} 
d^{\mf F}(x,y)=\alpha_0(\kappa(x,y)).
\end{equation}

Using the $U$-orbits one can define a vector-valued Busemann function centered at $\partial\tau(\infty)$: for any $x\in\X$, the $U$-orbit $U\cdot x$ intersects the standard flat $A\subset\X$ in exactly one point, and taking the logarithm we get a vector
$b^{\mf a}_{\partial\tau(\infty)}(\basepoint,x)\in\mf a,$
that records the relative distance from $x$ to $\partial\tau(\infty)$ compared to the basepoint $\basepoint$.
One can check using the contraction property of $U$ that if $(y_n)_n\subset A^+$ converge to a point in the interior of the Weyl Chamber $\partial\tau(\infty)$ then the $\Stab(y_n)$-orbits converge to level sets of $b_{\partial\tau(\infty)}^{\mf a}(\basepoint,\cdot)$.

Using the action of $K$, one can extend these vector-valued Busemann functions to Weyl chambers other than $\partial\tau(\infty)$.
For any Weyl Chamber $\xi=k\partial\tau(\infty)\in\mc F=G/P$, the \emph{vector-valued Busemann function} or horofunction
centered at $\xi$ between $\basepoint$ and $x$ is $b^{\mf a}_\xi(\basepoint,x)=b^{\mf a}_{\partial\tau(\infty)}(\basepoint,k^{-1}x)$, and more generally for $x,y\in \X$ we have: 
$$b^{\mf a}_\xi(x,y) = -b_\xi^{\mf a}(\basepoint,x)+b^{\mf a}_\xi(\basepoint,y) = \lim_{n\to\infty}\kappa(z_n,x)-\kappa(z_n,y)\in \mf a,$$
where $(z_n)_n\subset\X$ is any sequence converging to a point of the visual boundary in the interior of $\xi$.
One way to check the above formula is to find $k_n\in K$ such that $k_nz_n\in A^+$, use that $\kappa(z_n,x)=\kappa(k_nz_n,k_nx)$, and 
pass to a subsequence to make $k_n$ converge to some $k\in K$.

A horosphere in $\X$ based at a point in the visual boundary $\partial_\infty A$ of $A=A\cdot \basepoint$ equals the
$U$-obit of a horosphere in $A$.
Given a sequence $(z_n)_n$ in $A$ going to a point $p$ in the visual boundary of $A$, 
in the interior of the Weyl chamber $\partial_\infty A^+$, the sequence of Riemannian spheres 
in $A$ which are centered at $z_n$ and pass through the origin $0$ 
converges to the hyperplane in $A$ containing $0$ and perpendicular to the ray from $0$ to $p$.
One can check that on the other hand, the sequence of Finsler spheres, that is, the  
level sets of $d^{\mf F}$,  converges to the kernel of $\alpha_0$, which does not depend on $p$ and which, 
in the notation~\ref{nota:alpha}, is the hyperplane perpendicular to $\alpha_0^\myhash \in\partial_\infty A^+$.

The $\mathbb{R}$-valued 
Busemann function $b_p(x,y)$ associated to the Riemannian metric on~$\X$ and centered at 
a point $p$ in the interior of some Weyl chamber $\xi\subset\partial_\infty\X$  is the limit $\lim_{n\to\infty}d(x,z_n)-d(y,z_n)$ where $z_n\to p$. If $q$ is the intersection point of $\partial\tau(\infty)$ 
with the $G$-orbit of $p$ then we have 
$b_p(x,y)=\langle b^{\mf a}_\xi(x,y),v\rangle$ where $v\in\mf a$ is the unit vector pointing at $q$ and 
$\langle \cdot,\cdot\rangle$ is the inner product on $\mf a$.

The \emph{Busemann function} associated to our choice of Finsler metric is given by 
\begin{equation}\label{busemann}
b^{\mf F}_\xi(x,y)=\alpha_0(b^{\mf a}_\xi(x,y))=\lim_{n\to\infty}d^{\mf F}(x,z_n)-d^{\mf F}(y,z_n)\in\R.
\end{equation}

Here the last equality in the identity (\ref{busemann}) is valid since ${\mf F}$ is defined
by a linear functional on $\overline{\mathfrak{a}^+}$. 
Moreover, 
for any loxodromic element $g\in G$ with attracting fixed point $\xi\in \mc F$, the translation length of $g$ acting on $\X$ endowed with the Finsler metric $d^{\mf F}$ equals the quantity 
$ \vert b_\xi^{\mf F}(x, gx)\vert$. 

Note that using Notation~\ref{nota:alpha} we have the following link between the Finsler and Riemannian Busemann functions: 
if $p\in \partial_\infty \X$ is the intersection point of 
$\xi\in {\mc F}$ with the $G$-orbit of the point in $\xi\subset \partial_\infty \X$ corresponding to 
$\alpha_0^\myhash \in \mf a^+$ then
\begin{equation}\label{eq:finspheres are riemspheres}b_\xi^{\mf F}(x,y)=
\langle b_\xi^{\mf a}(x,y),\alpha_0^\myhash \rangle =b_p(x,y),\end{equation}
in other words, Finsler horospheres are Riemannian horospheres.

When $z_n$ converges to a point of the visual boundary which is not regular, 
that is, not in the interior of a Weyl Chamber, then the Riemannian Busemann functions are still well defined, the limit  $\lim_{n\to\infty}d(x,z_n)-d(y,z_n)$ still exists. For the Finsler metric 
the situation is more complicated: up to passing to a subsequence of $(z_n)_n$,
the limit $\lim_{n\to\infty}d^{\mf F}(x,z_n)-d^{\mf F}(y,z_n)$ is still well defined for all $x,y\in\X$ (we say that $z_n$ converge in the horoboundary), but it will give a more complicated, less algebraic, function of $x$ and $y$.
This was described by Kapovich--Leeb in Lemma 5.18 of~\cite{KL18}, and will be used in Section~\ref{sec:proj on diamond}.

Before we state the result let us analyze geometrically the limits of spheres in the flat $A$.
A Finsler ball is a convex polyhedron whose faces are contained in hyperplanes 
parallel to $\ker(\alpha_0)$ and their images under the action of the Weyl group by reflections.
If the centers of a sequence of such convex polyhedra tend to infinity 
away from the walls of the Weyl chambers, then the convex polyhedra converge to a halfspace 
bounded by the image of $\ker(\alpha_0)$ under an element of the Weyl group.
If the centers of such a sequence 
tend to infinity away from all walls \emph{but one}, 
then the convex polyhedra converge to the intersection of two such halfspaces.
For example, they could be $\{x:b_\xi^{\mf F}(x,0)\leq 0\}$ and $\{x:b_\eta^{\mf F}(x,0)\leq 0\}$ 
determined by two Weyl chambers $\xi$ and $\eta$ that share a codimension~1 face.
So the associated Busemann function associated to this intersection should be $f(x)=\max(b_\xi^{\mf F}(x,0),b_\eta^{\mf F}(x,0))$.



The precise statement is as follows.
Let $(z_n)_n\subset \X$ be a sequence converging to a point $p\in\partial_\infty\X$.
The point $p$ is contained in possibly infinitely many closed Weyl chambers, let $\mc B'\subset\mc F$ be the set of such Weyl Chambers. 
Let $\xi\in {\mc F}$ and let $C$ be the Weyl cone based at $\basepoint$ and asymptotic to $\xi$.
If $\xi\not\in {\mc B}$ then  $d(z_n,C)\to\infty$.
It could happen that $d(z_n,C)\to\infty$ even if $\xi\in \{\xi_1,\dots,\xi_\ell\}$.
Let $\mc B'\subset \mc B$ be the set of $\xi$ such that $d(z_n,C)$ remains bounded.
In this case let $p_\xi$ be the intersection point of $\xi$ with $G\cdot \alpha_0^\myhash $ where we view $\alpha_0^\myhash $ as a point in $\partial_\infty A^+$.
Up to passing to a subsequence, there exists $x_0\in \X$ such that for any $x\in\X$ we have
\begin{equation}\label{eq:nonreg busemann}
    d^{\mf F}(x,z_n)-d^{\mf F}(x_0,z_n) \underset{n\to\infty}{\rightarrow} 
    \max_{\xi\in {\mc B}'}b^{\mf F}_\xi(x,x_0)=\max_{\xi\in {\mc B}'}b_{p_\xi}(x,x_0).
\end{equation}
In other words, the Finsler balls centered at $z_n$ whose boundary contain $x_0$ converge to the intersections of the Riemannian horoballs centered at $p_\xi$ for $\xi\in \mc B'$ whose boundary contains $x_0$.

%
%
%

\medskip\noindent 
{\bf Concrete description of the above objects for $G=\PSL_d(\mathbb R)$}

    The Lie algebra $\mathfrak{g}$ is the algebra of trace free $(d,d)$-matrices. As Cartan subspace 
    $\mathfrak{a}$ we choose the linear subspace of 
    diagonal $(d,d)$-matrices with vanishing trace, and the open Weyl chamber $\mathfrak{a}^+$ is the open cone
    of diagonal matrices whose entries $(\lambda_1,\dots,\lambda_d)$ fulfill 
    $\lambda_1> \lambda_2> \cdots > \lambda_d$.

  The subgroup $K\subset\PSL_d(\R)$ is chosen as the group $\PSO_d(\R)$, and $P\subset\PSL_d(\R)$ is taken 
  as the image in $\PSL_d(\R)$ of the set of upper triangular matrices with positive entries on the diagonal and determinant one.

The flag variety $\mc F$ has the following explicit description. 
Namely, a \emph{full flag} in $\mathbb{R}^d$ is a sequence
\[\xi=(\xi_1\subset \xi_2\subset \cdots \subset \xi_{d}=\mathbb{R}^d)\]
where $\xi_i$ is a linear subspace of $\mathbb{R}^d$ of dimension $i$ for each $i\leq d$.
Clearly $\PSL_d(\mathbb{R})$ acts transitively on the space of all full flags, with point stabilizer
a minimal parabolic subgroup. Thus $\mc F$ is just the space of full flags in $\mathbb{R}^d$.

The Busemann functions also have a concrete description, 
using the identification between $\X$ and the set of inner products on $\R^d$ that induce the standard volume form. Namely,
given $x,y\in\X$, let $||\cdot||_x$ and $||\cdot||_y$ denote the norms of the associated 
inner products on $\R^d$ and on the exterior products $\bwedge^k\R^d$ $(1\leq k\leq d)$.
Let $\xi=(\xi_1\subset\xi_2\subset\dots\subset\xi_d)$ be a full flag in $\R^d$.
Let $v=(v_1,\dots,v_d)=b^{\mf a}_\xi(x,y)\in\mf a$ with $v_1+\dots+v_d=0$.
Then for all $k\leq d$, we have
\begin{equation*}
    v_1+\dots+v_k=\log\frac{||X||_x}{||X||_y}
\end{equation*}
where $X\in\bwedge^k\R^d$ is any representative of the $k$-plane $\xi_k$.

The homomorphism $\tau$ is obtained as follows. 
For $d\geq 3$ there exists up to conjugation a unique irreducible representation 
of $\SL_2(\R)$ on $\R^d$. This representation determines the following embedding 
    $\tau\colon\PSL_2(\R)\to\PSL_d(\R)$. 
    Let $\R^h_{d-1}[X,Y]$ be the set of degree $d-1$ homogeneous polynomials with real coefficients. 
    A matrix $M=\smallpmatrix{a}{b}{c}{e}\in\SL_2(\R)$ acts on the vector space 
    $\R[X,Y]$ of polynomials in two variables 
    by $M\cdot X=aX+cY$ and $M\cdot Y=bX+eY$. This action $\SL_2(\R)\acts\R[X,Y]$ preserves the $d$-dimensional linear
    subspace $\R^h_{d-1}[X,Y]$, with determinant one elements. 
    So it induces an embedding $\PSL_2(\R)\to\PSL_d(\R)$ which is just the representation $\tau$. 
    
    Using a suitable basis we have $\tau(\SO(2))\subset K$, the representation $\tau$ induces an isometric embedding 
    $\H^2=\PSL_2(\R)/\PSO(2)
    \hookrightarrow \X$. Denote by $\bH^2=\pi_\X\circ \tau(\PSL_2(\R))$ the image of~$\H^2$ inside $\X$. 
    
    The following statement is a consequence of the fact
    that $d\tau (T^1\mathbb{H}^2)$ consists of regular vectors contained in a single $G$-orbit, and each such vector
    is tangent to a unique maximal flat. It is well known and immediate from the above discussion.
    
    \begin{fact}\thlabel{fact:unique-flat}
   \begin{enumerate}
       \item 
        Every geodesic in $\bH^2$ lies in a unique maximal flat.
     \item    
        For any hyperbolic element $g\in \PSL_2(\R)$, the centraliser of $\tau(g)$ in $G$ acts by translations on the unique flat containing the image under $\pi_\X$ of the axis of $g$ in~$\H^2$.
   \end{enumerate}
    \end{fact}

\section{Hitchin grafting representations}\label{section-HG}

The \emph{Hitchin component} ${\rm Hit}(S)$ 
for conjugacy classes of representations 
$\pi_1(S)\to \PSL_d(\mathbb{R})$ is the connected component of the set of conjugacy classes of representations which 
factor through an irreducible representation $\PSL_2(\mathbb{R})\to \PSL_d(\mathbb{R})$.
In the sequel we always work with explicit representations rather than with conjugacy classes. 

The Hitchin representations we are interested in are the familiar \emph{bending} or \emph{bulging} deformations
of \emph{Fuchsian} representations, that is, representations  
which factor through the embedding $\tau:\PSL_2(\mathbb{R})\to \PSL_d (\mathbb{R})$. We refer to~\cite{Go86,AZ23} for an account
on the bending construction. 
The goal of this section is to introduce these representations as well as an abstract geometric model for them, and 
we establish some first geometric properties of the representations and the model.
The precise relation between the geometry of bending representations and the geometry of the model will be established  in 
Section~\ref{QI-proof} and constitutes the main result of this article.

The material in Subsections~\ref{sec:Abstract grafting} --~\ref{sec:graphofgroups} is well known, and
the purpose is to summarize the properties and the viewpoint we are going to pursue.

 \subsection{Abstract grafting}\label{sec:Abstract grafting}

    In this subsection we introduce 
    abstract grafting of a hyperbolic surface as initiated by Thurston. We refer to~\cite{T97} for an early account on this 
    construction.
    Contrary to the common definition in the literature, our grafting contains a twist which is needed for our purpose.
 
    Consider a closed oriented surface $S$ of genus $g\geq 2$ endowed with a hyperbolic metric.
	A \emph{simple (geodesic) multi-curve} $\gamma^*$ is the union of pairwise disjoint essential 
 mutually not freely homotopic simple closed curves (geodesics) on $S$. 
 We fix moreover an orientation on each component of $\gamma^*$.

 Consider the special direction $u=d\tau\hsl\in \mf a$ given by $\tau$.
 For any $z\in \mf a$ and $\ell>0$, let $\Cyl(\ell,z)\subset \mf a/\ell u$ be the cylinder obtained by quotienting  the strip $\{tu+sz: t\in\R,\ s\in[0,1]\}\subset\mf a$ under the translation by $\ell u$.
 The (Finsler) \emph{height} of such cylinder is defined as 
 \begin{equation}\label{eq:cylinder height}
     \mathrm{height} = \min\{\mf F(tu+z):t\in\R\}.
 \end{equation}
 
 We fix for every $\gamma\in \gamma^*$ a vector $z_\gamma\in \mf a$; the collection $z=(z_\gamma)_{\gamma\in\gamma^*}$ is interpreted as a grafting parameter.
 
	\begin{definition}\label{def:abstractgraf}
         The \emph{abstract grafting} of $S$ along the geodesic multi-curve 
         $\gamma^*$ is the surface $S_z$ obtained by cutting $S$ open along
         each of the components $\gamma$ of $\gamma^*$, inserting 
         flat cylinders $C_\gamma=\Cyl(\ell_S(\gamma),z_\gamma)$  and gluing the surface back with the translation by $z_\gamma$.

         If $z_\gamma$ is not parallel to $u$ for any $\gamma\in\gamma^*$, then this grafting comes with a natural homotopy equivalence $\pi_z:S_z\to S$ projecting the flat cylinders onto $\gamma^*$, which allow us to identify $\pi_1(S_z)$ and $\pi_1(S)$.
	\end{definition}

	More precisely, for each $\gamma\in \gamma^*$, the metric completion of $S-\gamma$ is a surface whose boundary consists of 
 two geodesics $\gamma_1,\gamma_2$ of the same length $\ell_S(\gamma)$. 
 The choice of a parameterisation $\gamma(t)$ defines parameterisations $\gamma_1(t)$, $\gamma_2(t)$. 
 Attach the flat cylinder $C_\gamma$ to $\gamma_1$ and $\gamma_2$ by identifying $[tu]\in C_\gamma$ with $\gamma_1(t)$ and $[tu+z_\gamma]$ with $\gamma_2(t)$.

    Let $C=\bigcup_{\gamma\in\gamma^*}C_\gamma\subset S_z$ and $S'$ be the metric completion of $S-\gamma^*$, so that 
    $S_z=(S'\sqcup C)/\sim$ where $\sim$ is the gluing procedure explained above.
    The projection map $\pi_z:S_z\to S$ satisfies the following.
    Its restriction $S'\to S$ is the continuous extension of the inclusion $S-\gamma^*\hookrightarrow S$.
    It projects each $[tu+sz_\gamma]\in C_\gamma$ to $\gamma(t)\in S$.

	We call this operation \emph{abstract grafting} to distinguish it from the \emph{Hitchin grafting} 
    that we will introduce for Hitchin representations. We shall refer to $S_z$ as a \emph{grafted surface}. 
	
    Note  that if $z_\gamma = 0$ for every $\gamma$ in $\gamma^*$, then the grafting is trivial and $S_z = S$. 
    If all $z_\gamma$ are parallel to $u$, then the grafted surface is hyperbolic and obtained from $S$ by shearing along $\gamma^*$
    with shearing length given by the size of the parameters $z_\gamma$.
    
    More generally, $S_z$ is an orientable surface which admits a canonical 
    piecewise smooth structure as well as a natural conformal structure which in turn 
    induces a global $C^1$-structure on $S$. 
    Any norm on $\mf a$ which coincides with the norm induced by the hyperbolic metric on the distinguished direction
    $u$ 
    induces a Finsler metric on $S_z$ which coincides with the hyperbolic metric on $S^\prime$ and whose restriction
    to the cylinders $C_\gamma$ is flat. 
    
    In particular, the norm defined by 
    the Riemannian metric of $\X$ can be used to endow the $C^1$-surface
    $S_z$ with a $C^0$ Riemannian metric which is smooth everywhere except at the gluing locus, 
    has constant curvature $-1$ in $S_z-\cup_\gamma C_\gamma$ and has constant curvature $0$ in the interior of the cylinders $C_\gamma$. 
    Since the curvature of this metric 
    is non-positive whereever it is defined 
    and the gluing is performed along geodesics, $S_z$ is non-positively curved in the sense of 
    Alexandrov and hence its universal covering $\tilde S_z$ is a ${\rm CAT}(0)$-space.
    
    Thus in this case every free homotopy class has a Riemannian 
    geodesic representative whose length is minimal in the free homotopy class. Such a Riemannian geodesic is unique unless it is a core curve
    of a flat cylinder. 
    If all the $z_\gamma$ are orthogonal to the special direction $u$, then the natural homotopy equivalence $\pi_z:S_z\to S$ is $1$-Lipschitz and hence in this case,
    free homotopy classes have bigger lengths in $S_z$ than in $S$. Moreover, 
    the unit tangent bundle $T^1S_z$ of $S_z$ is well defined, and there is a geodesic flow which is topologically mixing and
    admits a unique measure of maximal entropy~\cite{Kn98}. 


    As we are interested in Finsler metrics on 
    $\X$ 
    using $\alpha_0$ (see \eqref{eq:mfF}) rather than the Riemannian one, we also endow $S_z$ with a Finsler metric by equipping each cylinder $C_\gamma$ with the quotient of the non-Euclidean norm $\mf F$ on $\mf a$.
    Observe that in general, 
    for a given $C^1$-structure on $S_z$ as constructed above, 
    this metric is \emph{discontinuous} at the gluing locus between the flat cylinders and the hyperbolic part.
    Additionally 
    the metric on the flat part is sensitive in the direction of $z$, and does not depend only on the height of the grafting (contrarily to the Riemannian metric).
    Nevertheless it induces a well defined path metric on $S_z$.

    The following observation will be useful later on when estimating lengths.

    \begin{lemma}\label{lem:Lipschitz}
        If all $z_e$ are in $\ker(\alpha_0)$,
        then the natural projection $\pi_{z}:S_{z}\to S$ is $1$-Lipschitz for the Finsler metric on $S_{z}$.
        In particular, all free homotopy classes of curves have bigger Finsler lengths in $S_{z}$ than in $S$.
    \end{lemma}

    \begin{proof}
        By definition, the restriction of our projection map $\pi_{z}: S_{z}\to S$ to each flat cylinder $C_\gamma=\{tu+sz_\gamma\}/\ell_S(\gamma)u$ comes from the linear projection of $\mf a$ onto the line spanned by $u$, parallel to the direction $z_\gamma\in \ker(\alpha_0)$.
        To conclude it suffices to note that this projection is $1$-Lipschitz for the non-Euclidean norm on $\mf a$, which was defined using~$\alpha_0$~(see \eqref{eq:mfF}).
    \end{proof}

    \subsection{Particular case of an amalgamated product}

    In this section we explain briefly the construction of the following two Sections~\ref{sec:graphofgroups} and~\ref{sec-HG} in the special case where $\gamma^*$ has only one component and is separating.

    Let $\Sigma$ be a closed orientable smooth surface of genus at least $2$ 
    and let $\gamma^*\subset\Sigma$ be a separating simple closed curve.
    Then $\gamma^*$ splits $\Sigma$ into two subsurfaces $\Sigma_1$ and $\Sigma_2$, and~$\pi_1(\Sigma)$ can be written as an amalgamated product $\pi_1(\Sigma_1)\underset{\gamma^*}{*}\pi_2(\Sigma_2)$.

    Consider a discrete and faithful representation $\rho:\pi_1(S)\to\PSL_2(\R)\overset{\tau}{\to}G$ such that $\rho(\gamma^*)=\exp(\ell_\rho(\gamma^*)u)$ where $u=d\tau\hsl$ is the special direction of $\mf a$.
    Let $z\in\mf a$ be a grafting parameter.
    The \emph{Hitchin grafting representation} $\rho_z:\pi_1(S)\to G$ is defined 
    by requiring that $\rho_z(\gamma)=\rho(\gamma)$ for any 
    $\gamma\in\pi_1(\Sigma_1)$ and $\rho_z(\gamma)=\exp(z) \cdot \rho(\gamma) \cdot \exp(-z)$ for any $\gamma\in\pi_1(\Sigma_2)$.

    One can then define an 
    immersion $Q_z:S_z\to \rho_z\backslash \X$ whose restriction to any of the hyperbolic pieces of $S_z$ and to the flat
    cylinder is totally geodesic.
    Indeed, $\rho_z(\pi_1(\Sigma_1))=\rho(\pi_1(\Sigma_1))$ preserves $\wt\Sigma_1\subset\H^2\subset\X$, inducing a totally geodesic embedding $\Sigma_1=\rho_z(\pi_1(\Sigma_1))\backslash \wt\Sigma_1\hookrightarrow \rho_z\backslash \X$.
    Similarly, if one identifies $\Sigma_2$ with 
    $\rho(\pi_1(\Sigma_2))\backslash \wt\Sigma_2$ where $\wt\Sigma_2\subset\H^2\subset\X$, then 
    $\rho_z(\pi_1(\Sigma_2))=\exp(z) \cdot \rho(\pi_1(\Sigma_2))\cdot \exp(-z)$ preserves 
    $\exp(z)\wt\Sigma_2$ and hence induces a totally geodesic embedding $\Sigma_2\hookrightarrow \rho_z\backslash \X$.
    In general the image of the boundary components $\partial\Sigma_1$ and $\partial\Sigma_2$ in 
    $\rho_z\backslash \X$ are disjoint.
    However they can be connected by the natural totally geodesic embedding of the cylinder $C=\{tu+sz:t\in\R,\ s\in[0,1]\}/\rho_z(\gamma^*)$ into $\rho_z\backslash \X$.
    Gluing these three embeddings yield a piecewise totally geodesic embedding 
    of $S_z=(\Sigma_1\cup C\cup \Sigma_2)/\sim$ into $\rho_z\backslash \X$.

    \subsection{Graphs of groups decomposition and bending}\label{sec:graphofgroups}

    A classical reference for the theory of graph of groups is~\cite{BassSerre}. We collect some facts we need.
    Let $\Sigma$ be a closed orientable smooth surface and let $\gamma^*\subset\Sigma$ be a simple multi-curve.
    The multi-curve determines the following graph of groups decomposition of $\pi_1(\Sigma)$, 
    which will be used to define a family of Hitchin representations.
    
    Let $G_{\gamma^*}$ be the oriented graph such that each vertex $v\in V$ corresponds to a component $\Sigma_v$ of $\Sigma-\gamma^*$, and each edge $e\in E$ corresponds to an oriented component $\vec \gamma_e$ of $\gamma^*$. Given an edge $e\in E$, we denote by $\bar e$ the opposite edge of $e$, for which $\vec \gamma_{\bar e}$ corresponds to the curve $\vec\gamma_e$ with the reverse orientation.
    The oriented edge $e$ is adjacent to the two (not necessarily distinct) components $\Sigma_{o(e)},
    \Sigma_{t(e)}$
    of $\Sigma-\gamma^*$ which contain $\vec \gamma_e$ in their boundary.
    One can embed $G_{\gamma^*}$ into the surface $\Sigma$ such that each vertex $v$ lies in the interior of $\Sigma_v$ and each edge $e$ connects $o(e)$ to $t(e)$, crossing transversally $\vec\gamma_e$ once.
    Since we assume that $\Sigma$ is oriented, 
    choosing an orientation on $\gamma^*$ is the same as choosing for each pair of opposite 
    edges $e,\bar e\in E$ a preferred one by declaring that the ordered pair $(u_1,u_2)$ consisting of the 
    oriented tangent $u_1$ of the oriented edge $e$ at $x_e$ and the oriented 
    tangent of the oriented geodesic $\gamma^*$  defines the orientation of $\Sigma$. 
    
    The graph of groups decomposition of $\pi_1(S)$ defined by this datum associates to each 
    vertex $v\in V$ the fundamental group 
    $A_v:=\pi_1(\Sigma_v,v)$ where $v$ is seen as a point in the interior of 
    $\Sigma_v$. 
    To each edge $e$ is associated the fundamental group 
    $A_e:=\pi_1(\vec \gamma_e,x_e)\simeq\Z$ of~$\vec \gamma_e$, where $x_e$ is the intersection point of~$\vec\gamma_e$ with $e$ (which is seen as an arc in $\Sigma$ transverse to $\vec\gamma_e$).
    The inclusions $\vec\gamma_e\hookrightarrow \Sigma_{o(e)}$, $\vec\gamma_e\hookrightarrow \Sigma_{t(e)}$ determine the following monomorphisms, by connecting $x_e$ to respectively $o(e)$ and $t(e)$ via $e$.
    \[\alpha_{o(e)}:A_e=\pi_1(\vec \gamma_e,x_e)\hookrightarrow A_{o(e)}=\pi_1(\Sigma_{o(e)},{o(e)}) \text{ and } 
    \alpha_{t(e)}:A_e\hookrightarrow A_{t(e)}.\] 
    Note that $\alpha_{o(e)}(\vec\gamma_e)=\alpha_{t(\bar e)}(\vec\gamma_{\bar e})^{-1}$.

    That this construction indeed defines a decomposition of $\pi_1(\Sigma)$ as graph of groups with cyclic edge groups is well known. More precisely,  
    choose a spanning tree $T\subset G_{\gamma*}$ of~$G_{\gamma^*}$,
    with edge set $E_T\subset E$ invariant under the orientation reversing map $e\mapsto\bar e$.
    For a vertex $v\in V$ put $A_v$, and for an edge $e\in E$ put $A_e$. 
    Denote by $\vec \gamma_e$ the oriented geodesic defined by the oriented edge $e$.
    
    Let $\pi_1(G_{\gamma^*},T)$ be the quotient group 
    $$\pi_1(G_{\gamma^*},T)=\faktor{\left(*_v A_v\right)*F_E}{R}$$
    where 
    $*$ denote the free product, $F_E$ is the free group generated by the edge set $E$,  and $R$ is the normal subgroup of $\left(*_vA_v\right)*F_E$ generated by the union of the sets
    \begin{itemize}
        \item $e \cdot \bar e$ for all $e\in E$,
        \item $ e $ for all $e \in E_T$,
        \item $e \alpha_{o(e)} (g)e^{-1} \alpha_{t(e)}(g)^{-1}$ for all $e\in E$         
        and $g\in A_e$, which we think of as $e\alpha_{o(e)}(g)e^{-1}\equiv \alpha_{t(e)}(g)$.
    \end{itemize}
    Thus  $\pi_1(\Sigma)$ is obtained from simultaneous HNN-extension of the tree of groups defined by the spanning tree $T$.

    Recall that the isomorphism between $\pi_1(G_{\gamma^*},T)$ and $\pi_1(S)$ is constructed by choosing a basepoint $v_0\in V$, embedding each vertex group $A_v$ into $\pi_1(S,v_0)$ by connecting $v$ to $v_0$ via the spanning tree $T$, and mapping $F_E$ into  $\pi_1(S,v_0)$ by connecting the endpoints $o(e)$ and $t(e)$ of each $e$  to $v_0$ via the tree $T$.

    Take a discrete and faithful representation $\rho\colon\pi_1(G_{\gamma^*},T)\to \PSL_2(\R)\xrightarrow[]{\tau} G$ which factors through the embedding  
    $\tau:\PSL_2(\mathbb{R})\to \PSL_d(\mathbb{R})$. 
    We use the graphs of groups decomposition of $\pi_1(\Sigma)$ to perform a bending of the representation in $G$ with parameter $z=(z_\gamma)_{\gamma\in\gamma^*}\in \mf a^{\gamma^*}$.
    This construction can be thought of as
    bending the surface $S$ along the geodesic multicurve $\gamma^*$ in the space of representations into $G$.

    Let $\wt \rho:(*_vA_v)*F_E\to G$ be the composition of $\rho$ with the projection 
    \begin{equation}\label{alpha}
    (*_vA_v)*F_E\to \pi_1(G_{\gamma^*},T).\end{equation}
    Fix an orientation on $\gamma^*$, so that for every $\gamma\in\gamma^*$,  we get a preferred edge $e\in E$.

    Then there exists $B\in \PSL_d(\R)$ such that $\wt\rho(\alpha_{o(e)}(\vec\gamma_e))=B\exp(\ell_\rho(\gamma)u)B^{-1}$, where $u=d\tau\hsl$ is the special direction of $\mf a$ (note that the $\ell_\rho$-length of $\gamma$ does not depend on the orientation of $\gamma$ since $\alpha_0$ was taken symmetric).

    Set $\zeta_e=B\exp(z_\gamma)B^{-1}$, so that $\zeta_e$ commutes with $\wt\rho(\alpha_{o(e)}(\vec\gamma_e))$.
    Note that by definition of our relations $R$ and $\alpha_{o(\bar e)}(\vec\gamma_{\bar e})=\alpha_{t(e)}(\vec\gamma_{e})^{-1}$ we have 
    $$\wt\rho(\alpha_{o(\bar e)}(\vec\gamma_{\bar e}))=\wt\rho(e)\wt\rho(\alpha_{o(e)}(\vec\gamma_{e}))^{-1}\wt\rho(e)^{-1}=\wt\rho(e)B\exp(-\ell_\rho(\gamma)u)B^{-1}\wt\rho(e)^{-1}.$$
    Set $\zeta_{\bar e}=\wt\rho(e)B\exp(-z_\gamma)B^{-1}\wt\rho(e)^{-1}=\wt\rho(e)\zeta_e^{-1}\wt\rho(e)^{-1}$, that commutes with $\wt\rho(\alpha_{o(\bar e)}(\vec\gamma_{\bar e}))$ and satisfies $\wt\rho({\bar e})\zeta_{\bar e}\wt\rho({e})\zeta_{e}=1$.

    Geometrically, the group $A_{e}$ acts on $\H^2$ as a translation on a geodesic $\wt\gamma$. 
    By \thref{fact:unique-flat}, the image of $\wt\gamma\subset\H^2$ in $\X$ is contained in a unique 
    maximal flat,
    and $\zeta_{e}$ preserves this flat and acts on it as a translation. 
    
    A Hitchin grafting representation is obtained by performing a partial 
    conjugation of $\pi_1(G_{\gamma^*},T)$ by the elements $\zeta=(\zeta_e)_{e\in E}$.
    Fix a basepoint $v_0\in V$. 
    For any $v\in V$, we denote by 
    $$\omega_v=\zeta_{e_1}\cdots \zeta_{e_n}$$
    where $(e_1\cdots e_n)$ is an oriented path in the tree
    $T$ from $v_0$ to $v$. Since $\zeta_{\bar e}=\zeta_e^{-1}$ when $e$ is in $E_T$ and $T$ is a tree, 
    the value of $\omega_v$ does not depend on the chosen path. 
    
    Then define the representation $\wt\rho_z\colon\left(*_vA_v\right)*F_E\to G$ by
    \begin{enumerate}[i)]
        \item $\wt\rho_z(g)=\omega_v\wt\rho(g)\omega_v^{-1}$ for all $v\in V$ and $g\in A_v=\pi_1(\Sigma_v)$,
        \item $\wt\rho_z(e)=\omega_{o(e)}\wt\rho(e)\zeta_e\omega_{t(e)}^{-1}$ for all $e\in E$.
    \end{enumerate}
    
    \begin{lemma}
        The representation $\wt\rho_z$ contains $R$ in its kernel.
    \end{lemma}
    
    \begin{proof}
        For all $e\in E$, we have 
        $$\wt\rho_z(e\bar e) = \left(\omega_{o(e)}\wt\rho(e)\zeta_e\omega_{t(\bar e)}^{-1}\right)\left(\omega_{o(e)}\wt\rho(\bar e)\zeta_{\bar e}\omega_{t(\bar e)}^{-1}\right) = 1$$
        since $\omega_{t(\bar e)}=\omega_{o(e)}$ and $\wt\rho(\bar e)\zeta_{\bar e}\rho(e)\zeta_e=1$.

        For all $e\in E_T$, we have $\wt\rho(e)=1$ and $\omega_{t(e)}=\omega_{o(e)}\zeta_e$, so 
        $$\wt\rho_z(e)=\omega_{o(e)}\wt\rho(e)\zeta_e\omega_{t(e)}^{-1}=1$$
        
        Take $e\in E$ and $g\in A_e$. 
        Then
        \begin{align*}
            \wt\rho_z(e\alpha_e(g)e^{-1}) 
                &= \left(\omega_{o(e)}\wt\rho(e)\zeta_e\omega_{t(e)}^{-1}\right) \left(\omega_{t(e)}\wt\rho(\alpha_e(g))\omega_{t(e)}^{-1}\right) \left(\omega_{t(e)}\zeta_e^{-1}\wt\rho(e)^{-1}\omega_{o(e)}^{-1}\right) \\
                &= \omega_{o(e)}\wt\rho(e)\zeta_e\wt\rho(\alpha_e(g))\zeta_e^{-1}\rho(e)^{-1}\omega_{o(e)}^{-1} \\
                &= \omega_{o(e)}\wt\rho(e)\wt\rho(\alpha_e(g))\wt\rho(e)^{-1}\omega_{o(e)}^{-1} \qquad \text{ since $\zeta_e$ and $\wt\rho(\alpha_e(g))$ commute} \\
                &= \omega_{o(e)}\wt\rho(e\alpha_e(g)\bar e)\omega_{o(e)}^{-1}\\
                &= \omega_{t(\bar e)}\wt\rho(\alpha_{\bar e}(g))\omega_{t(\bar e)}^{-1} =\wt\rho_z(\alpha_{\bar e}(g))\qedhere
        \end{align*}
    \end{proof}

    \begin{definition}\label{def:hitchin grafting rep}
        We denote by $\mathrm{Gr}_z^{\gamma^*}\rho\colon\pi_1(G_{\gamma^*},T)\to G$ the representation induced by $\wt\rho_z$, 
        and sometimes just $\rho_z$ if there is only one hyperbolic structure involved. 
        We call it the \emph{Hitchin grafting representation} with data $z$ along $\gamma^*$.
    \end{definition}
    
    Up to conjugation, the representation $\rho_z$ does not depend the choices made for the graph of group decomposition.

    \subsection{The characteristic surface for Hitchin grafting representations}\label{sec-HG}

    Consider a Fuchsian representation $\rho:\pi_1(S)\to \PSL_2(\R)\to G$ and 
    denote by $S$ the hyperbolic surface defined by this representation.
    Choose some grafting datum $z$ and let 
    $\rho_z$ be the Hithin grafted representation defined by $\rho$ and $z$.
    As this representation is contained 
    in the Hitchin component, it follows from Labourie~\cite{Lab06} and Fock--Goncharov~\cite{FG06} that $\rho_z$ is 
    faithful, with discrete image. 
    In particular, the quotient manifold $\rho_z\backslash \X$ is homotopy equivalent to $S$; 
    in fact $\rho$ induces a natural homotopy class of homotopy equivalences between 
    $\rho_z\backslash \X$ and $S$.

    The goal of this subsection is to construct a geometrically controlled homotopy equivalence from an 
    abstract grafted surface into $\rho_z\backslash \X$. 
    The following 
    proposition is the main result of this subsection.

    \begin{proposition}\thlabel{piecewise-isometry}
        Consider a Hitchin grafting representation $\rho_z$ obtained from $\rho$ and with 
        grafting datum $z$.
        Recall that $S_{z}$ denotes the  abstract grafting of $S$ from Definition~\ref{def:abstractgraf}, with universal covering $\wt S_{z}$.
        Then there exists a piecewise totally geodesic immersed surface $\wt S^\iota_{z}\subset\X$ and a $\rho_z$-equivariant 
        immersion 
        $\wt Q_z\colon\wt S_{z}\to\wt S_{z}^\iota\subset\X$.
        
        The map $\wt Q_z$ is a path isometry for the Riemannian (resp.\ Finsler) metric on $\wt S_z$ and the induced path metric on $\wt S^\iota_{z}$ from the Riemannian (resp.\ Finsler) metric on $\X$.
    \end{proposition}

    Before we prove the proposition, note that the surface 
    $\wt S_{z}^\iota$ is $\rho_z$-invariant and hence descends to 
    compact piecewise smooth immersed surface $S_{z}^\iota\subset \rho_z\backslash \X$. 
    We call this surface \emph{characteristic}. 
    We shall show in Proposition~\ref{prop:embedding} that the corresponding map 
    $S_{z}\to S_{z}^\iota\subset \rho_z\backslash \X$ is actually an embedding, 
    and hence that $S_{z}^\iota$ is not only an immersed surface but an embedded one.

    Recall that the  Riemannian (resp. Finsler) \emph{cylinder height} of the Hitchin grafting representation $\rho_z$ 
    is the minimum of all 
    $d(z_\gamma,\R u)$ (resp. $d^{\mf F}(z_\gamma,\R u)$) for all $\gamma\in\gamma^*$, where $u=d\tau\hsl$ is the special direction in $\mf a$.

    \begin{proof}[Proof of Proposition~\ref{piecewise-isometry}]
     Denote by $\wt S$ and $\wt\gamma^*$, respectively, the universal cover of $S$ and the pre-image of $\gamma^*$ in $\wt S$. 
    Using the hyperbolic metric, we can fix an identification 
    $S=\pi_1(S)\backslash \H^2 $ so that $\wt S=\H^2$. 

    Let $\wt S_{z}$ be the universal cover of the abstract grafted surface $S_{z}$.
    This surface consists of a countable union $S_z^{hyp}$  
    of simply connected hyperbolic surfaces with geodesic boundary, called hyperbolic pieces in the 
    sequel, and a countable union of flat strips 
    separating these hyperbolic pieces. 
    Let ${\mathcal T}\subset \wt S$ be an embedded graph with one vertex $v$ in the interior
    of each of the hyperbolic pieces $\wt\Sigma_v$ of $\wt S-\wt\gamma^*$ and where two such points are connected by an edge $e$ if the
    pieces containing them are separated by a single 
    component $\tilde\gamma_e$ of $\wt\gamma^*$.
    To each vertex $v$ of $\mathcal T$ 
    is also associated a hyperbolic piece $\wt \Sigma^z_v\subset \wt S_z$ which is naturally isometric to $\wt \Sigma_v$.

    By construction, for any vertex $v$ of $\mc T$ the stabilizer $A_v:=\Stab_{\pi_1(S)}(\wt\Sigma_v)$ is mapped by $\rho_z$ onto a conjugate $g_v\rho(A_v)g_v^{-1}$ of $\rho(A_v)$ in $G$ 
    and hence it stabilises a unique totally geodesic embedded bordered surface $\widehat\Sigma^z_v=g_v\wt \Sigma_v\subset\X$ which is naturally isometric to $\wt \Sigma_v$ and $\wt \Sigma_v^z$.  
    Define $(\wt Q_z)_{\vert \wt \Sigma_v^s}:\wt \Sigma_v\to \widehat \Sigma^z_v$ 
    to be this natural isometry.
    By the construction of $\rho_z$,
    the thus defined map $\wt Q_z:\wt S_z^{hyp}\to \X$ is equivariant
    with respect to the representation $\rho_z$. 

    Consider an edge $e$ of $\mc T$ between two vertices $v=o(e)$ and $w=t(e)$ that projects onto a component $\gamma\subset\gamma^*$ matching the fixed orientation on $\gamma^*$.
    We also call $e\in\pi_1(S)$ the preferred generator of the stabilizer of $\wt\gamma_e=\wt\Sigma_v\cap\wt\Sigma_w$.
    Its holonomy $\rho_z(e)$ acts cocompactly by translation on boundary components $\wt c_1$ and $\wt c_2$ of $\widehat\Sigma^z_v$ and $\widehat\Sigma^z_w$, respectively.
    
    Let as before $u=d\tau\hsl$ be the special direction of $\mf a$.
    By construction, there exists a unique $h\in\PSL_d(\R)$ such that 
    \begin{itemize}
    \item 
    $h\rho_z(A_v)h^{-1}\subset\PSL_2(\R)\subset\PSL_d(\R)$; 
    \item $h\widehat\Sigma_v^z\subset\H^2\subset\X$;  
    \item $h\rho_z(e)h^{-1}=\exp(\ell_\rho(e)u)$;
    \item
    $h\wt c_1\subset\H^2\cap \mf a\subset\X$ is an axis of $\exp(u)$, that is, it is invariant under
    $\exp(u)$ and $\exp(u)$ acts on it as a translation.
    \end{itemize}
    
    Recall that $h\rho_z(A_v)h^{-1}=hg_v\rho(A_v)g_v^{-1}h^{-1}$ and $hg_v\in\PSL_2(\R)$.
    One can then check the following formula for the holonomy of the adjacent stabilizer $A_w$:
    $$
    h\rho_z(A_w)h^{-1}=\exp(z_\gamma)hg_v\rho(A_w)g_v^{-1}h^{-1}\exp(-z_\gamma),
    $$
    and hence $h\widehat \Sigma^z_w\subset\exp(z_\gamma)\H^2$ and $\wt c_2=\exp(z_\gamma)\wt c_1\subset\mf a$ is another axis of $\exp(u)$.
    Thus the flat strip $h^{-1}\{tu+sz_\gamma\}$ is $\rho_z(e)$-invariant, connects $\widehat\Sigma^z_v$ to $\widehat\Sigma^z_w$, is the only such flat strip, and is naturally isometric to the flat strip between $\wt\Sigma^z_v$ and $\wt\Sigma_w^z$ in $\wt S_z$.

 Doing this for all flat strips in $\wt S_{z}$ yields an extended map 
 $\wt Q_z\colon\wt S_{z}\to\X$, which is an isometry on each hyperbolic and flat piece. Furthermore,
 by construction, the map $\wt Q_z$ is continuous and $\rho_z$-equivariant.
           
   \begin{figure}
        \begin{center}
            \begin{picture}(123,50)(0,0)
            \put(0,0){\includegraphics[height=50mm]{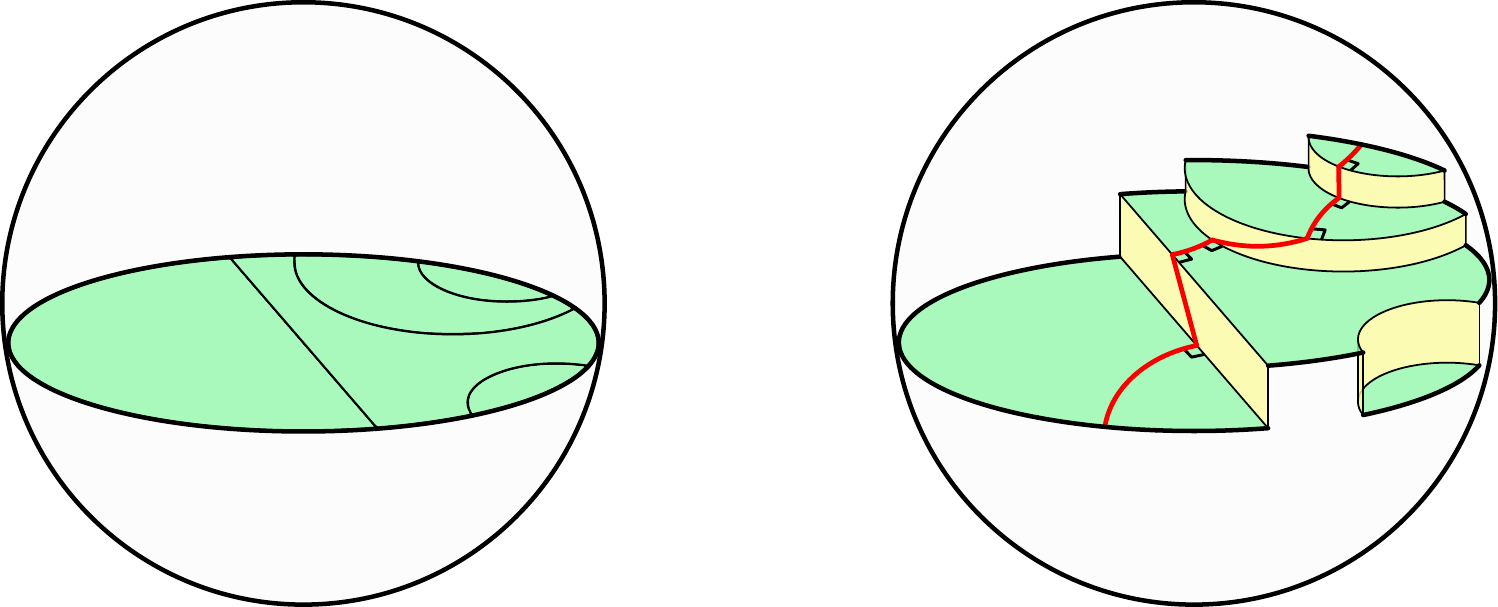}}
            \put(35,40){$\X$}
            \put(15,10){$\wb\H^2$}
            \put(30,17.5){$\pi(\wt\gamma^*)$}
            \end{picture}
        \end{center}
        \caption{Geometric description of the Hitchin representation: the hyperbolic part in green, the flat part in yellow and an admissible path in red. The choices of directions of the flat parts (up or down) are arbitrary; recall that $\gamma^*$ is a multicurve.}
        \label{fig-hitchin-rep}
    \end{figure}
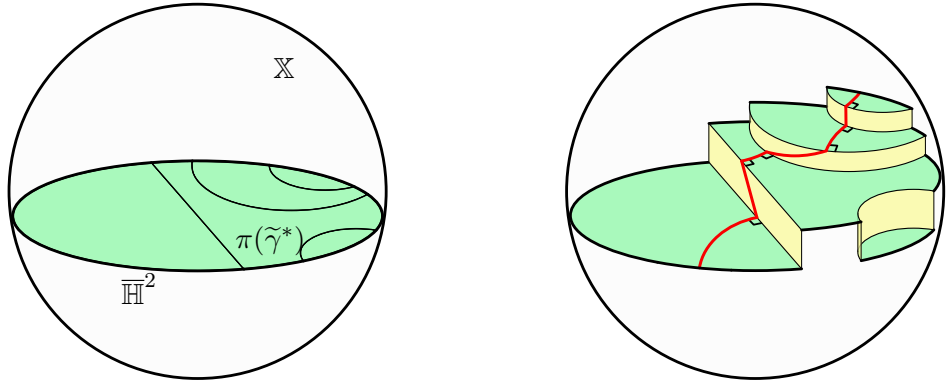

    \end{proof}

\subsection{Admissible paths}\label{sec:admissible-paths}

In Section~\ref{QI-proof} we shall show that the characteristic surface not only is embedded, but 
it also can be used effectively to compare the large scale geometry of the locally symmetric manifold 
$\rho_z\backslash \X$ to the large scale geometry of the grafted surface. This comparison relies on the analysis of 
some specific paths which we introduce now.

 \subsubsection{Admissible paths in abstract grafted surfaces}
    
   We begin with introducing a family of paths in grafted surfaces, called \emph{admissible paths}, which are from a technical point of view easier to handle than geodesics.
    
    \begin{figure}
        \begin{center}
            \begin{picture}(80,21)(0,0)
            \put(0,0){\includegraphics[width=80mm]{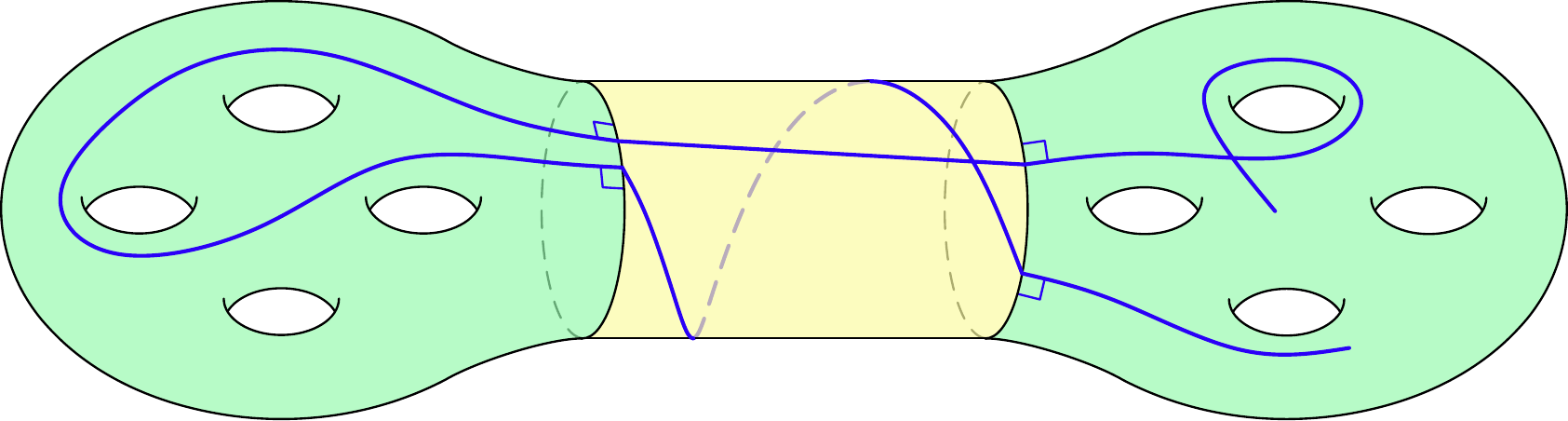}}
            \end{picture}
        \end{center}
        \caption{Admissible path in the closed surface obtained as an abstract grafting.}
        \label{fig-adm-path}
    \end{figure}

    Let $S_z$ be an abstract grafted surface with hyperbolic part $S^\hyp$ and cylinder part~$\mc C$.
    In a nutshell, an admissible path $c$ is a continuous path of $S_z$ which is a geodesic everywhere 
    except possibly at $S^\hyp\cap \mc C$, where it might have a singularity.
    Moreover, we require that the "hyperbolic part'' $c\cap S^\hyp$ of the path $c$ is orthogonal to $S^\hyp\cap\mc C$ where it meets it.

    It is clear that lifts of admissible paths to the universal cover are quasi-geodesics (although we will not need it).
    Our goal will be to show that the images of admissible paths under the map constructed in Proposition~\ref{piecewise-isometry}
    are quasi-geodesics of the symmetric space, with control on the multiplicative constant.

    \begin{definition}
     Consider a closed hyperbolic surface $S$, a  multicurve $\gamma^*\subset S$ and a grafting parameter $z$. 
     Then $S_z$ is the abstract grafted surface with hyperbolic part $S^\hyp\subset S_z$ and flat (cylindrical) part $\mc C\subset S_z$.
     An \emph{admissible path} in $S_z$ is a continuous path $c\subset S_z$ such that
     \begin{itemize}
         \item $c$ is geodesic outside $\mu=S^\hyp\cap\mc C$; 
         \item the hyperbolic part $c\cap S^\hyp$ intersects $\gamma^*$ orthogonally;
         \item a component of the flat part $c\cap \mc C$ connects the two distinct
         boundary components of the flat cylinder containing it.
     \end{itemize}
     Similarly one can define \emph{admissible loops}.
    \end{definition}

    Note that if $z$ is trivial then $S_z=S$ and the above definition still makes sense. 
    The flat part $\mc C$ is just $\gamma^*$, and the path is allowed to contain arcs in $\gamma^*$ 
    separating two geodesic arcs which emanate to the two distinct sides of $\gamma^*$ in a 
    tubular neighborhood of $\gamma^*$.

    An \emph{admissible path} in the universal cover $\tilde S_z$ is the lift of an admissible path in $S_z$.

    Note that any two points of $\tilde S_z$ are connected by a unique admissible path; in other words, any path of $S_z$ is homotopic (with fixed endpoints) to a unique admissible path.

    Similarly, any loop in $S_z$ not homotopic to a component of $\gamma^*$  
    is freely homotopic to a unique admissible loop.

    \begin{observation}\label{obs:corresp admissible paths in S and Sz}
        The image under $\pi_z:S_z\to S$ (or the lift $\wt S_z\to \wt S$) of admissible paths in $S_z$ are admissible paths in $S$.

        In fact, this induces a correspondence in the sense that any admissible path in $S$ is the image under $\pi_z$ of a unique admissible path in $S_z$.
    \end{observation}

    \subsubsection{Admissible paths in the symmetric space: geometric description}

    There are complete analogs of admissible paths in grafted surfaces for 
    the symmetric space $\X$ of $G$, which are also called admissible paths.
    Such paths include the image of all admissible paths in $\tilde S_{z}$ under a path isometry
    $Q_\zeta:\tilde S_{z}\to\X$ constructed in Proposition~\ref{piecewise-isometry}.


    Roughly speaking, admissible paths are piecewise geodesics that alternate between following a geodesic of the same type as the geodesics in the embedded $\H^2\hookrightarrow \X$, and then following a geodesic in a flat, orthogonal to the previous geodesic, and then following a $\H^2$-type geodesic orthogonal to the previous flat... etc, see Figure~\ref{fig-adm-path}.
    
    The above description is not quite correct, in particular because it does not 
    encapsulate the positivity assumption which is crucial in our proofs.
    There are several ways to define rigorously admissible paths.
    We are going to start with a geometric definition, which is easier to picture, and then in the next section we will give an algebraic definition.
    In the sequel the geometric definition will never be used, instead all the proofs will use the algebraic one, in particular because the positivity property of admissible paths is more naturally encoded in the algebraic definition.

    Recall that $\H^2$ embeds isometrically into $\X$.
    In fact there are many isometric embeddings, and $\PGL_d(\R)$ acts transitively on the set of all isometric embeddings.
    Let us call an \emph{$\H^2$-frame} the datum of a point $x\in\X$ and a pair of orthogonal unit tangent vectors $(v,w)$ which are tangent to a common embedded $\H^2$.
    Let $Y$ be the space of $\H^2$-frames, on which $\PGL_d(\R)$ also acts transitively. This action
    is even simply transitive since $\PGL_d(\R)$ is real split.

    On $Y$ there is a natural geodesic flow $(\mr{geod}_t)_{t\in \R}$: given $(x,v,w)\in Y$ one can follow the geodesic ray spanned by $v$ and parallel transport $v$ and $w$ along it. In other words, this action is 
    the action of the one-parameter group of transvections on $\X$ along the geodesic ray spanned by $v$.

    There is also a natural action of $\mf a$, which we shall call the ``orthogonal sliding action'' and denote $(\mr{slide}_z)_{z\in\mf a}$: given $(x,v,w)\in Y$ there is a unique maximal flat $F$ containing $w$, and a unique identification of the tangent space of 
    $F$ with $\mf a$ such that $w$ is sent into~$\mf a^+$. 
    Thus given $z\in \mf a$ one can follow the geodesic ray spanned by the associated vector in~$F$ and parallel transport $v$ and $w$ along it. Note that the image of the vector $v$ under this 
    sliding action remains orthogonal to the flat $F$. 

    We define an $(\omega,L)$-admissible path in $\X$ to be a path obtained by choosing an~$\H^2$-frame and pushing it via the geodesic flow for some time at least $\omega$, and then sliding orthogonally via $(\mr{silde}_{tz})_t$ for some time at least $L$ using some direction $z\in\mf a$, and then pushing along the geodesic flow again for time at least $\omega$...\ etc.

    In particular, an admissible path does not backtrack in any obvious way because it remembers directions along which 
    the path can be continued. This is the property which can be thought of as a geometric interpretation of positivity
    in the sense of~\cite{FG06}.
    In fact we get \emph{quantitative} positivity properties from the lower bound $\omega$ on the times we push along the geodesic flow.

    \subsubsection{Admissible paths in the symmetric space: algebraic definition}

    Let us now give an algebraic definition of admissible paths in $\X$.
    For this we will first define admissible paths in $G$. The description of these paths uses a 
    basepoint for the action of $G$ which is determined by the Fuchsian representation $\tau$.

    \begin{notation}\label{nota:a't}
     We set
     \begin{itemize}
      \item $a_t:=\tau\begin{psmallmatrix}e^t&0\\0&e^{-t}\end{psmallmatrix}\in G$;
      \item $r_\theta:=\tau \begin{psmallmatrix}\cos(\theta/2) & \sin(\theta/2) \\ -\sin(\theta/2) & \cos(\theta/2)\end{psmallmatrix}\in G$;
      \item $a'_t:=r_{\pi/2}\cdot a_t\cdot r_{\pi/2}^{-1}\in G$;
      \item for every $t\in\R\cup\{\infty\}=\partial\H^2$ we write $\xi_t=\partial_\infty \tau(t)$.
     \end{itemize}
    \end{notation}

    The group $G=\PSL_d(\R)$ identifies with one component of the space of $\H^2$-frames $Y$ introduced in the previous section via the orbit map $G\to Y$; $g\mapsto g\cdot F_o$, where $F_o=(o,v_o,w_o)$ is a fixed $\H^2$-frame, so that $o$ is fixed by $r_\theta$,  and $v_o=\tfrac{d}{dt}_{|t=0}a'_t\cdot o$ and $w_o=\tfrac{d}{dt}_{|t=0}a_t\cdot o$ are tangent to the axes of $a'_t$ and $a_t$, respectively.

    Under this identification, the geodesic flow on $Y$ corresponds to the multiplication on the right by $a'_t$: i.e.\ $\mr{geod}_t(gF_o)=(ga'_t)F_o$.
    On the other hand, the orthogonal sliding flow corresponds to the multiplication on the right by 
    $\exp(z)$: that is,\ $\mr{slide}_z(gF_o)=(g \cdot \exp(z))F_o$ for any $z\in \mf a$.
    This leads us to the following definition of admissible path.

    \begin{definition}\label{def:admissible in G}
        A path $c:[0,T]\to G$ or $c:[0,\infty)\to G$ is said to be of 
        \begin{itemize}
            \item \emph{flat type} if $c(t)=g \cdot \exp(tz)$ for some $g\in G$ and $z\in \mathfrak a$ of norm $1$ for the Finsler metric $\mf F$;
            \item \emph{hyperbolic type} if $c(t)=ga'_t$ for some $g\in G$.
        \end{itemize}

        An \emph{admissible path} of $G$ is a \emph{continuous} (possibly infinite) concatenation of paths of flat and hyperbolic type.

        It is moreover called \emph{$(\omega,L)$-admissible} for some parameters $\omega,L>0$ if all hyperbolic (resp.\ flat) pieces, except maybe the first and last pieces, have length at least $\omega$ (resp.\ $L$).

        A \emph{$(\omega,L)$-admissible path in $\X$} is a path of the form $t\mapsto c(t)\cdot \basepoint$ where $c$ is a $(\omega,L)$-admissible path of $G$; note that it is piecewise geodesic.
    \end{definition}

    \begin{remark}
        Another way to describe admissible paths in $G$ is the following: a path $c:[0,T]\to G$ is admissible, starting with a hyperbolic piece, if there exist $t_0=0<t_1<t_2<\dots<t_{n-1}<t_n=T$ and $z_1, z_3,\dots,z_k\in\mf a$ of norm $1$ (where $k$ is the biggest odd integer $<n$) such that for any $0\leq i<n$, for any $t\in [0,t_{i+1}-t_i]$, 
        \begin{itemize}
            \item if $i$ is even then $c(t_i+t)=c(t_i)\cdot a'_t$,
            \item if $i$ is odd then $c(t_i+t)=c(t_i)\cdot \exp(tz_i)$.
        \end{itemize}
    \end{remark}

The following is fairly immediate from the definition of the construction of the characteristic surface of a Hitchin
grafting representation $\rho_z$ and the map $\tilde Q_z$ from Proposition~\ref{piecewise-isometry}. In its formulation,
the \emph{collar size} of a simple closed multi-geodesic $\gamma^*\subset S$ is the supremum of all numbers $r>0$ such that
the tubular neighborhood of radius $r$ about $\gamma^*$ is a union of annuli about the components of $\gamma^*$. 
By hyperbolic geometry, an upper bound on the length of the components of $\gamma^*$ yields a lower bound on the collar size 
of $\gamma^*$.
    
    \begin{observation}\label{obs:admissiblemapstoadmissible}
     Consider a closed hyperbolic surface $S$, a multicurve $\gamma^*\subset S$ with collar size $\omega$ and a grafting parameter $z$ such that all cylinder heights are at least $L$. 
     Then the image under the grafting map $\tilde Q_z$ of any admissible path of $\tilde S_z$ is a $(\omega,L)$-admissible path of $\X$.
    \end{observation}

    \begin{numremark}\thlabel{rem:admissible closed}
        We define \emph{admissible loops} of a quotient of $\X$ as quotients of periodic infinite admissible paths.
    \end{numremark}

In Section~\ref{sec:fockgoncharov} we recall the notion of positivity in $G$ in the language of Lusztig~\cite{lusztig} and some basic results, and see that admissible paths have interesting positivity properties, coming from the fact that $(a'_t)_{t>0}$ are totally positive matrices, and $(\exp(z))_{z\in\mf a}$ are totally nonnegative matrices.

\section{A Morse-type lemma in the symmetric space}\label{sec:Morse}

The goal of this section is to establish a Morse-type lemma in the symmetric space.
Recall that in $\delta$-hyperbolic geodesic metric spaces, the Morse lemma says that any $(\lambda,C)$-quasi-geodesic is at distance at most $C'$ from a geodesic, where $C'$ depends on $\delta,\lambda,C$.
In $\R^2$, equipped with any norm, this lemma does not hold true anymore for all quasi-geodesics, and the same can be said for
the higher-rank symmetric spaces endowed with a Finsler or Riemannian metric, since they contain totally geodesic copies of $\R^2$.
Recall that a path~$c(t)$ in a metric space is a $(\lambda,C)$-quasi-geodesic if for all times $t,s$ we have
\[\lambda^{-1}|t-s|-C\leq d(c(t),c(s))\leq \lambda|t-s|+C.\]

Kapovich--Leeb--Porti~\cite{KLPMorse} (see also Section 12.1 of~\cite{KL18}) proved a Morse lemma
for certain families of well-behaved quasi-geodesics in 
symmetric spaces of arbitrary rank, equipped with the standard Riemannian metric, 
and in Euclidean buildings, see also Section 7 of~\cite{BPS19}. 
There is also a version of the Morse Lemma for quasi-flats instead of quasi-geodesics, see~\cite{KleinerLeeb,EskinFarb}.
We propose here a different approach to a generalization of the Morse lemma: we prove that nearby every \emph{Finsler} $(1,C)$-quasi-geodesic, so with \emph{multiplicative} error term of $1$, there is at least one {Finsler} geodesic.
Other Finsler geodesics could be far, as in $(\X,d^{\mf F})$ there are Finsler geodesics with the same endpoints and  
arbitrarily large Hausdorff distance.

We first present our result using the notion of quasi-ruled paths as in~\cite{BHM11}, and then translate it in terms of $(1,C)$-quasi-geodesics.
A \emph{$C$-quasi-ruled path} 
in a metric space $(X,d_X)$ is a map $c:[0,T]\to X$ 
such that for any $0\leq t\leq s\leq u\leq T$, 
$$d_X(c(t),c(s))+d_X(c(s),c(u))\leq d_X(c(t),c(u))+C.$$
Note that any reparameterization of a quasi-ruled path is quasi-ruled.

\begin{theorem}\label{thm:morse}
    For any $C>0$ there exists $C'>0$ such that 
    any Finsler $C$-quasi-ruled continuous path $c:[0,T]\to\X$
    is at Hausdorff distance at most $C'$ from a {Finsler} geodesic in $(\X,d^{\mf F})$
    connecting $c(0)$ to $c(T)$.
\end{theorem}

One can translate the above theorem in terms of $(1,C)$-quasi-geodesic paths, using the following lemma, which is probably well-known to experts. We provide a proof in Subsection~\ref{sec:rem qgeod}.

 \begin{lemma}\label{fact:qruled vs qgeod} 
     Let $(X,d)$ be a geodesic metric space and $C\geq 0$.
     Any $(1,C)$-quasi-geodesic in $X$ is $3C$-quasi-ruled and is at Hausdorff distance at most 
     $1+C$ from a continuous $(1,2(1+C))$-quasi-geodesic with the same endpoints.

     Conversely, any continuous $C$-quasi-ruled path is at Hausdorff distance at most $3C$ from a $(1,3C)$-quasi-geodesic.
 \end{lemma}
 Note that for $C=0$ this lemma says that any $(1,0)$-quasi-geodesic is a (continuous, $0$-quasi-ruled) geodesic, 
 and that for any continuous $0$-quasi-ruled path $c:[0,T]\to X$ from $x$ to $y$ there exists a geodesic $c':[0,d(x,y)]\to X$ from $x$ to $y$ whose image is exactly the same as that of $c$. 

\begin{remark} 
Theorem~\ref{thm:morse}  is false for the Euclidean metric on $\mathbb{R}^d$ $(d\geq 2)$, and for the Riemannian metric on $\X$,
as can be seen as follows.

Let $\ell:\mathbb{R}\to  \mathbb{R}^d$ be a line through 
$\ell(0)=0$ parameterized by arc length for the Euclidean metric. For $n\geq 1$ put 
$x_n=\ell(-n),y_n=\ell(n)$. Consider the balls $B_n^-,B^+_n$ of radius $n$ about 
$x_n,y_n$. As $n\to \infty$, the boundaries $\partial B_n^{\pm}$ 
of the balls $B_n^{\pm}$ converge in the 
pointed Gromov Hausdorff topology of $(\mathbb{R}^d,0)$ to the hyperplane through $0$ 
orthogonal to $\ell$. Thus for any $m>0$ and sufficiently large $n$, there are points 
$z_n^{\pm}$ on $\partial B_n^{\pm}$ of distance $m$ to $\ell$ with 
$d(z_n^-,z_n^+)\leq 1$ (here $d$ is the euclidean distance). 
Since the subsegment of 
$\ell $ connecting $x_n$ to $y_n$ is the unique euclidean geodesic between these points, 
the piecewise geodesics connecting $x_n$ to $y_n$ with breakpoints at 
$z_n^-,z_n^+$ violate the conclusion of 
Theorem~\ref{thm:morse}.
\end{remark}

This section is subdivided into four subsections. 
The first subsection is very short and provides a proof of Lemma~\ref{fact:qruled vs qgeod} for the reader's convenience.
In the second subsection, which is the longest, we establish
Theorem~\ref{thm:morse} for polyhedral norms on $\mathbb{R}^d$, that is, norms whose norm one ball is a finite sided
symmetric convex polyhedron. 
In the third subsection, we prove that any quasi-ruled continuous path in the symmetric space lies near a flat, using a description of Finsler horoballs of Kapovich--Leeb~\cite{KL18}.
Finally the last section, which is very brief, contains the proof of Theorem~\ref{thm:morse}.

\subsection{Proof of Lemma~\ref{fact:qruled vs qgeod}}\label{sec:rem qgeod}
Let $(X,d)$ be a geodesic metric space. Then for any $(1,C)$-quasi-geodesic 
$c:[0,T]\to X$ the following holds true.
\begin{itemize}
\item $c$ is $3C$-quasi-ruled.
\item $c$ is at Hausdorff distance at most $1+C$ from a continuous $(1,2+2C)$-quasi-geodesic with the same endpoints.
\end{itemize}

The first property is an elementary computation which is left to the reader, and the second property
is Lemma 1.11 in Chapter III.H of~\cite{bridson99}. The two properties together yields the first part of 
the lemma. 

    Let us prove the second part of the lemma. For an arbitrary $C\geq 0$
    consider a continuous $C$-quasi-ruled path $c:[0,T]\to X$ from $x$ to $y$.

    We now use an idea we found in the arXiv version of~\cite{BHM11}, Lemma A.2.
    Observe that the map
    $$f:c[0,T]\to [0,d(x,y)+C];\quad z\mapsto d(x,z)$$
    is continuous and is a $(1,C)$-quasi-isometry.
    Indeed, if $0\leq t\leq s\leq T$ then
    $$|f(c(t))-f(c(s))|=|d(x,c(t))-d(x,c(s))|\leq d(c(t),c(s)),$$
    and
    $$d(c(t),c(s))\leq d(x,c(s))-d(x,c(t))+C\leq |f(c(t))-f(c(s))|+C.$$
    Moreover, since $c[0,T]$ is path-connected, and $f$ is continuous and attains the values $0$ and $d(x,y)$, by the Intermediate Value Theorem $f$ attains all values in $[0,d(x,y)]$ and hence $f$ is $C$-quasi-surjective.
    By a classical result from coarse geometry $f$ admits a $(1,3C)$-quasi-inverse $g:[0,d(x,y)]\to c[0,T]\subset X$.

\subsection{A Morse-type lemma for normed vector spaces}\label{sec:MorseRd}

This subsection is entirely devoted to the study of the geometry of $\mathbb{R}^d$, equipped with 
a Finsler metric defined by a translation invariant norm on $T\mathbb{R}^d$. 
We begin with defining the Finsler metrics we are interested in.
To this end call a cone in $\mathbb{R}^d$ \emph{properly convex} if it is convex and its closure does not contain any affine subspace of $\R^d$ of dimension at least 1.

A (symmetric) polyhedral norm $|\cdot|$ on $\R^d$ is a norm of the form 
$$|v|=\max\{\alpha(v):\alpha\in \mc A\},$$
where $\mc A$ is a finite set of nonzero linear forms which spans $(\R^d)^*$, and which is symmetric in the sense that $-\mc A=\mc A$. 
This norm induces a metric $d(x,y)=|x-y|$ on $\R^d$ that is invariant under translations.
The goal of this section is to show.

\begin{proposition}\label{prop:finding a geod at bdd dist}
    For any polyhedral norm $\vert \cdot \vert$ on $\R^d$, there exists $\mu>0$ such that for any $C\geq 1$,
     any $C$-quasi-ruled continuous path $c:[0,T]\to\R^d$ is at Hausdorff distance at most $\mu C$ from a geodesic in $(\R^d, \vert \cdot\vert)$ connecting $c(0)$ to $c(T)$.
\end{proposition}

Note that this statement is false for a Euclidean norm on $\R^d$.

Proposition~\ref{prop:finding a geod at bdd dist} has the following reformulation in terms of $(1,C)$-quasi-geodesics, thanks to Lemma~\ref{fact:qruled vs qgeod}.

\begin{corollary}\label{cor:finding a geod at bdd dist}
    For any polyhedral norm $\vert \cdot \vert$ on $\R^d$, there exists $\mu>0$ such that for any $C\geq 1$,
    any $(1,C)$-quasi-geodesic $c:[0,T]\to\R^d$ is at Hausdorff distance at most 
    $\mu C$ from a geodesic in $(\R^d, \vert \cdot\vert)$ from $c(0)$ to $c(T)$.
\end{corollary}

\subsubsection{Diamonds}

In this section we introduce several geometric objects relative to our polyhedral norm, including diamonds.
We prove that $(1,C)$-quasi-geodesics stay at bounded distance from diamonds, which is the main technical
step towards the proof of Proposition~\ref{prop:finding a geod at bdd dist}.

For any $\alpha\in\mc A$, the set $\mc C_\alpha=\{v\in\R^d:|v|=\alpha(v)\}$ is a polyhedral convex cone based at $0$.
Note that $\mc C_{-\alpha}=-\mc C_{\alpha}$.
Up to removing unnecessary elements of $\mc A$, we may assume that $\mc C_\alpha$ has nonempty interior.
We call \emph{special cones} (based at $x\in\R^d$) the cones of $\R^d$ that are translates of a cone $\mc C_\alpha$ (by the translation $y\mapsto y+x$).

The unit closed ball $\bar B(0,1)$, and more generally any closed ball $\bar B(x,r)$ for such a norm, is a polyhedral convex set, that is,  a finite intersection of (affine) half-spaces of $\R^d$.
More precisely,
$$\bar B(x,r)=\bigcap_{\alpha\in\mc A} \{y\in\R^d:\alpha(y-x)\leq r\}=x+r\cdot \bar B(0,1).$$
The codimension-1 faces of $\bar B(x,r)$ are the intersections of its boundary $\partial B(x,r)$ with the special cones based at $x$.

\begin{definition}
    Denote by $\mc C(x\to y)$ the intersection of all special cones based at $x$ that contain $y$.
    We define the \emph{diamond} of the pair $x,y$ to be $D(x,y)=\mc C(x\to y)\cap \mc C(y\to x)$ (see Figures~\ref{fig-morse-proof} and~\ref{fig-morse-lemma} for illustrations).
\end{definition}

Note that  $\mc C(y\to x)=y-x-\mc C(x\to y)$.
This follows from the fact that for any $\alpha\in\mc A$, the special cone $x+\mc C_\alpha$ based at $x$ contains $y$ if and only if the special cone $y+\mc C_{-\alpha}$ based at $y$ contains $x$.

\begin{lemma}\label{lem: concatenation of geod}
    For any $x,y\in\R^d$, we have
    \begin{align*}
        D(x,y)
            &=\{z\in \R^d\mid d(x,z)+d(z,y)=d(x,y)\} \\
            &=\cup\{\text{geodesics from $x$ to $y$}\}
    \end{align*}
In particular, 
for any $z\in D(x,y)$, the concatenation of a geodesic from $x$ to $z$ with a geodesic from $z$ to $y$ is a geodesic from $x$ to $y$. 
\end{lemma}

Given a cone $\mc C=\mc C(0\to y)$, set $\alpha_{\mc C}$ to be the mean of all $\alpha\in\mc A$ for which $\mc C_\alpha$ contains~$\mc C$.
It follows from the definition that $|z|\geq \alpha_{\mc C}(z)$ holds for all $z\in\R^d$, with equality exactly on $\mc C$.

\begin{proof}
For a cone $\mc C=\mc C(x\to y)$, consider the form $\alpha_{\mc C}$ defined above. The point $z$ belongs to $D(x,y)$ if and only if  $|z-x|=\alpha_{\mc C(x\to y)}(z-x)$ and $|z-y|=\alpha_{\mc C(y\to x)}(z-y)=-\alpha_{\mc C(x\to y)}(z-y)$. 
This implies 
$$d(x,y)\leq d(x,z)+d(z,y) = \alpha_{\mc C(x\to y)}(z-x) + \alpha_{\mc C(x\to y)}(y-z) = \alpha_{\mc C(x\to y)}(y-x)= d(x,y)$$
and so $d(x,y) = d(z,x)+d(z,y)$.

Conversely if  $|z-x|>\alpha_{\mc C(x\to y)}(z-x)$ or $|z-y|>-\alpha_{\mc C(x\to y)}(y-z)$, then the above inequality yields $d(z,x)+d(z,y)>d(x,y)$.

It is clear that if $z$ lies on a geodesic from $x$ to $y$, then $d(x,z)+d(z,y)=d(x,y)$. Reciprocally, if $d(x,z)+d(z,y)=d(x,y)$ then the concatenation of any geodesic from $x$ to $z$ and any geodesic from $z$ to $y$ is a geodesic from $x$ to $y$.
\end{proof}

From the previous lemma and the triangle inequality we infer that for any point $z$ not too far from a diamond $D(x,y)$ we almost have the triangle equality $d(x,z)+d(z,y)\simeq d(x,y)$.
The following lemma is the key technical result towards the 
proof of Proposiion~\ref{prop:finding a geod at bdd dist}; it says that the converse also holds.

\begin{lemma}\label{lem:diamondtech}
    There is $\lambda_1>0$ such that for all $x,y,z\in\R^d$ it holds
    $$d(z,D(x,y))\leq \lambda_1(d(x,z)+d(z,y)-d(x,y)).$$
\end{lemma}

The two terms are equal to zero when $z$ belongs to $D(x,y)$. 
One can think of the lemma in the following way: 
$z\mapsto f_{x,y}(z)=d(x,z)+d(z,y)-d(x,y)$ is convex, non-negative, and piecewise affine. 
Take a point $z\in\partial D(x,y)$ and follow a ray $\{z+tv,t\geq 0\}$ for a 
unit vector $v$ for $\vert \cdot\vert$ at $z$ whose euclidean angle (for some fixed euclidean
inner product) to $D(x,y)$ is at least $\pi/2$. By this we mean the angle between $v$ and
any line segment in $D(x,y)$ starting at $z$.
The restriction of $f_{x,y}$ to the ray is convex, piecewise affine and is equal to zero exactly at $z$. 
It follows that it grows at least linearly in~$t$, the slope being given by the derivative at $t=0$. 
And so for $z'=z+tv$, one has $f_{x,y}(z')\geq t\cdot f_{x,y}'(0)\geq \mathrm {Cst}f_{x,y}'(0)\cdot d(z',D(x,y))$.

The issue is that the slope does not vary continuously in $z$, not even lower semi-continuously, so one can not hope 
to use a compactness argument to obtain a uniform bound on the union of the rays. 
One might study carefully the combinatorics of the map~$f$ to obtain a uniform bound on the slope. 
We instead take a slightly different approach, which requires one intermediate lemma.

Let us fix a Euclidean inner product $\langle , \rangle$ defining the 
Euclidean metric $d_\eucl$ on $\R^d$.
By "orthogonal projection" to a closed convex set $C$ 
we will mean closest-point projection for $d_\eucl$ to $C$, which is well defined by convexity of $d_\eucl$.

Given a cone $\mc C\subset\R^d$ based at $0$, define the dual cone of $\mc C$ to be the set 
$${\mc C}^\prime=\{x\in\R^d, \langle x,{\mc C}\rangle \leq 0\}
=\{x\in\R^d \text{ whose orthogonal projection to $\mc C$ is } 0\}.$$

\begin{lemma}\label{lem:dist cone and hyperplane}
 Let $\mc C$ be a polyhedral convex cone of $\R^d$ based at $0$, that is, the intersection of finitely many closed half-spaces $H_1,\dots,H_n$ containing $0$ in their boundary.
 Let $\mc C'\subset\R^d$ be the polyhedral convex dual cone to $\mc C$.
 Then there exists $\lambda>0$ such that for any $x\in\mc C'$,
 $$d_{\eucl}(x,\mc C)\leq \lambda \max\left(d_{\eucl}(x,H_1),\dots,d_{\eucl}(x,H_n)\right).$$
\end{lemma}

The results holds true for all $x\in\R^d$ (for a bigger constant $\lambda$), but the special case $x\in\mc C'$ is shorter to prove.

\begin{proof}
    This is an immediate consequence of the fact that the function 
    $$f(x)= \max\left(d_{\eucl}(x,H_1),\dots,d_{\eucl}(x,H_n)\right)$$
    is homogeneous, continuous, and positive on $\mc C'-\{0\}$. 
\end{proof}

Note that in the previous lemma we allow $\mc C$ to have empty interior, or be reduced to~$\{0\}$, or to be the entire space $\R^d$ (but this last case is not very interesting since then the dual $\mc C'$ is just $\{0\}$).

\begin{proof}[Proof of Lemma~\ref{lem:diamondtech}]
Denote by $\{H_\alpha,\alpha\in\mc A\}$ the finite family of closed 
half spaces given by $H_\alpha=\{w\in\R^d, \alpha(w)\leq 0\}$.

For any subset $S$ of $\mc A$, the intersection
\[\mc C_S=\cap_{\alpha\in S}H_\alpha\] 
is a polyhedral convex cone. 
Let $K_S$ be the dual cone to $\mc C_S$.
We can apply Lemma~\ref{lem:dist cone and hyperplane} to~$K_S$, and get a number $\lambda_S>0$.
Let $\lambda=\max \{\lambda_S\mid S\subset\mc A\}$.

We now prove the inequality for $x,y,z\in\R^d$ fixed.
When $z$ belongs to $D(x,y)$, we have $d(x,z)+d(z,y)- d(x,y)=0= d(z,D(x,y))$ by Lemma~\ref{lem: concatenation of geod} and so the inequality holds.

Suppose that $z\not\in D(x,y)$.
If $d(x,z)\geq d(x,y)$ holds, then one has 
\[d(x,z)+d(z,y)- d(x,y)\geq d(z,y)\geq d(z,D(x,y))\] since $y\in D(x,y)$.
So by symmetry in $x,y$ we may assume that $r:=d(x,z)$ is smaller than $R:=d(x,y)$.

Let $B_x=\bar B(x,r)$ and $B_y=\bar B(y,R-r)$ be closed balls for the polyhedral norm $|\cdot |$, illustrated in Figure~\ref{fig-morse-proof}.
They are polyhedral convex sets, \ie finite intersections of affine half-spaces. More precisely 
$$B_x=H_1\cap\dots\cap H_n \quad \text{and} \quad B_y=H_1'\cap \dots\cap H_n'$$
where $H_i=\{w\in\R^d,\alpha_i(w-x)\leq r\}$  and $H_j'=\{w\in\R^d,\alpha'_j(w-y)\leq R-r\}$ for some orderings $\mc A=\{\alpha_1,\dots,\alpha_n\}=\{\alpha_1',\dots,\alpha'_n\}$ (it will be convenient later to have two different orderings).

\begin{figure}
    \begin{center}
        \begin{picture}(100,83)(0,0)
        \put(0,0){\includegraphics[width=100mm]{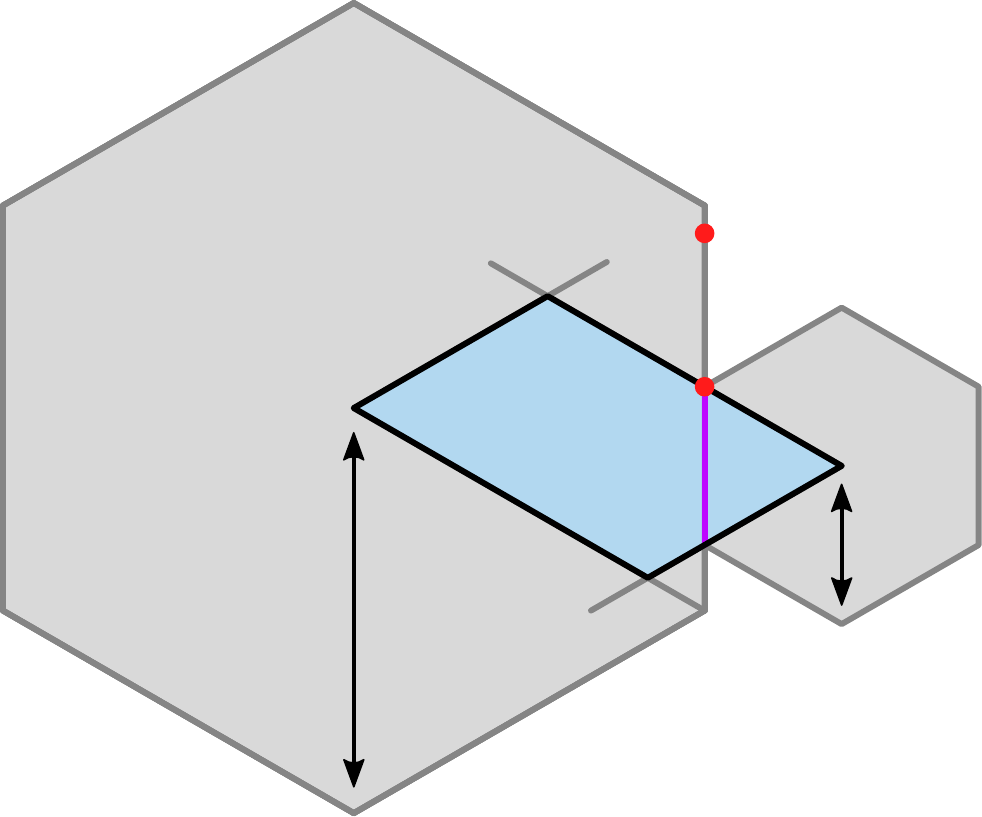}}
        \put(32.5, 41){$x$}
        \put(87, 35){$y$}
        \put(48, 41){$D(x,y)$}
        \put(73.5, 47){$p$}
        \put(73.5, 58.5){$z$}
        \put(32, 19){$r$}
        \put(87, 26.5){$R-r$}
        \put(64.8, 33){$B_x\cap B_y$}
        \end{picture}
    \end{center}
    \caption{Illustration of the proof of Lemma~\ref{lem:diamondtech}. Three points $x,y,z$ and a fourth point $p\in D(x,y)$ with $d(z,p)\leq\lambda(d(x,z)+d(y,z)-d(x,y))$.}
    \label{fig-morse-proof}
\end{figure}
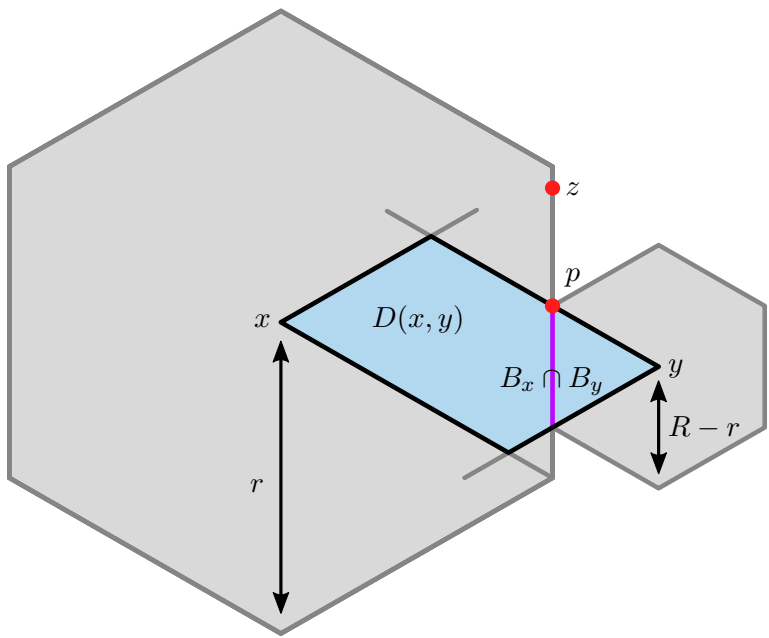

The intersection 
$$B_x\cap B_y=H_1\cap\dots\cap H_n\cap H_1'\cap\dots\cap H_n'$$
is a closed polyhedral convex subset of the diamond $D(x,y)$ by Lemma~\ref{lem: concatenation of geod}, 
with empty interior, and it is not empty (one can verify that $B_x\cap B_y$ contains the point $\frac{r}{R}y+\frac{R-r}{R}x$). 

Let $p$ be the Euclidean closest-point projection of $z$ to $B_x\cap B_y$.
Up to translation, we may assume that $p=0$ to be able to use Lemma~\ref{lem:dist cone and hyperplane}.
Up to reordering we may also assume that the half-spaces containing $p$ in their boundary are $H_1,\dots,H_k$ and $H'_1,\dots,H'_\ell$; note that $k$ and $\ell$ are both positive since $p\in\de B_x\cap\de  B_y$.
Since $p=0$ we have $H_i=\{w:\alpha_i(w)\leq 0\}$ for $i\leq k$ and $H_j'=\{w:\alpha'_{j}(w)\leq 0\}$ for $j\leq l$.
Let $S=\{\alpha_i,i\leq k\}\cup\{\alpha'_j,j\leq \ell\}\subset\mc A$.
Then using the notation introduced at the beginning of the proof we have
$$B_x\cap B_y\subset H_1\cap\dots\cap H_k\cap H_1'\cap\dots\cap H'_\ell=\mc C_S.$$

Observe that $p$ is also the Euclidean closest-point projection of $z$ to $\mc C_S$.
Indeed if by contradiction there was $p'\in \mc C_S$ closer to $z$, 
then by convexity of the euclidean distance function,
any point of the line segment $(p,p']$ would be closer to $z$. But any point of $(p,p']$ close enough to $p$ is 
contained in each $H_{k+1},\dots,H_n$ and $H_{\ell+1}',\dots,H_n'$ (since $p$ is in their interior), and hence is in $B_x\cap B_y$, which contradicts that $p$ is the Euclidean closest-point projection of $z$ on $B_x\cap B_y$. In particular, $z$ is 
contained in the dual cone $K_S$ to $\mc C_S$.

By Lemma~\ref{lem:dist cone and hyperplane}, the distance $d(z,p)$ is comparable to the distance between $z$ and one of the half spaces $H_1,\dots,H_k,H_1',\dots,H'_\ell$. But since $z\in B_x$, it must be 
contained in every $H_i$, so Lemma~\ref{lem:dist cone and hyperplane} implies that there exists a half space $H'_j$, $j\in\{1,\dots,\ell\}$, for which 
\begin{align*}
    d_\eucl(z,D(x,y)) 
        &\leq d_\eucl(z,p) \\
        &\leq \lambda_Sd_\eucl(z,H_j') \\
        &\leq \lambda d_\eucl(z,B_y).    
\end{align*}

Let us translate what this means for the polyhedral norm, using a constant $\nu$ such that $\nu^{-1}d\leq d_\eucl\leq \nu d$. The previous equation yields 
$$d(z,D(x,y))\leq \nu^2\lambda d(z,B_y).$$

Let $q$ be the intersection point of $[y,z]$ with $\partial B_y$, which satisfies $d(z,q)=d(z,B_y)$.
Indeed for any $q'\in \partial B_y$ we have $d(q',y)=d(q,y)$, and so
$$d(z,q')=d(z,q')+d(q',y)-d(q',y)\geq d(z,y)-d(q,y)=d(z,q)=R-r$$
since $[y,z]$ is a geodesic for $d$.
Then 
\begin{align*}
    d(x,z)+d(z,y)-d(x,y) & = r + d(z,y) - R \\
    & = d(z,q) + d(q,y) +r - R \\
    & = d(z,q) = d(z, B_y)\\ 
    & \geq \nu^{-2} \lambda^{-1} d(z,D(x,y)).\qedhere
\end{align*}
\end{proof}


\subsubsection{Proof of Proposition~\ref{prop:finding a geod at bdd dist}}

In addition to diamonds we shall consider two other geometric objects: crowns and cores of diamonds. 
We refer to~\cite{KLPMorse} for closely related constructions and to Figure~\ref{fig-morse-lemma} for an illustration.
\begin{itemize}
    \item $x,y\in\R^d$ are called \emph{generic} if the diamond between them has nonempty interior, 
    that is, if there is only one $\alpha\in\mc A$ such that $|x-y|=\alpha(x-y)$.
    \item The \emph{Crown} $\Cr(x,y)$ is $\partial \mc C(x\to y)\cap \partial \mc C(y\to x)$ if
    $x,y$ are generic, and otherwise it is just $D(x,y)$. Note that if $z\in \Cr(x,y)$ then the pair $(x,z)$ is \emph{not} generic.
    \item The \emph{Core} $\Co(x,y)$ is the convex hull of the crown, which is just $D(x,y)$ if $x$ and $y$ are not generic.

\end{itemize}
Note that for generic $x,y$, the core ${\rm Co}(x,y)$ separates $x$ from $y$ in $D(x,y)$, in the sense that $x$ and $y$ are in different connected components of $D(x,y)\smallsetminus\Co(x,y)$.

Note also that for generic $x,y$ the intersection of the core $\Co(x,y)$ with the boundary $\partial D(x,y)$ of the diamond is exactly the crown $\Cr(x,y)$.

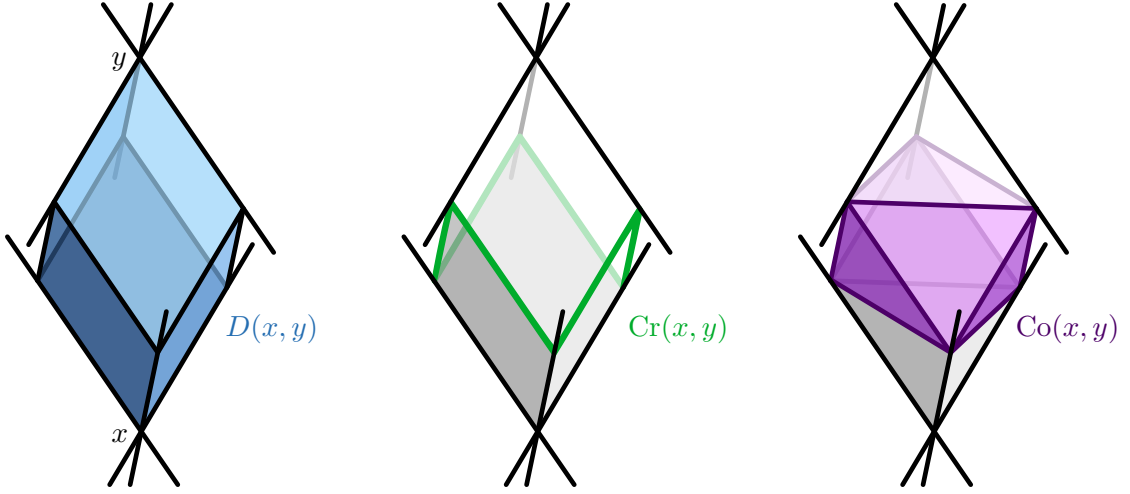
\begin{figure}
    \begin{center}
        \begin{picture}(140,65)(0,0)
        \put(0,0){\includegraphics[width=140mm]{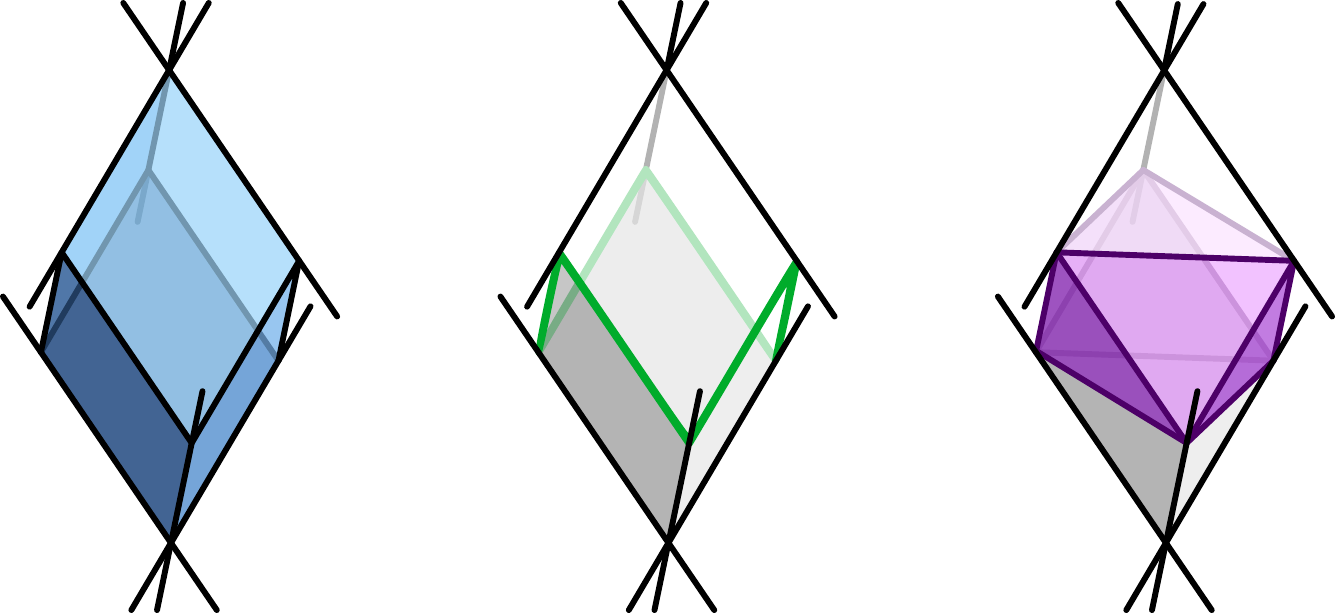}}
        \put(14, 6){$x$}
        \put(14, 56){$y$}
        \color[RGB]{37, 110, 178}
        \put(29, 20){$D(x,y)$}
        \color[RGB]{0, 172, 43}
        \put(82, 20){$\Cr(x,y)$}
        \color[RGB]{77, 0, 104}
        \put(133, 20){$\Co(x,y)$}
        \end{picture}
    \end{center}
    \caption{Illustration of the diamond, crown and core of two points $x,y\in\R^d$ in generic position.}
    \label{fig-morse-lemma}
\end{figure}

We will need the following elementary result on properly convex cones, which implies that if a tip of a diamond is not too close to the crown then it is not too close to the core.

Given a set $K\subset\R^d$, denote by $\Conv(K)$ the convex hull of $K$.

\begin{lemma}\label{lem:prop3}
Let $\mc C\subset \R^r$ be a closed properly convex cone with vertex $0$.
Then there exists $\lambda_2>0$ such that for any compact set $K\subset \mc C$, we have
\[d(0,K)\leq \lambda_2 d(0,\Conv(K)).\]
\end{lemma}
\begin{proof} 
Since $\mc C$ is properly convex, $0$ is an extremal point of it and does not belong to the convex hull $\Conv(\mc C- B(0,1))$ where $B(0,1)$ is the open ball of radius one around $0$ for the metric $d$.
Let $\lambda^{-1}=d(0,\Conv(\mc C- B(0,1)))$ be the distance from $0$ to $\Conv(\mc C- B(0,1))$.

Let $K\subset\mc C$ be a compact subset.
If $0\in K$ then $d(0,K)=0\leq \lambda d(0,\Conv(K))$.

Suppose $0\not\in K$ and and put $a=d(0,K)^{-1}$.
Then the compact $a\cdot K$ is included in $\mc C-B(0,1)$ and so $\Conv(a\cdot K)\subset \Conv(\mc C-B(0,1))$, which yields $d(0,\Conv(a\cdot K))\geq \lambda^{-1}$.
As $\Conv(a \cdot K)=a\cdot \Conv(K)$, one has
\begin{equation*}d(0,\Conv(K))=d(0,K)\cdot d(0,a\cdot\Conv(K))\geq \lambda^{-1} d(0,K).\qedhere\end{equation*}
\end{proof}

Finally we will need the following observation about quasi-ruled paths, whose proof is a simple calculation.
\begin{observation}\label{obs:finite distance to qruled}
Let $a,b:[0,T]\to X$ and $c:[T,T']\to X$ be paths in a metric space.
    \begin{enumerate}
        \item If $a$ is $C$-quasi-ruled and $d(a(t),b(t))\leq C'$ for all $t$ then $b$ is $C+6C'$-quasi-ruled.
        \item If $a$ is $C$-quasi-ruled and if $d(c(t),a(T))\leq C'$ for all  $t$ then the concatenation of the path $a$ with $c$ is $C+2C'$-quasi-ruled.
    \end{enumerate}
\end{observation}

\begin{proof}[Proof of Proposition~\ref{prop:finding a geod at bdd dist}]
We proceed by induction on the dimension $d$. 
In the case $d=1$ there is nothing to show, so assume that for some $d\geq 2$, 
the claim holds true 
for all dimensions $<d$, with a constant depending on the dimension and on the 
polyhedral norm.

Let $\vert \cdot\vert$ be a polyhedral norm on $\R^d$. 
Note that the restriction of  $\vert \cdot\vert$ to any 
linear subspace is polyhedral.
Thus by the induction assumption, there
exists a constant $\mu>1$ so that the proposition is valid with this constant
for paths contained in the linear subspaces $\{\alpha_{i_1}=\cdots =\alpha_{i_k}\mid 
\alpha_{i_j}\in {\mc A}\}$ which are the linear spans of the faces of 
the special cones ${\mc C}_\alpha$ for $\alpha \in {\mc A}$.

Let $\lambda_1>0$ be as in  Lemma~\ref{lem:diamondtech}, $\nu>1$ be such that $\nu^{-1}d\leq d_\eucl\leq \nu d$
where $d(x,y)=\vert x-y\vert$, and $\lambda_1'=\nu^2\lambda_1$.
Then for all $x,y\in \R^d$, if 
$$\Pi_{xy}:\R^d\to D(x,y)$$
is the Euclidean closest-point projection onto $D(x,y)$ (which is well defined continuous since $D(x,y)$ is compact and convex, 
contrarily to the closest-point projection for $d$), then by Lemma~\ref{lem:diamondtech} we get
\begin{equation}\label{eq:diamondtech}
d(z,\Pi_{xy}(z))\leq \nu d_\eucl(z,\Pi_{xy}(z))
= \nu d_\eucl(z,D(x,y)) 
\leq \lambda_1' (d(x,z)+d(z,y)-d(x,y)).\end{equation}

Let $\lambda_2>0$ be the maximum of the constants 
from Lemma~\ref{lem:prop3}, applied to the special cones $\mc C_\alpha$, $\alpha \in\mc A$.
Then for all $x,y\in\R^d$ we have
\begin{equation}\label{eq:crown and core}d(x,\Cr(x,y))\leq \lambda_2d(x,\Co(x,y)).\end{equation}

We claim that the statement of the proposition holds true for 
$(\R^d,\vert \cdot\vert)$ with the constant $\mu'=(1+2\lambda)(\mu(1+6\lambda_1')+\lambda_1')$, where $\lambda=\max(\lambda_1',\lambda_2)$.

To this end we proceed by induction on $k$ where 
$d(x,y)\in (k-1,k]$.

Let $C\geq 1, k=1$ and let $c$ be a $C$-quasi-ruled continuous path from $x$ to $y$ such that $d(x,y)\leq 1$. Then the segment $[x,y]$ is at distance at most $C+1$ from $c$ and hence the claim holds true in this case.

Let $k\geq 2$ and assume that
the claim holds true for all continuous $C$-quasi-ruled paths from $x$ to $y$ such that $d(x,y)\leq k-1$.
Let $c$ be a $C$-quasi-ruled continuous path from $x$ to $y$ such that $d(x,y)\in(k-1,k]$.

There are two possible cases. 
In the first case, 
$c$ stays $C\lambda $-far from the crown $\Cr(x,y)$.

Fix a point $z\in {\rm Co}(x,y)$ which is in the interior of the diamond $D(x,y)$.
Let $Z\subset {\rm Co}(x,y)$ be the union of segments $[z,p]$ where $p\in {\rm Cr}(x,y)$.
Then $Z\cap \partial D(x,y)={\rm Cr}(x,y)$ and $Z$ separates $x$ from $y$ in $D(x,y)$: if $a(t)\in D(x,y)$ is a continuous path from $x$ to $y$ then it must cross $Z$. 
If it crosses $z$ there is nothing to prove.
Otherwise, for each $t$  the ray from $z$ passing through $a(t)$ must cross $\partial D(x,y)$ at some point $b(t)$ that depends continuously on $t$, and at some time $t$ we have $b(t)\in {\rm Cr}(x,y)$ hence $a(t)\in[b(t),z]\subset Z$.

The projection $\Pi_{xy}\circ c$ is a continuous path from $x$ to $y$ in $D(x,y)$, 
so it must cross $Z$ at some time $t$.

Note that we have $\Pi_{xy}(c(t))=c(t)$.
Namely, otherwise $\Pi_{xy}(c(t))$ is contained in the boundary $\partial D(x,y)$, 
and hence contained in the crown. 
But since $\Pi_{xy}(c(t))$ is $C\lambda_1' $-close to $c(t)$ by Inequality~(\ref{eq:diamondtech}),
then $c(t)$ is $C\lambda_1' $-close and hence $C\lambda $-close to the crown, which contradicts our assumption.

Moreover, we must also have the inequalities 
$$d(x,c(t))\geq d(x,\Co(x,y))\geq 1 \text{ and } d(y,c(t))\geq 1.$$
Otherwise by Equation~(\ref{eq:crown and core}),
it holds $d(x,\Cr(x,y))\leq \lambda_2\leq \lambda\leq C\lambda$, which contradicts our assumption 
that $c$ stays $C\lambda$-away from the crown. Lemma~\ref{lem: concatenation of geod} yields that 
$$d(x,c(t))=d(x,y)-d(y,c(t))\leq k-1$$
and similarly $d(y,c(t))\leq k-1$.
We can apply the induction hypothesis (on $k$) to the 
path $c[0,t]$ and the path $c[t,T]$. 
As the concatenation of a geodesic connecting 
$x$ to~$c(t)\in D(x,y)$ and $c(t)$ to $y$ is a geodesic, this suffices for the induction step.

In the second case, $c$ passes at some time $t$ at distance less than $C\lambda$ from the crown.
Let $p\in \Cr(x,y)$ be such that $d(p,c(t))\leq C\lambda$.
Recall that this means the pairs $(x,p)$ and~$(y,p)$ are not generic.

Concatenate $c_1=c[0,t]$ with a geodesic from $c(t)$ to $p$, to get a continuous 
$(1+2\lambda)C$-quasi-ruled path $c'_1$ from $x$ to $p$ by Observation~\ref{obs:finite distance to qruled}.
By Inequality~(\ref{eq:diamondtech}), the projection $\Pi_{xp}\circ c'_1$ 
is at distance at most $\lambda_1'(1+2\lambda)C$ from $c_1$. Using again  Observation~\ref{obs:finite distance to qruled},
it is a continuous $(1+2\lambda)C(1+6\lambda_1')$-quasi-ruled path  
in $D(x,p)$ from $x$ to $p$.

As the pair $(x,p)$ is not generic, we can apply 
our induction on the dimension and deduce that 
there is a geodesic $c_1''\subset D(x,p)$ 
at distance at most $\mu C(1+2\lambda)(1+6\lambda_1')$ from $\Pi_{xp}\circ c_1'$, which is then at distance at most 
$C(1+2\lambda)(\mu(1+6\lambda_1')+\lambda_1')$
from the original path $c$.

With a similar construction for $c_2=c[t,T]$, one obtains a geodesic $c_2''$ from $p$ to $y$ which is at distance at most 
$C(1+2\lambda)(\mu(1+6\lambda_1')+\lambda_1')$
from $c$.

The concatenation of $c_1''$ and $c_2''$ is by Lemma~\ref{lem: concatenation of geod} a geodesic from $x$ to $y$ at distance at most 
$C(1+2\lambda)(\mu(1+6\lambda_1')+\lambda_1')$
from $c$, which concludes the proof.
\end{proof}

\subsection{Projecting to a flat}\label{sec:proj on diamond}

In this subsection we extend Proposition~\ref{prop:finding a geod at bdd dist} 
to the symmetric space $\X=\PSL_d(\R)/\PSO(d)$ 
equipped with the Finsler metric
$d^{\mf F}$. We begin with extending the geometric 
notions from Section~\ref{sec:MorseRd} to $\X$.

Recall, for instance from~\cite{KL18}, that the \emph{diamond} between two points $x,y\in\X$ is defined as follows.
A \emph{Weyl cone} of a flat of $\X$ is the translation under an element of $\PSL_d(\R)$ 
of the standard 
Weyl cone $\exp \mf a^+\subset \exp \mf a$ based at the basepoint $\basepoint$ of the standard flat $\exp \mf a$.
Consider a flat $F$ containing $x$ and $y$, a Weyl cone 
$W\subset F$ based at $x$ and containing $y$, and the opposite Weyl cone $W'$ based at $y$ (which automatically contains $x$).
Then the \emph{diamond} $D(x,y)$ is defined as
\[D(x,y)=W\cap W^\prime.\]
It does not depend on choices, and as an intersection of convex subsets of 
$\X$, it is convex. 
The analog of Lemma~\ref{lem: concatenation of geod} holds true.

\begin{proposition}[{Lemma 5.10 of~\cite{KL18}}]\label{prop:diamond}
    For all $x,y\in\X$, the diamond $D(x,y)$ is the set of points $z\in \X$ such that $d^{\mf F}(x,z)+d^{\mf F}(z,y)=d^{\mf F}(x,y)$.
\end{proposition}

The following non-uniform version of Lemma~\ref{lem:diamondtech} for $(\X,d^{\mf F})$ 
is used to reduce Theorem~\ref{thm:morse} to 
Proposition~\ref{prop:finding a geod at bdd dist}. 

\begin{proposition}\label{prop:coarse diamond}
    For any $C\geq 1$ there exists $C'>0$ such that for all $x,y\in\X$, any $z\in\X$ such that $d^{\mf F}(x,z)+d^{\mf F}(z,y)\leq d^{\mf F}(x,y)+C$ is at distance at most $C'$ from $D(x,y)$.
\end{proposition}

\begin{proof}
For the purpose of this proof, let us simplify notations by writing $d$ instead of $\df$.
   
  Suppose for a contradiction that there exists sequences $(x_n)_n$, $(y_n)_n$ and $(z_n)_n$ such that $d(x_n,z_n)+d(z_n,y_n)\leq d(x_n,y_n)+C$ for all $n$ but the distance from $z_n$ to $D(x_n,y_n)$ tend to infinity as $n\to\infty$.

  First note that $d(x_n,z_n)\geq d(z_n,D(x_n,y_n))\to\infty$, and similarly $d(z_n,y_n)\to\infty$.
  Since $d(x_n,y_n)\geq d(x_n,z_n)+d(z_n,y_n)-C$ we have $d(x_n,y_n)>d(x_n,z_n)$ for large enough $n$.

  Let $w_n\in D(x_n,y_n)$ be at distance exactly $d(x_n,z_n)$ from $x_n$. Such a point exists by compactness of $D(x_n,y_n)$. 
  Up to translating everything, we can assume that $w_n$ equals the basepoint $\basepoint$ of the symmetric space
  for all $n$, and that $x_n,y_n$ are contained 
  in the standard flat $A\subset \X$. It follows from the explicit construction of the diamond $D(x_n,y_n)$ that the points $x_n,y_n$ 
  are contained in antipodal Weyl cones with vertex $\basepoint$, that is, we may assume that 
  $y_n$ is contained in the standard (closed) Weyl cone $A^+$, and $x_n$ is contained in the opposite 
  (closed) Weyl cone $A^-$.

  Consider a Riemannian geodesic ray $(w_n(t))_{t\geq0}$ (with unit \emph{Finsler} speed) starting at $\basepoint$ and passing through $z_n$, say at time $R_n=d(\basepoint,z_n)$.

  We claim that for any $0\leq t\leq R_n$,
  \begin{equation}\label{claim: almost triangle equality}
    d(x_n,w_n(t))+d(w_n(t),y_n)- d(x_n,y_n)\leq C 
  \end{equation}
   Namely, since by 
   \cite[Eq.\,(5.4)]{KL18} Finsler balls of fixed radius are convex, the segment $t\to w_n(t)$ $(t\in [0,R])$ 
   stays in the Finsler balls around $x_n$ and $y_n$ of respective radius $d(x_n,z_n)$ and $d(y_n,z_n)$.

  Up to extracting a subsequence, we can assume that the sequence of rays $w_n(t)$ converge to a ray $w(t)$ starting at $\basepoint$ and ending at some $w(\infty)\in\partial_\infty \X$.
  We now claim that 
  \begin{equation}\label{claim: ray in flat}
      w(t) \text{ is contained in the flat } A.
  \end{equation}
  Let us prove it.
  The ball of radius $d(y_n,z_n)$ around $y_n$ contains $w_n(t)$ for any $t\leq R_n$, and this ball converges as $n\to\infty$ to a Finsler horoball, which hence contains $w(t)$ for all $t$.
  By the paragraph about Busemann functions in Section~\ref{sec:lietheoryI}, this Finsler horoball is contained in a Riemannian horoball $B_+$ around the point $\xi_+=\alpha_0^\myhash\in\partial_\infty A^+$ (see Notation~\ref{nota:alpha}).

  Similarly, $w(t)$ is contained in a Riemannian horoball $B_-$ around $\xi_-\in \partial_\infty A^-\cap G\cdot \alpha_0^\myhash$, which is antipodal to $\xi_+$ as $\alpha_0$ was chosen symmetric with respect to the Cartan involution.

  By Lemma~\ref{lem-Tits-distance-bounded} below, the \emph{Tits distance} between $\xi_\pm$ and $w(\infty)$ is at most $\tfrac\pi2$.
  Recall that the Tits distance between $a,b\in\partial_\infty\X$ is the angle between two geodesic rays going to respectively $a$ and $b$ and that are contained in the same flat (see for instance~\cite[Part II, Ch.\,9]{bridson99} for an account on the Tits metric).

  Since $\xi^-$ and $\xi^+$ are antipodal, their Tits distance is exactly $\pi$.
  It follows from Lemma~\ref{lem-Tits-distance-pi} that the limit point $w(\infty)$ lies on the boundary at infinity of the unique flat which contains $\xi^-$ and $\xi^+$ in its boundary, that is the standard flat $A$.
  Since $w(0)$ is also in that flat, the whole geodesic ray $(w(t))_t$ remains in the same flat $A$, and this ends the proof of the claim \eqref{claim: ray in flat}.

  Now we combine \eqref{claim: almost triangle equality}, \eqref{claim: ray in flat} and Lemma~\ref{lem:diamondtech} to conclude.
  Fix $\lambda>0$ such that for all $x,y,z\in A$ we have 
  $$d(z,D(x,y))\leq \lambda(d(x,z)+d(z,y)-d(x,y)).$$
  Then for any $t$, taking $n$ large enough we have $d(x_n,w(t))+d(w(t),y_n)- d(x_n,y_n)\leq C+1$ and hence
  $$d(w(t),D(x_n,y_n))\leq \lambda(C+1),$$
  since $w(t)\in A$.
  Coming back to $w_n(t)$ we deduce, for $n$ large enough, that 
    $$d(w_n(t),D(x_n,y_n))\leq \lambda(C+1)+1.$$
  But this should be a contradiction, as $w_n(t)$ is a ray going to $z_n$ which is very far from $D(x_n,y_n)$.
  To get a true contradiction, one must be careful in the choice of $w_n$ we made at the beginning: one must choose a
  point of the compact set 
  $\partial B(x_n,d(x_n,z_n))\cap D(x_n,y_n)$ which is closest to $z_n$ for the Finsler metric. Then for any $t\leq R_n$ it holds
  $$d(w_n(t),\partial B(x_n,d(x_n,z_n))\cap D(x_n,y_n))=t,$$
  from which one easily deduces, using $d(x_n,z_n)-C\leq d(x_n,w_n(t))\leq d(x_n,z_n)$, that
  $$d(w_n(t), D(x_n,y_n))\geq \frac t2 - C,$$
  whence a contradiction.
\end{proof}

\begin{lemma}\label{lem-Tits-distance-bounded}
  If a Riemannian geodesic ray $(w(t))_{t\geq 0}$ is contained in a Riemannian horoball centered at $\xi\in\partial_\infty\X$, then the Tits distance between $\xi$ and the limit of $(w(t))_t$ inside $\partial_\infty\X$ is at most $\tfrac\pi2$.
\end{lemma}

\begin{proof}
  Denote by $b_\xi$ the Riemannian Busemann function. 
  The derivative at time $t$ of $b_\xi(w(t),\basepoint)$ is $-\cos$ of the Riemannian angle at $w(t)$ between the ray $w$ and the ray from $w(t)$ to $p$ (see for instance~\cite[\S3.1]{KL18}).
  When $t\to\infty$, this angle converges to the Tits distance between 
  the endpoint $w(\infty)\in \partial_\infty \X$ of the ray and $\xi$ (see for instance~\cite[Part II, Prop.\,9.8]{bridson99}).
  If we want $b_\xi(w(t),\basepoint)$ to remain bounded from above, its derivative cannot converge to a positive number.
  As a consequence, the Tits distance between $w(\infty)$ and $\xi$ is at most $\tfrac\pi2$.
\end{proof}

Denote by $d^{\mathrm{Tits}}$ the Tits distance on $\partial_\infty\X$.

\begin{lemma}\label{lem-Tits-distance-pi}
  Let $\xi_1,\xi_2,\zeta\in\partial_\infty\X$ be three points whose Tits distance satisfy the following 
  $$d^{\mathrm{Tits}}(\xi_1,\zeta)+d^{\mathrm{Tits}}(\zeta,\xi_2)=d^{\mathrm{Tits}}(\xi_1,\xi_2)=\pi$$
  Suppose that $\xi_1$ (and equivalently $\xi_2$) is a regular point and let $F$ be the unique flat that contains $\xi_1$ and $\xi_2$ in its boundary at infinity $\partial_\infty F$.
   
  Then $\zeta$ belongs to $\partial_\infty F$.
\end{lemma}

\begin{proof}
  Assume that $\zeta$ is distinct from $\xi_1$ and $\xi_2$, otherwise it is trivial.
  The Tits metric is CAT(1) (see~\cite[Part II, Th.\,9.13]{bridson99}). It implies that each pair of points at distance $d<\pi$ are joined by a unique geodesic of length $d$.
   
  Let $\gamma_i:[0,\tfrac\pi2]\to\partial_\infty\X$ be the unique geodesic between $\xi_i$ and $\zeta$. 
  Since for the triple $\xi_1,\zeta,\xi_2$, equality holds in the triangle inequality, 
  the concatenation of $\gamma_1$ and $\gamma_2$ (traveled backward) is a geodesic $\gamma\colon[0,\pi]\to\partial_\infty\X$ from $\xi_1$ to $\xi_2$. Since $\xi_1$ is regular, it lies in the interior of a Weyl chamber of $\partial_\infty F$. The definition of Tits distance forces $\gamma(t)$ to remains in the same Weyl chamber for small values of $t>0$. In particular $\gamma(t)$ belongs to $\partial_\infty F$ for small $t$.
   
  The points $\gamma(t)$ and $\xi_2$ are at distance $\pi-t<\pi$, and so there is a unique geodesic from $\gamma(t)$ to $\xi_2$. The boundary at infinity $\partial_\infty F$ is totally geodesic for the Tits metric. So by uniqueness, the geodesic ($\gamma([t,\pi])$) remains in the boundary of $\partial_\infty F$. It follows that $\zeta=\gamma(d^{\mathrm{Tits}}(\xi_1,\zeta))$ lies in $\partial_\infty F$.
\end{proof}

\subsection{Proof of Theorem~\ref{thm:morse}}

We are now ready for the proof of 
Theorem~\ref{thm:morse}.

Consider a continuous path $c:[0,T]\to\X$ such that for any $0\leq t\leq s\leq u\leq T$, we have $d^{\mf F}(c(t),c(s))+d^{\mf F}(c(s),c(u))\leq d^{\mf F}(c(t),c(u))+C$.

Let $F$ be a flat containing $c(0)$ and $c(T)$, so that it also contains the diamond $D(c(0),c(T))$.
By Proposition~\ref{prop:coarse diamond} there exists $C'>0$ only depending on $C$ such that for any $0\leq t\leq T$,  
there exists $a(t)\in D(c(0),c(T))$ at distance at most $C'$ from $c(t)$, with $a(0)=c(0)$ and $a(T)=c(T)$.
By the triangle inequality, the path $t\to a(t)$ is $C+6C^\prime$ quasi-ruled. By Lemma~\ref{fact:qruled vs qgeod}, up to enlarging $C+6C^\prime$
to a constant which also only depends on $C$, we may assume that $t\to a(t)$ is continuous. 

Thus we 
can apply Proposition~\ref{prop:finding a geod at bdd dist} to $a$, to find a geodesic $b:[0,T]\to F$ such that $d^{\mf F}(a(t),b(t))\leq C''$ for some $C''>0$ which depends only on $C$. Recall that geodesics for the restricted metric on $F$ are also geodesics for the metric on $\X$.

We conclude that $d^{\mf F}(c(t),b(t))\leq C''+C'$ for any $t$ where $C''+C'$ depends on $C$ but not on the path $c$.

\section{Fock--Goncharov positivity}\label{sec:fockgoncharov}

This section is devoted to a geometric interpretation of positivity as introduced
by Lusztig~\cite{lusztig} and imported into the context of Hitchin representations by Fock and Goncharov
\cite{FG06}. We collect the relevant algebraic results and relate them
to admissible paths on the characteristic surface of a Hitchin grafting representation.
Throughout this section, we put $G=\SL_d(\R)$ although most of the discussion 
is valid for all split real simple Lie groups and although ultimately we are interested in 
$\PSL_d(\R)$. 
For completeness, note that for $G=\SL_d(\R)$,
most of Lusztig's results and concepts 
were already known (see for instance~\cite{ando}), but we still use Lusztig's notation and formalism.
    In particular, we will use Lusztig's work to introduce the subsets 
\[G_{>0}\subset G_{\geq 0}\subset G \text{  and }\mc F_{>0}\subset \mc F_{\geq0}\subset\mc F\] 
and some of their basic properties.

As $G=\SL_d(\R)$, 
the subset $G_{>0}\subset G$ is the set of 
\emph{totally positive matrices}, which are the matrices $A\in \SL_d(\R)$ such that for any $1\leq k\leq d-1$ the exterior product $\bwedge^kA\in\SL(\bwedge^k\R)=\SL_{d_k}(\R)$ is positive, i.e.\ all its entries are positive.
For general split Lie groups the definition of $G_{>0}$ is more complicated, see Sections 2.2, 2.12, 5.10 and 8.8 of~\cite{lusztig}.


    One can check that $\tau:\PSL_2(\R)\to \PSL_d(\R)$ is positive in the sense that it maps 
    projectivizations of positive matrices to projectivizations of totally positive matrices.
    Note also that Lusztig did not introduce the concept of positive representation from $\PSL_2(\R)$ into $G$: this is due to Fock--Goncharov~\cite{FG06}, who used it to prove among other things that all Hitchin representations are discrete.
    Many ideas in this section are inspired by the work of Fock--Goncharov.

    Finally, the concept of positivity also exists in certain \emph{nonsplit} real Lie group: see~\cite{BH12,GW18}.
    However in these settings the positive cone lies in a partial flag manifold instead of the full flag manifold, which is not enough for our purposes.

\subsection{Reminders on positivity}

The following summarizes the results from~\cite{lusztig} we are going to use. As before, ${\mc F}$ denotes the variety
of full flags in $\R^d$. 

    \begin{theorem}[\cite{lusztig}]\label{fact:positivity}
        There exist semigroups $G_{>0}\subset G_{\geq0}\subset G$ and a subset $\mc F_{> 0}\subset \mc F$ with the following properties.
        \begin{enumerate}
            \item \label{itempos:exp a in pos} {[By definition, see \S 2.2]} $\exp({\mathfrak a})\subset G_{\geq 0}$, in particular $1\in G_{\geq 0}$.
            \item \label{itempos:pos compo} {[Th.\,4.8]} For the standard embedding 
            $\tau:\SL_2(\R)\to \SL_d(\R)$, the set 
            $G_{>0}$ is an (open) connected component of 
            $$\{g\in G:g\partial \tau(\infty)\pitchfork \partial \tau(\infty)\, \text{ and } 
            \, g\partial \tau(0)\pitchfork \partial \tau(0)\}.$$
            \item \label{fact:poslox} [Th.\,5.6] Every element of $G_{>0}$ is loxodromic, and in particular does not fix any point of $\X$.
            \item {[Th.\,4.3 \& Rem.\,4.4]} $G_{\geq0}$ is the closure of $G_{>0}$.
            \item \label{nestedcon} {[before Prop.\,2.13]} $G_{>0}G_{\geq 0}\subset G_{>0}$ and $G_{\geq 0}G_{> 0}\subset G_{>0}$.
            \item \label{itempos:cone compo} {[Prop.\,8.14]} $\mc F_{> 0}$ is an (open) connected component of $\partial \tau(\infty)^\pitchfork \cap \partial \tau(0)^\pitchfork$.
            \item \label{itempos:cone preserved} {[Prop.\,8.12]} $G_{\geq0}\mc F_{> 0}\subset\mc F_{> 0}$.
            \item \label{itempos:cone is pos orbit} {[By \eqref{itempos:cone preserved} above and by definition, see Th.\,8.7]} $\mc F_{>0}=G_{>0}\cdot \partial \tau(\infty)$.
            \end{enumerate}
    \end{theorem}

\begin{example}    
    In the special case $G=\SL_3(\R)$, one can visualise $\mc F_{>0}$: namely, recall 
    that in this case, the flag variety $\mc F$ is identified with the set of pairs $(p,\ell)$ where $p$ is a point of~$\R\PP^2$ and $\ell$ is a line containing $p$.
    
    Let $x$, $y$ and $z\in\R\PP^2$ be the image of the canonical basis of $\R^3$, let $[x,y]$, $[y,z]$ and $[z,x]\subset\R\PP^2$ be the image of the segments between the vectors of the canonical basis, and let $T\subset \R\PP^2$ be the triangle enclosed by these segments.
    Then $\mc F_{>0}$ is the set of pairs $(p,\ell)$ such that $p$ is 
    contained in the interior of $T$ and $\ell$ intersects the (relative) interior of the segments $[x,y]$ and $[y,z]$.
    Then 
    \[G_{>0}=\{g\in \SL_3(\R)\mid g\overline{\mc F_{>0}}\subset\mc F_{>0}\}.\]
    One can see that $g\in G_{>0}$ 
    maps $\overline T$ into its interior (this corresponds to the fact that the entries of $g$, i.e.\ the minors of size $1$, are positive), and hence $g$ has an attracting fixed point in $T$.
    In general it is true that any totally positive matrix is diagonalisable with distinct positive eigenvalues.
\end{example}

We will also use the following, which should be well-known to the experts, but for which we did not find a reference.

    \begin{lemma}\label{factfact}
    \begin{enumerate}\setcounter{enumi}{8}
        \item \label{itempos:interior} $\mc F_{>0}$ is the interior of its closure, denoted by $\mc F_{\geq 0}$.
        \item \label{itempos:contracts} $G_{>0}\mc F_{\geq 0}\subset \mc F_{>0}$.
        \item \label{itempos:transverse}  Let us denote $G_{<0}:=(G_{>0})^{-1}$ and 
        $\mc F_{<0}=G_{<0}\cdot \partial \tau(\infty)$, and $G_{\leq 0}$ and $\mc F_{\leq 0}$ their respective closures.
        Then any pair in $\mc F_{<0}\times\mc F_{\geq 0}$ is transverse.
    \end{enumerate}
    \end{lemma}

\begin{remark}
    It is clear that $\mc F_{\geq 0}$ contains $G_{\geq 0}\cdot \de\tau(\infty)$ but they are not equal in general. 
    For instance consider the case $\SL_3(\R)$. 
    Denote the usual basis of $\R^3$ by $(e_1,e_2,e_3)$.
    Consider the flags $\de\tau(\infty)=(span(e_3),span(e_2,e_3))$
    and 
    $F=(span(e_2),span(e_2, e_3-e_1))$.
    
    Then $F$
    lies in $\mc F_{\geq 0}$ but not in $G_{\geq 0}\cdot \de\tau(\infty)$.
    Indeed one can check $F\in\mc F_{\geq 0}$ by computing that $F=\lim_{\lambda\to0^+}A_\lambda\de\tau(\infty)$
    where \(A_\lambda=\begin{psmallmatrix}
        1 & 3\lambda & \lambda \\
        \lambda & 2 & 1 \\
        \lambda^3 & \lambda & \lambda
    \end{psmallmatrix}\)
    is totally positive for small positive values of $\lambda$.
    (One can renormalize $A_\lambda$ to make it determinant $1$.)

    But if a matrix $B\in G_{\geq 0}$ sends $\de\tau(\infty)$ to the flag $F$, then it sends $e_2$ to a point $B\cdot e_2$ in $ span(e_2, e_3-e_1)\setminus span(e_2)$. Thus one can write $B\cdot e_2=a e_2 +b(e_3-e_1)$ with $b\neq 0$. Then one of the coefficients of $B$, in position $(3,1)$ or $(3,3)$, has a negative entry. 
    It contradicts the fact that $B\in G_{\geq 0}$, and that matrices in $G_{\geq 0}$ have non-negative entries. 
    The flag $F$ may be thought as corresponding to a point at infinity of $G_{\geq 0}$ (for the compactification $\SL_3(\R)\hookrightarrow\PGL_3(\R)$).
\end{remark}

    \begin{proof}[Proof of Lemma~\ref{factfact}]
        \emph{Proof of \eqref{itempos:interior}.} 
     Put $\xi_\infty=\partial \tau(\infty)$ and $\xi_0=\partial \tau(0)$. By definition, the sets
        $\xi_\infty^\pitchfork$ and $\xi_0^\pitchfork$ are open Bruhat cells in 
        $\mc F$ and hence $Z=\xi_{\infty}^\pitchfork \cap \xi_0^\pitchfork$ is open. 
        The complements of the Bruhat cells $\xi_\infty^\pitchfork$ and $\xi_0^\pitchfork$
        are (real) connected projective varieties all of whose
        irreducible components are of codimension one. 
        Then every point of the boundary of $\xi_\infty^\pitchfork \cap \xi_0^\pitchfork$ is 
        contained in an irreducible subvariety of codimension one. Thus 
        by property \eqref{itempos:cone compo}, the statement\eqref{itempos:interior}
        is equivalent to the following. Let $V$ be a component of $Z$. Denoting by 
        $\overline{U}$ the closure of a set $U$, it holds
        \begin{equation}\label{takeclosure}V={\mc F}-\overline{{\mc F}-\overline{V}}.\end{equation}

        As for any two open sets $U_1,U_2$ we have $\overline{U_1\cup U_2}=\overline{U_1}\cup \overline{U_2}$, 
        it suffices to write $V$ as a finite intersection $V=\cap_{j=1}^{d-1}U_j$ 
        where each of the sets $U_j$ has property (\ref{takeclosure}).
        To construct such sets note that 
        if $\xi_\iota=(\xi_\iota^1\subset \cdots \subset \xi_\iota^{d-1})$ $(\iota=\infty,0)$ 
        then the linear hyperplanes $\xi_\infty^{d-1},\xi_0^{d-1}$ are transverse and hence
        they decomposes $\mathbb{R}^d$ into four connected components which are paired be the 
        reflection $x\to -x$, say the components $A,-A,B,-B$. The closures of the components
        $A,-A$ and $B,-B$ intersect in a linear subspace of codimension~$2$.
        A component $V$ of 
        $Z=\xi_\infty^\pitchfork \cap \xi_0^\pitchfork$ consists of flags 
        $\zeta=\zeta_1\subset \cdots \subset \zeta_{d-1}$ with the property that $\zeta_1$ is transverse
        to both $\xi_\infty^{d-1},\xi_0^{d-1}$. But this means that for either $A$ or $B$, say for $A$, any nonzero 
        point on the line $\zeta_1$ is contained in $A\cup -A$. As a consequence, up to exchanging 
        $A$ and $B$, if we define 
        $U_1$ to be the set of all flags $\zeta=\zeta_1\subset \cdots \subset \zeta_{d-1}$ such that
        $\zeta_1-\{0\}\subset A\cup -A$ then $V\subset U_1$, and $U_1$ is an open set with property
        (\ref{takeclosure}). 

        Now note that $U_1$ can also be described as follows. Choose a generator $\omega_\iota$ of 
        $\wedge^{d-1} \xi^{d-1}_\iota$ $(\iota =0,\infty$) and define $U_1$ to be the set of all flags
        $\zeta$ with the property that for some basis element $e$ of $\zeta_1$, the wedge products 
        $e\wedge \omega_0,e\wedge \omega_\infty$ define the same (or the opposite) orientation of 
        $\R^d$. Then $U_1$ is one of the sets described in the previous paragraph.

        For $j\leq d-1$ define the set $U_j$ as the set of all flags $\zeta$ so that for some generator
        $e$ of $\bwedge^j \zeta_j$ and some generators $\omega_0^{d-j},\omega_\infty^{d-j}$ of 
        $\bwedge^{d-j}\xi_0^{d-j},\bwedge^{d-j}\xi_\infty^j$ the orientations defined by 
        $e\wedge \omega_0^{d-j},e\wedge \omega_\infty^{d-j}$ coincide (or are opposite). For suitably choices of the
        sets $U_j$, we then have $V=\cap_j U_j$. Together with the
        first paragraph of this proof, \eqref{itempos:interior} follows.

        \emph{Proof of \eqref{itempos:contracts}.}
        Since $G_{>0}$ is open (by \eqref{itempos:pos compo}), every $G_{>0}$-orbit in $\mc F$ is open.
        Being a union of such orbits, $G_{>0}\mc F_{\geq 0}$ is also open.
        Moreover it is contained in $\mc F_{\geq 0}$ (by \eqref{itempos:cone preserved}) and hence it is contained in its interior, which is precisely $\mc F_{>0}$ by \eqref{itempos:interior}.

        \emph{Proof of \eqref{itempos:transverse}.}
        Consider $(\xi,\eta)\in\mc F_{<0}\times \mc F_{\geq 0}$.
        Since $\mc F_{>0}$ is open, $G_{>0}$ contains $1$ in its closure, and $G_{>0}\mc F_{\geq 0}\subset \mc F_{>0}$, there exists $g\in G_{>0}$ such that $g\xi\in \mc F_{<0}$ and $g\eta\in \mc F_{>0}$.

        By definition, there exists $h\in G_{>0}$ such that $g\xi=h^{-1}\xi_\infty$.
        Then $hg\xi=\xi_\infty$ and $hg\eta\in h\mc F_{>0}\subset\mc F_{>0}\subset \xi_\infty^\pitchfork$ by \eqref{itempos:cone compo}.
        Therefore $hg\xi$ and $hg\eta$ are transverse, and so are $\xi$ and~$\eta$.
    \end{proof}

\subsection{Positivity and injectivity of admissible paths}

    We now explain the assumption that the representation $\tau:\SL_2(\R)\to \SL_d(\R)$ is 
    positive: it means that $\tau$ maps every $2\times 2$ matrix with positive entries into $G_{>0}$.
    It has the following consequences, which should be well-known to experts.

    \begin{lemma}\label{fact:positivity tau}
    We have the following.
        \begin{enumerate}
            \item \label{itempos:a't} $a'_t\in G_{>0}$ for any $t>0$.
            \item \label{itempos:rpi conj} $G_{<0}=r_\pi G_{>0} r_\pi$ (see Notation~\ref{nota:a't}).
        \end{enumerate}
    \end{lemma}
    \begin{proof}
        \emph{Proof of \eqref{itempos:a't}.}
        This is an immediate consequence of the fact that 
        $$\begin{pmatrix}
            \frac{\sqrt{2}}2 & \frac{\sqrt{2}}2 \\
            -\frac{\sqrt{2}}2 & \frac{\sqrt{2}}2
        \end{pmatrix}
        \begin{pmatrix}
            e^t & 0 \\
            0 & e^{-t}
        \end{pmatrix}
        \begin{pmatrix}
            \frac{\sqrt{2}}2 & -\frac{\sqrt{2}}2 \\
            \frac{\sqrt{2}}2 & \frac{\sqrt{2}}2
        \end{pmatrix}$$
        has positive entries.

        \emph{Proof of \eqref{itempos:rpi conj}.}
        Let us prove that $r_\pi G_{>0}^{-1} r_\pi=G_{>0}$.
        The maps $g\mapsto g^{-1}$ and $g\mapsto r_\pi gr_\pi$ both preserve 
        $$\{g\in G:g\partial \tau(\infty)\pitchfork \partial \tau(\infty)\ \text{ and } 
        \ g\partial \tau(0)\pitchfork \partial \tau(0)\},$$
        and hence permute the connected components, and consequently so does $g\mapsto r_\pi g^{-1}r_\pi$.
        To prove that this map preserves the connected component $G_{>0}$ (see Theorem~\ref{fact:positivity} \eqref{itempos:pos compo}), it suffices to show that it fixes a point of $G_{>0}$.
        It is clear that it fixes for instance $a'_1\in G_{>0}$.
    \end{proof}
    
     
     
    

    A first important consequence of all the facts about positivity that we have listed is the following.

    \begin{corollary}\label{cor:admissible paths are positive}
        For any admissible path $c:[0,T]\to G$, for all $0\leq s<t\leq T$ we have $c(s)^{-1}c(t)\in G_{\geq 0}$.
    \end{corollary}

    \begin{proof}
        By definition, $c(s)^{-1}c(t)$ is a product of elements of $G$ of the form $a'_r$ for some $r>0$ or 
        $\exp(v) $ for some $v\in\mathfrak a$.
        All these elements belong to $G_{\geq 0}$ by Theorem~\ref{fact:positivity} and 
        Lemma~\ref{fact:positivity tau}, and so does their product since $G_{\geq0}$ is a semigroup.
    \end{proof}

The previous result, combined with the fact that totally positive matrices are not the identity, tells us that admissible paths are injective, as explained below.
    
    \begin{proposition}\label{lem:admissible injective}
        Any admissible path in $\X$ is injective.
    \end{proposition}
    \begin{proof}
        Let $c:[0,T]\to G$ be an admissible path and let $\basepoint \in \X$ be an arbitrarily fixed point.
        Consider $0\leq s<t\leq T$, and let us prove that $c(s)\basepoint \neq c(t)\basepoint$, \ie that $c(s)^{-1}c(t)\basepoint\neq\basepoint$.
        By definition, $c(s)^{-1}c(t)$ is a product of elements of $G$ of the form $a'_r$ for some $r>0$ 
        or $\exp(v)$ for some $v\in\mathfrak a$.
        There are two cases.
    
        Case 1: $c(s)^{-1}c(t)=\exp(v)$ for some nonzero $v\in\mathfrak a$, then it is clear that 
        $\exp(v)\basepoint\neq \basepoint$.
    
        Case 2: $c(s)^{-1}c(t)$ is a product of elements of $G$ of the form $a'_r$ for some $r>0$ or 
        $\exp(v)$ for some $v\in\mathfrak a$, with at least one element of the form $a'_r$.
        All the $\exp(v)$'s belong to~$G_{\geq 0}$ by Theorem~\ref{fact:positivity}.\ref{itempos:exp a in pos}, 
        and all the $a'_r$'s belong to $G_{>0}$ by Lemma~\ref{fact:positivity tau}.\ref{itempos:a't}, and so $c(s)^{-1}c(t)$ belongs to $G_{>0}$ by Theorem~\ref{fact:positivity}.\ref{nestedcon}.
        Therefore $c(s)^{-1}c(t)$ does not fix any point of $\X$ by Theorem~\ref{fact:positivity}.\ref{fact:poslox}, and $c(s)^{-1}c(t)\basepoint\neq\basepoint$.
    \end{proof}
    
    This implies that the characteristic surface we have constructed in Section~\ref{sec-HG} is embedded.
    
    \begin{corollary}\label{prop:embedding}
        The map $Q_z:S_z\to \rho_z\backslash \X$ constructed in Proposition~\ref{piecewise-isometry} is injective.
    \end{corollary}
    \begin{proof}
        Consider $x,y\in S_z$ such that $Q_z(x)=Q_z(y)$, and let us prove that $x=y$.
    
        Consider two lifts $\tilde x$ and $\tilde y\in\tilde S_z$ of respectively $x$ and $y$.
        Then $\tilde Q_z(\tilde x)$ and $\tilde Q_z(\tilde y)$ have the same projection in 
        $\rho_z\backslash \X$, which means that there exists $\gamma\in\pi_1(S)$ such that 
        $$\tilde Q_z(\tilde x)=\rho_z(\gamma)\tilde Q_z(\tilde y)=\tilde Q_z(\gamma\tilde y).$$
    
        Consider an admissible path $c:[0,T]\to\tilde S_z$ from $\tilde x$ to $\gamma\tilde y$.
        Then by Observation~\ref{obs:admissiblemapstoadmissible} $\tilde Q_z\circ c:[0,T]\to\X$ is an admissible path from $\tilde Q_z(\tilde x)$ to itself.
    
        Since admissible paths of $\X$ are injective by Proposition~\ref{lem:admissible injective}, this means that $T=0$ and $\tilde x=\gamma\tilde y$, and hence that $x=y$.
    \end{proof}

\subsection{Positivity gives a control on default of the triangle inequality}\label{sec:pos implies triangle eq}

The goal of this section is to prove the following result about totally positive transformations.
We then use it to prove that admissible paths are (Finsler) quasi-ruled, quasi-convex and quasi-geodesics, which implies that the characteristic surface in the symmetric space associated to a Hitchin grafting representation is Finsler quasi-convex.

As before, we write $G=\SL_d(\R)$, and we choose a basepoint $\basepoint \in \X=G/K$, thought
of as the projection of the identity in $G$.

\begin{lemma}\label{lem:triangle equality pos}
    For any $\omega>0$,
    there exists $C_\omega>0$ such that
    for all $g_+\in G_{\geq 0}$ and $g_-\in a'_{-\omega}G_{\leq 0}$, we have
    $$d^{\mf F}(g_-\basepoint,\basepoint) + d^{\mf F}(\basepoint,g_+\basepoint) 
    \leq d^{\mf F}(g_-\basepoint,g_+\basepoint) + C_\omega.$$
\end{lemma}

To prove it we will need the following technical result. 
In its formulation, we use a $K$-invariant metric $d_{\mc F}$ on the flag variety ${\mc F}$.
Furthermore, we denote by $\xi_0\in {\mc F}$ the simplex $\partial \tau(\infty)=\exp({\mf a}^+)\in {\mc F}$.
Distances are taken with respect to the distance function $d$ defined by the symmetric Riemannian metric.

\begin{lemma}\label{lem:dist to weylcone}
    For every $\epsilon >0$ there exists a number $C_\epsilon>0$ only depending on $\epsilon$ with the 
    following property. 
    Consider $g\in G$ decomposed as $g=k\exp(\kappa(g))\ell$ with $k,\ell\in K$ 
    (the maximal compact subgroup) and $\kappa(g)\in\mf a^+$.
    Let $\xi\in\mc F$ be at $d_{\mc F}$-distance at least $\epsilon>0$ from $k\xi_0^{\not\pitchfork}$.
    Then $\basepoint$ is at distance at most $C_\epsilon$ from the Weyl cone with vertex at
    $g\basepoint$ and boundary at infinity the simplex in $\partial_\infty \X$ corresponding to 
    $\xi$.
\end{lemma}
\begin{proof}
    %
    Since $d_{\mc F}(\xi,k\xi_0^{\not \pitchfork})\geq \epsilon$, the simplices $\xi,k\xi_0$ are transverse and  
    hence they are contained in a unique maximal flat $F$. We claim that 
    there is $x\in F$ at distance at most $C_\epsilon>0$ from $\basepoint$, where $C_\epsilon$ only depends on $\epsilon$. 
    Indeed, this follows from the compactness of the set $\{(\xi,\eta)\in\mc F^2,d_{\mc F}(\xi,\eta)\geq \epsilon\}$ 
    and continuity of the map which associates to two transverse flags the unique maximal flat 
    whose visual boundary contains the Weyl chambers that corresponds to the two flags.

    Note that $g\basepoint$ is contained in the Weyl cone connecting 
    $\basepoint$ to $k\xi_0$ (because $k\basepoint=\basepoint, \ell \basepoint=\basepoint$ and 
    $\exp(\kappa(g))\basepoint$ is contained in the Weyl Cone 
    connecting $\basepoint$ to $\xi_0$).
    Thus as Weyl cones are convex cones, 
    the endpoint $p\in\partial_\infty\X$ of the geodesic ray 
    $[\basepoint,p)$ starting at $\basepoint$ 
    and passing through $g\basepoint$ is contained in the Weyl Chamber $k\xi_0$.
    The ray $[x,p)$ is contained in the Weyl Cone from~$x $ to $k\xi_0$.
    
    We now apply the ${\rm CAT}(0)$-property for the Riemannian symmetric metric 
    to the asymptotic rays $[x,p)$ and $[\basepoint,p)$. It yields that the point $y\in [x,p)$ 
    at distance exactly $d(\basepoint,g \basepoint)$ from $x$ is of distance 
    at most  $d(\basepoint,x)\leq C_\epsilon$ from $g\basepoint$.

    By construction, the geodesic ray $[x,p)$ is contained in the flat $F$, and its endpoint is contained in the 
    Weyl chamber $k\xi_0$. The unique geodesic line $\eta$ extending $[x,p)$ is contained in $F$ and 
    is backward asymptotic to a point $q$ in the unique Weyl chamber $\xi$ in the visual boundary of $F$ which 
    is transverse to $k\xi_0$. Recall that $\eta$ passes through the $C_\epsilon$-neighborhood of 
    $\basepoint$. 
    
    Using once more the ${\rm CAT}(0)$-property, this time applied to the subray of 
    $\eta$ which connects $y$ to $q$ and the geodesic ray $\zeta$ connecting $g\basepoint$ to $q$, we conclude
    that $\zeta$ passes through the $C_\epsilon$-neighborhood of $x$ and hence through the 
    $2C_\epsilon$-neighborhood of $\basepoint$. On the other hand, by construction, this ray 
    is contained in the Weyl cone connecting $g\basepoint$ to $\xi$. Together this is what we
    wanted to show.
\end{proof}

\begin{proof}[Proof of Lemma~\ref{lem:triangle equality pos}]
  
Decompose $g_\pm=k_\pm e^{\kappa(g_\pm )}\ell_\pm $ with $k_\pm,\ell_\pm\in K$ (the maximal compact subgroup) and $\kappa(g_\pm)\in \mf a^+$.
    The plan is to use positivity and Lemma~\ref{lem:dist to weylcone} to find Weyl Chambers $\xi_\pm$ such that their images $g_\pm\xi_\pm$ are transverse, the flat $F$ through them passes near $\basepoint$, and the Weyl Cone from 
    $\basepoint$ to $g_\pm\xi_\pm$ passes near $g_\pm\basepoint$.
    
    Recall the definition of the set ${\mc F}_{\geq 0}$, and for $\omega >0$ the element $a_\omega^\prime$. 
    Since $a'_\omega \mc F_{\geq 0}$ has nonempty interior and $\xi_0^{\not \pitchfork}$ is a closed set with empty interior, there exists $\epsilon=\epsilon_\omega>0$ such that for every $k\in K$ there is $\xi\in a'_\omega \mc F_{\geq 0}$ at $d_{\mc F}$-distance at least $\epsilon$ from $k\xi_0^{\not\pitchfork}$.
    Similarly, for every $k\in K$ there is $\xi\in a'_{-\omega} \mc F_{\leq 0}$ at $d_{\mc F}$-distance at least $\epsilon$ from $k\xi_0^{\not\pitchfork}$.

    Let $\xi_+\in a'_\omega \mc F_{\geq 0}$ be at $d_{\mc F}$-distance at least $\epsilon$ from 
    $\ell_+^{-1}\xi_0^{\not\pitchfork}$ and $\xi_-\in a'_{-\omega} \mc F_{\leq 0}$ be at $d_{\mc F}$-distance 
    at least $\epsilon$ from $\ell_-^{-1}\xi_0^{\not\pitchfork}$.

    We know $\xi_+\in\mc F_{\geq 0}$ and $g_+\in G_{\geq 0}$, so by Theorem~\ref{fact:positivity}.\ref{itempos:cone preserved} we have $g_+\xi_+\in\mc F_{\geq 0}$, and similarly  $g_-\xi_-\in a'_{-\omega}\mc F_{\leq 0}$.
    In particular, by Lemma~\ref{factfact}.\ref{itempos:transverse} $g_+\xi_+$ and $g_-\xi_-$ are transverse.
    More precisely, if we denote as before by $d$ the distance function of the symmetric metric, then 
    the flat $F=F(g_-\xi_-,g_+\xi_+)$ through them contains a point $x$ with $d(x,\basepoint)\leq q_\omega$
    for some $q_\omega>0$ only depending on $\omega$, because every pair in the set $a'_{-\omega}\mc F_{\leq 0}\times \mc F_{\geq 0}$ is transverse and this set is compact.

    Since $d_{\mc F}(\xi_\pm,\ell_\pm^{-1}\xi_0^{\not\pitchfork})\geq\epsilon$, Lemma~\ref{lem:dist to weylcone} 
    implies that $g_\pm\basepoint$ is at Riemannian distance at most $C_\epsilon$ to the Weyl Cone connecting
    $\basepoint$ to $g_\pm\xi_\pm$.  
    By the CAT(0) property of $(\X,d)$, applied to the geodesics connecting $\basepoint$ and $x$ to all points in 
    $g_{\pm}\xi_{\pm}$, the Hausdorff distance (for $d$) between 
    this Weyl cone and the Weyl cone $W_\pm\subset F$ connecting $x$ to $g_\pm\xi_\pm$ is 
    at most $d(\basepoint,x)\leq q_\omega$. As a consequence, there is $x_\pm\in W_\pm$ 
    with 
    \[d(x_{\pm},g_{\pm} \basepoint)\leq 
    q_\omega+C_\epsilon.\]

    Since $W_+$ and $W_-$ are two opposite Weyl Cones in $F$ based at $x$, and since $x_\pm\in W_\pm$, we deduce from Proposition~\ref{prop:diamond} that 
    $$d^{\mf F}(x_-,x)+d^{\mf F}(x,x_+)=d^{\mf F}(x_-,x_+).$$
    
    To conclude, recall that the Riemannian and Finsler metrics are comparable; that is,
    there exists $\lambda>0$ such that $\lambda^{-1}d\leq d^{\mf F}\leq \lambda d$.
    Then by the triangle inequality
    \begin{align*}
        d^{\mf F}(g_-\basepoint,\basepoint)& +d^{\mf F}(\basepoint,g_+\basepoint)-
        d^{\mf F}(g_-\basepoint,g_+\basepoint)\\
        &\leq d^{\mf F}(x_-,x)+d^{\mf F}(x,x_+)-d^{\mf F}(x_-,x_+) + 6\lambda(q_\omega+C_\epsilon) \\
        &\leq 6\lambda(q_\omega+C_\epsilon).\qedhere
    \end{align*}
   \end{proof}

We now use the previous result to prove that admissible paths are quasi-ruled.
For this we will need the following general fact about quasi-ruled paths.

\begin{lemma}\label{lem:quasiruled for concats}
 Let $(X,d)$ be a metric space and $x_1,\dots,x_n\in X$ be such that for some constant $C>0$ we have $d(x_i,x_j)+d(x_j,x_k)\leq d(x_i,x_k)+C$ for all $i<j<k$.
Then any concatenation of geodesics $[x_1,x_2], [x_2,x_3],\dots,[x_{n-1},x_n]$ is $4C$-quasi-ruled. 
\end{lemma}
\begin{proof}
 Let $x,y,z$ be three points on such a concatenation, in this order, and let us check that $d(x,y)+d(y,z)\leq d(x,z)+4C$.
 Let $[x_i,x_{i+1}]$, $[x_j,x_{j+1}]$ and $[x_k,x_{k+1}]$ be three geodesic pieces of the concatenation containing respectively $x,y,z$, such that $i\leq j\leq k$.
 Let us assume for the rest of the proof that $i<j<k$; if instead we have $i=j<k$ or $i<j=k$ then the proof is similar (in fact easier), and the case $i=j=k$ is obvious.

 Using the triangle inequality and our assumption, and denoting $(a,b)=d(a,b)$ to lighten our estimates on distances, we have the following, which concludes the proof.
\begin{align*}
(x&,y)+(y,z) 
\leq (x,x_{i+1})+(x_{i+1},x_j)+(x_j,y)+(y,x_{j+1}) + (x_{j+1},x_k)+(x_k,z)\\
&\leq -(x_i,x)+(x_i,x_{i+1})+(x_{i+1},x_j)+(x_j,x_{j+1}) + (x_{j+1},x_k)+(x_k,x_{k+1})-(z,x_{k+1})\\
&\leq -(x_i,x)+(x_i,x_j)+C+(x_j,x_k)+C+(x_k,x_{k+1})-(z,x_{k+1})\\
&\leq 2C -(x_i,x)+(x_i,x_k)+C+(x_k,x_{k+1})-(z,x_{k+1})\\
&\leq 3C -(x_i,x)+(x_i,x_{k+1})+C-(z,x_{k+1})\\
&\leq 4C +(x,z).\qedhere
\end{align*}
\end{proof}

We now prove that admissible paths are quasi-ruled.

\begin{proposition}\label{prop:triangle equality for admissible paths}
    For any $\omega>0$ there exists $C_\omega$ such that  any $(\omega,0)$-admissible path $c$ in $\X$ is Finsler $C_\omega$-quasi-ruled, and hence is at Hausdorff distance at most $C'_\omega$ from some Finsler geodesic by Theorem~\ref{thm:morse}, where $C'_\omega$ only depends on $C_\omega$.
\end{proposition}
\begin{proof}  
By definition $c(r)=a(r)\basepoint$ for any $r$, where $a(r)$ is an admissible path in $G$.

Take $0\leq t< s< u\leq T$, and let us show $d(c(t),c(s))+d(c(s),c(u))\leq d(c(t),c(u))+C_\omega$ for a well-chosen $C_\omega$.
By Lemma~\ref{lem:quasiruled for concats} above we may assume that each of the points $a(t),a(s),a(u)$ is at the junctions of two pieces of the admissible path $a$, one of hyperbolic type, and the other of flat type (see Definition~\ref{def:admissible in G}).

 By Corollary~\ref{cor:admissible paths are positive}, $a(s)^{-1}a(u)\in G_{\geq0}$ and $a(s)^{-1}a(t)\in G_{\leq 0}$.

Note that $a(s)$ is adjacent to a hyperbolic-type piece of $a$, which has length at least $\omega$, unless this piece is the first or last piece of $a$.
If this hyperbolic-type piece is first or last and has length less than $\omega$, then $c(s)$ is $\omega$-close to either $c(t)$ or $c(u)$, and one conclude easily with the triangle inequality (taking $C_\omega\geq \omega$).
Let us assume this hyperbolic-type piece has length at least $\omega$.

If this piece is after $a(s)$ then $a(s)^{-1}a(u)\in a'_{\omega} G_{\geq0}$.
If on the contrary this piece is before $a(s)$ then $a(s)^{-1}a(t)\in a'_{-\omega}G_{\leq 0}$. 

In any case, by Lemma~\ref{lem:triangle equality pos}, we can conclude:
    \begin{equation*}d^{\mf F}(c(t),c(s))+d^{\mf F}(c(s),c(u)) \leq d^{\mf F}(c(t),c(u)) + C_\omega.\qedhere\end{equation*}
\end{proof}

Finally, we deduce that the characteristic surface associated to a Hitchin grafting representation is quasi-convex.

\begin{corollary}\label{cor:quasiconvex}
    Consider the map $\wt Q_z:\wt S_z\to \wt S_z^\iota\subset \X$ constructed in Proposition~\ref{piecewise-isometry}, such that the grafting locus $\gamma^*\subset S$ has collar size at least $\omega>0$.
    Then $\wt S_z^\iota$ is Finsler $C_\omega$-quasi-convex for some $C_\omega$ depending on $\omega$, in the sense that any two points of $\wt S_z^\iota$ can be connected by a Finsler geodesic that stays at distance at most $C_\omega$ from $\wt S_z^\iota$.
\end{corollary}
\begin{proof}
    This is an immediate consequence of Proposition~\ref{prop:triangle equality for admissible paths} and Observation~\ref{obs:admissiblemapstoadmissible}.
\end{proof}

\subsection{Estimates on eigenvalues of products of totally positive matrices}

    We will also need a quantitative version of the classical result that all elements of $G_{>0}$ are loxodromic.
    This quantitative result is probably well known to experts; since we did not find a precise reference for what we need,
    we give a proof (in the case $G=\SL_d(\R)$).
    We will not need in the present paper the estimates on angles presented below, but they will be useful in the companion paper \cite{BHMM25}.

    Recall that for any matrix $g$ and $1\leq k\leq d$, we denote by $\lambda_k(g)$ the logarithm of the norm of the $k$-th eigenvalue of $g$, such that $\lambda_1(g)\geq \lambda_2(g)\geq\dots\geq\lambda_d(g)$.
    Recall that $g$ is loxodromic if we have strict inequalities; in this case this gives a natural ordering on the eigenspaces of $g$ (which are eigenlines).

\begin{proposition}\label{fact:quantitative Pos are Lox}
    For any $g\in G_{>0}$ there exists $\omega, \theta>0$ such that the following holds.
    \begin{itemize}
      \item for any $h\in G_{\geq 0}$, for any $1\leq k<d$, we have $\lambda_k(gh)\geq \lambda_{k+1}(gh)+\omega$ (so $gh$ is loxodromic) and the angle between the sum of the $k$ first eigenlines of $gh$ and the sum of remaining eigenlines is at least $\theta$.
      \item for all $h_1,\dots,h_n\in G_{\geq 0}$, denoting  $h=gh_1gh_2\cdots gh_n$ we have $\lambda_k(h)\geq \lambda_{k+1}(h) + n\omega$ for any $1\leq k \leq d-1$.

          In particular, $d^{\mf F}(\basepoint,h\basepoint)\geq n\omega'$ for some constant $\omega'$ that only depends on $\omega$.
      \item for all $h,h'\in gG_{\geq 0}$, for any $k$, the angle between the sum of the $k$ first eigenlines of $h$ and the sum of last $d-k$ eigenlines of $h'$ is at least $\theta$.
    \end{itemize}

\end{proposition} 

To prove Proposition~\ref{fact:quantitative Pos are Lox}, 
we first establish an intermediate result about positive matrices and use the fact that a matrix of size $d$ is totally positive if and only if all its exterior products, seen as matrices of size $d_k=\dim(\bwedge^k\R^d)$ where $1\leq k\leq d-1$, are \emph{positive}, i.e.\ with positive entries.

\begin{lemma}
    For any positive matrix $g$ of size $d$, there exists  $\omega>0$ such that the following holds.
    \begin{itemize}
        \item for any nonegative $h$, we have $\lambda_1(gh)\geq \lambda_{2}(gh)+\omega$ and the angle between the attracting eigenline of $gh$ and the repelling hyperplane is at least $\theta$.
        \item for all $A_1,\dots,A_n$ nonnegative, if $A=gA_1\cdots gA_n$ then $\lambda_1(A)\geq \lambda_{2}(A) + n\omega$.
        \item for all $h,h'$ nonnegative, for any $k$, the angle between the sum of the attracting eigenline of $gh$ and the repelling hyperplane of $gh'$ is at least $\theta$.
    \end{itemize}
    
\end{lemma}
\begin{proof}
    We will prove all three points at the same time.
    By density of positive matrices it suffices to prove the lemma for $A_1,\dots,A_n$ positive.

    Let $\mc C\subset\R^d$ be the open convex cone of positive vectors, and
    $\Omega\subset \R\PP^{d-1}$ its projectivisation, which is a properly convex domain in the sense that there is an affine chart of~$\R\PP^{d-1}$ containing $\Omega$ and in which $\Omega$ is bounded and convex.
    
    Note that $A\overline{\mc C}\subset {\mc C}\cup\{0\}$ for any positive matrix $A$, so $A\overline\Omega\subset \Omega$.
    
    Any properly convex domain $\Omega'\subset\R\PP^{d-1}$ can be endowed with a classical Finsler metric $d_{\Omega'}$ called the Hilbert metric, locally equivalent to the usual Riemannian metric of $\R\PP^{d-1}$, such that (see~\cite{Birkhoff_perronfrob}, or see~\cite{HilbertHandbook} for a broad introduction to Hilbert geometry)
    \begin{enumerate}
        \item it is projectively equivariant: $d_{h\Omega'}\circ h=d_{\Omega'}$ for any projective transformation $h$;
        \item it is monotone with respect to inclusion: $d_{\Omega'}\leq d_{\Omega''}$ (on $\Omega''$) for any $\Omega''\subset \Omega'$;
        \item if $\overline\Omega''\subset\Omega'$ then there is $r<1$ such that $d_{\Omega'}\leq r d_{\Omega''}$ (on $\Omega''$).
    \end{enumerate}

    Let $g$ be a positive matrix: As $g\overline{\Omega}\subset \Omega$ there is $r<1$ such that $d_{\Omega}\leq r d_{g\Omega}$.
    Then $g:\Omega\to\Omega$ is $r$-Lipschitz  for $d_\Omega$ --- hence it is a contraction --- by equivariance of the Hilbert metric.
    In fact, $gA:\Omega\to\Omega$ also is an $r$-$d_\Omega$-Lipschitz map for any positive matrix $A$ since $gA\Omega\subset g\Omega$.

    By the Banach fixed-point theorem, $gA$ has a fixed point $p\in\Omega$ such that $(gA)^nx\to p$ for any $x\in \Omega$.
    In particular, $gA$ is proximal with attracting line $p\in g\Omega$, and the repelling hyperplane does not intersect $\Omega$.
    
    This implies the existence of a positive lower bound $\theta$ on the angle (for the standard Euclidean metric) 
    between the attracting line of $gA$ and the repelling hyperplane of $gA'$ for all positive matrices $A,A'$.
    
    This settles our claims about angles in the statement of the lemma.
    It remains to prove the estimates on the gaps between the first two eigenvalues.
    This will be done by reinterpreting this gap as a contraction rate for the action of proximal transformations at their attracting eigenline and using our observation above that positive matrices contract the cone of positive vectors.

    One can observe that
    for any proximal matrix $h$ 
    such that the angle between the attracting line and the 
    repelling hyperplane is at least $\theta$, 
    the number $e^{\lambda_2(h)-\lambda_1(h)}$ is comparable to the contraction rate of $h$ at its attracting fixed point in $\R\PP^{d-1}$ for the usual Riemannian metric, where the comparison error 
    only depends on $\theta$.
    In other words, if $h$ is $R$-Lipschitz (for the Riemannian metric) at its attracting eigenline then
    $$
    e^{\lambda_2(h)-\lambda_1(h)}\leq C_1R
    $$
    for some constant $C_1$ depending on $\theta$.

    Moreove, as we already mentioned, the restriction of the Hilbert metric $d_\Omega$ to the compact subset 
    $\overline{g\Omega}$ is uniformly comparable to the standard Riemannian metric on $\R\PP^{d-1}$.
    Thus for  any  transformation $h:\Omega\to g\Omega$,  if $h$ is $R$-$d_\Omega$-Lipschitz on $g\Omega$ then it is $RC_2$-Riemannian-Lipschitz for some constant $C_2$ that depends on $g\Omega$ and $\Omega$.
    In particular it is $RC_2$-Lipschitz at its attracting eigenline, and hence
    $$
    e^{\lambda_2(h)-\lambda_1(h)}\leq CR
    $$
    where $C=C_1C_2$.

    In particular, for all $A_1,\dots, A_n$ positive, $A=gA_1\cdots gA_n$ is $r^n$-$d_\Omega$-Lipschitz so
    $$
    \lambda_1(A)-\lambda_2(A)\geq -\log(Cr^n).
    $$
    In fact, for any $k$, the transformation $A^k$ is $r^{nk}$-$d_\Omega$-Lipschitz so
    $$
     \lambda_1(A)-\lambda_2(A)=\frac1k (\lambda_1(A^k)-\lambda_2(A^k))\geq \frac{-\log C}k+n\log\left(\frac1r\right),
    $$
    and letting $k\to\infty$ we get
    $$
    \lambda_1(A)-\lambda_2(A)\geq n\log\left(\frac1r\right).
    $$
    Hence $\omega=\log\frac1r$ is the positive number we were looking for. 
\end{proof}

\begin{proof}[Proof of Proposition~\ref{fact:quantitative Pos are Lox}]
    This is an immediate consequence of the previous lemma, the fact that a matrix $g\in\GL_d(\R)$ is totally positive if and only if $\bwedge^kg\in\GL_{d_k}(\R)$ is positive for any $1\leq k\leq d$, and the fact that 
    \begin{equation*}
    \lambda_{k}(g)-\lambda_{k+1}(g)=\lambda_1(\bwedge^kg)-\lambda_2(\bwedge^kg).
    \end{equation*}
    We also used that the $\bwedge^{d-k}\R^d$ is naturally the dual to $\bwedge^k\R^d$, and that given a $k$-plane spanned by $v_1,\dots,v_k\in\R^d$ and a transverse $d-k$-plane spanned by $w_1,\dots,w_{d-k}$, the angle between these two subspaces is equal to the angle in $\bwedge^k\R^d$ between the vector $v_1\wedge\dots\wedge v_k$ and the hyperplane kernel of $w_1\wedge\dots\wedge w_{d-k}\in\bwedge^{d-k}\R^d$ seen as a linear form on $\bwedge^k\R^d$.

    That it holds $d^{\mf F}(\basepoint,h\basepoint)\geq \tfrac n{C'}$ comes from  the fact that 
    $d^{\mf F}(\basepoint,h\basepoint)$ is bounded from below by the Finsler translation length of $h$ acting on $\X$, which is given by $\alpha_0(\lambda_1(h),\dots,\lambda_d(h))$. Furthermore, the restriction of
    $\alpha_0$ to the set of diagonal matrices with ordered diagonal entries 
    $(v_1,\dots,v_d)$ such that $v_k\geq v_{k+1}$ for each $k$ is uniformly comparable to the maximums norm on the diagonal 
    entries. 
\end{proof}

Proposition~\ref{fact:quantitative Pos are Lox} has the following consequence in terms of admissible paths.

    \begin{proposition}\label{prop:inj quantique}
        For any $\omega>0$ there exists $C_\omega>0$ such that 
        for any $(\omega,0)$-admissible path $c:[0,T]\to G$,
        we have
        $$
        d^{\mf F}(c(0)\cdot \basepoint,c(T)\cdot\basepoint)\geq \frac{k-2}{C_\omega},
        $$
        where $k$ is the number of singularities (i.e.\ $k+1$ is the number of geodesic pieces of $c$).
    \end{proposition}
    Observe that we need the $-2$ term in $(k-2)/C_\omega$ because we allow the first and last pieces of $c$ to have length less than $\omega$.
    \begin{proof}
        Without loss of generality we can assume $k\geq 3$, so that $c$ contains at least one piece of hyperbolic type of length at least $\omega$, and that $c(0)=1$.
        Let $r$ be the number of hyperbolic pieces of $c$ of length at least $\omega$; note that 
        $$\frac{k-2}2\leq r\leq \frac{k+2}2.$$

        Then by definition of admissible (Definition~\ref{def:admissible in G}) and Theorem~\ref{fact:positivity}, we can write $c(T)\in G_{>0}$ as the following product:
        $$
        c(T)=g_0a'_\omega g_1 a'_\omega g_2\cdots g_{r-1} a'_\omega g_r,
        $$
        where $g_0,\dots,g_r\in G_{\geq 0}$.

        By Proposition~\ref{fact:quantitative Pos are Lox}, there is $C>0$ only depending on $\omega$ such that
        \begin{equation*}
        d^{\mf F}(c(0)\basepoint,c(T)\basepoint) \geq \frac rC \geq \frac {k-2}{2C}.\qedhere
        \end{equation*}
    \end{proof}

\section{Geometric control: Uniform quasi-isometry}\label{QI-proof}

    This section contains the main geometric results of this article. 
    Recall that $S$ is a hyperbolic closed surface, let
    $G=\PSL_d(\R)$
    and $\tau:\PSL_2(\R)\to G$ be the usual irreducible representation.
    In Section~\ref{sec:graphofgroups}, given a collection of disjoint closed curves on $S$ and an element of the Cartan subspace of $G$ for each of these curves, we have recalled the definition of bending $\tau(\pi_1(S))$ inside $G$ along these closed curves via the elements of the Cartan subspace.
    Moreover, in Section~\ref{sec-HG}, we 
    associated to such a bending an abstract grafting $S_z$ of $S$ 
    (where $z$ is the grafting parameter) and an equivariant, 1-Lipschitz, and piecewise totally geodesic map 
    \[\tilde Q_z:\tilde S_z\to \X\] from its universal covering 
    $\tilde S_z$ to the symmetric space $\X$ of $G$ which projects to a map $Q_z:S_z\to \rho_z\backslash \X$.

    Note that $G$ is real split and $\tau$ is a regular and positive representation in the sense that it maps 
    (projectivizations of) positive matrices in $\PSL_2(\R)$ to 
    (projectivizations of) totally positive matrices in $G$ (see Section~\ref{sec:fockgoncharov}).
    Then the bent representation of $\pi_1(S)$ is Hitchin, which implies by independent (and different) work of Labourie~\cite{Lab06} and Fock--Goncharov~\cite{FG06} that our equivariant map $\tilde S_z\to\X$ is a quasi-isometric embedding.
    
    In this section we give an upper bound for the multiplicative error of this quasi-isometric embedding and establish
    a more precise version of Theorem \ref{main1}. 
    Our proof does not rely directly on 
    the work of Labourie  and Fock--Goncharov, but follows from the results on totally positive matrices proved in Section~\ref{sec:fockgoncharov}, which were greatly inspired by Fock--Goncharov's use of positivity. 

    \begin{theorem}\label{thm:quasiisom}
     For every $\sigma>0$, 
     there exists $C_{\sigma}>0$ such that
     the following holds.
    
     Consider a closed hyperbolic surface $S$, a multicurve $\gamma^*\subset S$ whose components have length at most $\sigma$, and a grafting parameter $z$ such that all cylinder heights of the abstract grafting $S_z$ are bounded 
     from below by 
     some number $L>0$.
     
     Let us endow $\X$ with the $G$-invariant 
      admissible Finsler metric $\mf F$ and $S_z$ with the pullback of this metric under $Q_z$, denoted by 
      $d_{\tilde S_z}^{\mf F}$.
     Then the grafting map $\tilde Q_z:\wt S_z\rightarrow\X$ is an injective quasi-isometric embedding with multiplicative constant $(1+C_\sigma/(L+1))$ and additive constant $C_{\sigma}$; more precisely, for all $x,y\in\tilde S_z$ we have
     $$\left(1+\frac{C_\sigma}{L+1}\right)^{-1}d^{\mf F}_{\tilde S_z}(x,y)-C_{\sigma}
     \leq d^{\mf F}(\tilde Q_z(x),\tilde Q_z(y)) \leq d^{\mf F}_{\tilde S_z}(x,y).$$
     Moreover, the image $\tilde S_z^\iota=\tilde Q_z(\tilde S_z)$ is $C_\sigma$-Finsler-quasiconvex in the sense that for all $x,y\in\tilde Q_z(\tilde S_z)$, there is a Finsler geodesic from $x$ to $y$ at distance at most $C_\sigma$ from $\tilde S_z^\iota$.
    \end{theorem}  
    The facts that $\tilde Q_z$ is injective and $\tilde S_z^\iota$ is quasi-convex have already been established in Corollaries~\ref{prop:embedding} and~\ref{cor:quasiconvex}.
    
    Note that the upper bound for $d^{\mf F}(\tilde Q_z(x),\tilde Q_z(y))$
    is an immediate consequence of the definition of $d^{\mf F}_{\tilde S_z}$ as the pullback of $d^{\mf F}$.

    The remaining estimate will be obtained as a consequence of Observation~\ref{obs:admissiblemapstoadmissible} 
    and an intermediate proposition stating that 
    the images of admissible paths in $\X$ 
    are quasi-geodesics, with control on the multiplicative error term.
    This intermediate result will be proved using Proposition~\ref{prop:triangle equality for admissible paths} (that admissible paths are quasi-ruled) and Proposition~\ref{prop:inj quantique} (a lower bound on the displacement of admissible paths).

    The collar lemma for hyperbolic surfaces states that  
    for any $\sigma>0$, if 
    \[\omega=\sinh^{-1}\left(\frac{1}{\sinh(\sigma/2)}\right)\] then any simple 
    closed geodesic $\gamma^*\subset S$ of length at most $\sigma$ will have a \emph{collar size} bounded 
    from below by $\omega$, in the sense that $N_\omega(\gamma^*)=\{z\mid d(z,\gamma^*)\leq \omega\}$
    is an annulus. 
    Then Observation~\ref{obs:admissiblemapstoadmissible} says that for any multicurve $\gamma^*\subset S$  with components of length at most~$\sigma$, for any grafting parameter $z$, the image of an admissible path of $\wt S_z$ under $\wt Q_z:\wt S_z\to \X$ will be a $(\omega,0)$-admissible path of $\X$.



    We will also prove the following coarse estimates on lengths.

    \begin{theorem}\label{thm:quasiisom lengths}
        In the setting of Theorem~\ref{thm:quasiisom}, let $(\rho_z)_z$ be the associated family grafted Hitchin representations.
        Then there is $C_\sigma'$ only depending on $\sigma$ such that for any $\gamma\in\pi_1(S)$,
        $$
        \ell^{\mf F}(\rho_z(\gamma)) \geq \frac{L+1}{C'_\sigma}\iota(\gamma,\gamma^*).
        $$
        Moreover, recalling that $z$ is the datum of a vector $z_e\in\mf a$ for each component $e\subset\gamma^*$, then~$C'_\sigma$ may be chosen so that if  $z_e\in \ker(\alpha_0)$ for any $e$ then
        $$
        \ell^{\mf F}(\rho_{z}(\gamma)) \geq  \left(1+\frac{C'_{\sigma}}{L+1}\right)^{-1}\ell_S(\gamma),
        $$
        where $\ell_S(\gamma)$ is the length of $\gamma $ in $S$.
    \end{theorem}

    \subsection{Elementary observations}

The triangle inequality easily yields the following.
    
   \begin{observation}\label{obs:qruled concat of geod}
       For any piecewise geodesic curve $c:[0,T]\to (\X,d^{\mf F})$ with $m\geq 1$ geodesic pieces, that is additionally $C$-quasi-ruled, we have
       $$d^{\mf F}(c(0),c(T))\geq \Len^{\mf F}(c) - (m-1)C.$$
   \end{observation}

We will also need the following technical estimates.

\begin{lemma}\label{lem:horrible computations}
    Consider $a,b,k,L\geq 0$ and $C\geq 1$, such that
    $$a\geq b-kC \quad \text{and}\quad b\geq (k-2)L \quad \text{and}\quad a\geq\frac{k-2}C.$$
    Then
    $$a\geq \left(1+\frac{4(C+1)^3}{L+1}\right)^{-1}b-2C.$$
\end{lemma}
\begin{proof}
We can assume $L>0$.
Use the second equation to get
$$k\leq \tfrac{b}{L}+2$$ 

Consider first the case that
$L\geq 2C\geq 2$. We 
By the first equation, we have
$$ a\geq b(1-\tfrac CL) - 2C$$
with the additional inequality
$$1-\tfrac CL\geq \left(1+2\tfrac{C}{L}\right)^{-1} \geq \left(1+\tfrac{4(C+1)^3}{L+1}\right)^{-1} $$
(indeed one can check that $0\leq x\leq \tfrac12$ implies $(1-x)(1+2x)\geq 1$).

If $L<2C$ then use the third equation to obtain
$$a\geq\frac{b-2C}{1+C^2}\geq\frac{b}{1+C^2}-2C$$
and insert 
\begin{equation*} \left(1+C^2\right)^{-1} \geq  \left(1+\tfrac{4(C+1)^3}{L+1}\right)^{-1}.\qedhere \end{equation*}
\end{proof}




    \subsection{Admissible paths are uniform quasi-geodesics}\label{sec:admissible pathsII}

    To prove that the embedding $\wt Q_z:\wt S_z\to X$ is quasi-isometric,
    by Observation~\ref{obs:admissiblemapstoadmissible} it suffices to show 
    that admissible paths in $\X$ are quasi-geodesics.
    We prove this
    now with uniform constants 
    using the consequences of positivity established in Section~\ref{sec:fockgoncharov} 
    and the elementary observations of the previous section.

    \begin{proposition}\label{thm:quasigeod}
        For all $\omega>0$, 
        there exist $C_{\omega}>0$ such that 
        for every $L\geq 0$,
        all $(\omega,L)$-admissible paths are $\left(1+\tfrac{C_\omega}{(L+1)},C_{\omega}\right)$-quasi-Finsler-geodesics (where the first constant is the multiplicative constant).
    \end{proposition}
    \begin{proof}
        Let $c:[0,T]\to \X$ be an $(\omega,L)$-admissible path.
        It is clear that $d^{\mf F}(c(0),c(T))\leq \Len^{\mf F}(c)$, so we only need to obtain a converse inequality.
        
        By Proposition~\ref{prop:triangle equality for admissible paths}, $c$ is $C_\omega$-quasi-ruled, for some constant $C_\omega$ depending on $\omega$.
        By Observation~\ref{obs:qruled concat of geod} we get
        $$d^{\mf F}(c(0),c(T))\geq \Len^{\mf F}(c)-kC_\omega,$$
        where $k$ is the number of singularities of $c$ (i.e.\ $k+1$ is the number of geodesic pieces).

        Since $c$ contains at least $(k-2)/2$ geodesic pieces of flat type and length at least $L$, we also know that 
        $$\Len^{\mf F}(c)\geq (k-2)\tfrac L2.$$

        Finally by Proposition~\ref{prop:inj quantique} we also have
        $$d^{\mf F}(c(0),c(T))\geq \frac{k-2}{C'_\omega}$$
        for some constant $C'_\omega$ depending on $\omega$.

        We conclude thanks to Lemma~\ref{lem:horrible computations}.
    \end{proof}

\subsection{Proof of Theorem~\ref{thm:quasiisom}}

It is an immediate consequence of Observation~\ref{obs:admissiblemapstoadmissible}, Corollary~\ref{prop:embedding}, Corollary~\ref{cor:quasiconvex} and Proposition~\ref{thm:quasigeod}. 
More precisely, by compactness of $S$, Corollary~\ref{prop:embedding} ensures 
that the natural map $S_z\xrightarrow[]{Q_z}\rho_z\backslash \X$ is an embedding. 

Consider an equivariant lift $\tilde Q_z:\tilde S_z\to \X$ of $Q_z$. By  
Corollary~\ref{cor:quasiconvex}, any two points in $\tilde Q_z(\tilde S_z)$ can be
connected by a Finsler geodesic in $\X$ which remains at 
distance at most $C_\omega$ from $Q_z(S_z)$. 
This establishes the last statement in Theorem~\ref{thm:quasiisom}.
Note that 
these Finsler geodesics then project to 
Finsler geodesics in the quotient $\rho_z\backslash \X$ in the $C_\omega$-neighborhood of 
$Q_z(S_z)$.

To show the distance-length control, note that 
any admissible path in $S_z$ is sent by~$Q_z$ to an $(\omega,L)$-admissible path inside 
$\rho_z\backslash \X$ (Observation~\ref{obs:admissiblemapstoadmissible}), which is therefore a quasi-geodesic (Proposition~\ref{thm:quasigeod}).

\subsection{Proof of Theorem~\ref{thm:quasiisom lengths}}

Fix $[\gamma]\in[\pi_1(S)]$ transverse to $\gamma^*$ and recall that it corresponds to a free homotopy class in the characteristic surface $S_{z}^\iota\subset \rho_z\backslash \X$.
Let $c\subset S^\iota_{z}\subset \rho_z\backslash \X$ be the unique admissible loop in this free homotopy class.
It has a $\rho_z(\gamma)$-invariant lift $\tilde c:\R\to \X$ such that $0$ is a singularity and the geodesic piece of $\tilde c$ starting at time $0$ is of hyperbolic type.

Denote by $T$ the period of $c$, so that $\wt c(t+T)=\rho_z(\gamma)\wt c(t)$ for any $t$.
By Proposition~\ref{thm:quasigeod}, for any $n\geq 1$ we have
$$d^{\mf F}(\wt c(0),\rho_z(\gamma)^n\wt c(0))=d^{\mf F}(\wt c(0),\wt c(nT ))\geq \left(1+\frac{C_\sigma}{L+1}\right)^{-1}n\Len^{\mf F}(c)-C_\sigma.$$
Dividing by $n$ and letting $n\to\infty$ we get
\begin{equation}\label{eq:periodic qgeod}\ell^{\mf F}(\rho_z(\gamma))\geq \left(1+\frac{C_\sigma}{L+1}\right)^{-1}\Len^{\mf F}(c).\end{equation}

We may assume that $\iota(\gamma,\gamma^*)\geq 1$ (the case $\iota(\gamma,\gamma^*)=0$ is trivial). The number $2\iota(\gamma,\gamma^*)$ of singularities of $c$ is even and bounded from below by $2$, and the same is true for the number of geodesic pieces, half of which are of flat type and have length at least $L$, and the other half are of hyperbolic type and have length at least $\omega=\sinh^{-1}(\sinh(\sigma/2)^{-1})$. Consequently, 
$$\ell^{\mf F}(\rho_z(\gamma))\geq \left(1+\frac{C_\sigma}{L+1}\right)^{-1}\frac{L+\omega}2\iota(\gamma,\gamma^*) \geq \frac{L+\omega}{2+2C_\sigma'}\iota(\gamma,\gamma^*).$$

Finally, if for each component $e\subset\gamma^*$ the vector $z_e$ is taken in $\ker(\alpha_0)$ then we can apply Lemma~\ref{lem:Lipschitz}, which says that $\Len^{\mf F}(c)$ is bounded from below by the length of the image of $c$ under the projection maps $S_z\to S$, which is itself greater than or equal to~$\ell_S(\gamma)$, and this means Equation~\ref{eq:periodic qgeod} implies the desired inequality.

\subsection{Proof of Theorem~\ref{main2}}

We begin with the proof of the third part of Theorem \ref{main2}. 
Thus let $S$ be a closed surface with hyperbolic metric $h$ and let $S_0\subset S$ be an essential subsurface, 
bounded by simple closed geodesics $\partial S_0=\{\gamma_1,\dots,\gamma_k\}$.
Consider a one-parameter family $\rho_t$ of Hitchin grafting representations with 
grafting datum $tz$ for a tuple $z=(z_1,\dots,z_k)\in \mathfrak{a}^k$ 
whose components are linearly independent from the direction of 
a tangent vector of $\H\subset \X$. Then for each $t$, the bordered hyperbolic surface $S_0$ is totally geodesic 
embedded in $\rho_t\backslash \X$. 

Choose once and for all a basepoint $\basepoint\in S_0$ and view this as a basepoint in $\rho_t\backslash \X$ for all~$t$. 
By Theorem~\ref{thm:quasiisom lengths} and equivalence of the Riemannian and the Finsler metric, 
for each $R>0$ there exists a number $t=t(R)>0$ so that the shortest closed
geodesic in $\rho_t\backslash \X$ which is not contained in $S_0$ intersects the complement of the $R$-ball about 
$S_0$. Thus for $t>t(R)$, the normal injectivity radius of $S_0$ in $\rho_t\backslash \X$ 
for the locally symmetric Riemannian metric is at least $R$, and the 
ball $B(S_0,R)\subset \rho_t\backslash \X$ 
of radius $R$ about $S_0$ is homotopy equivalent to $S_0$.

By passing to a subsequence, we may assume that the pointed manifolds $(\rho_t\backslash \X,\basepoint)$ converge in the pointed
Gromov Hausdorff topology to a locally symmetric pointed manifold $(N,\basepoint)$. This manifold contains $S_0$ as a totally 
geodesic embedded surface of infinite normal injectivity radius. But this just means that 
$N$ equals the manifold defined by the Fuchsian representation $\rho\vert S_0$. This is precisely the statement of the
last part of Theorem~\ref{main2}.

To show the second part of Theorem \ref{main2},  
choose a non-principal ultrafilter $\omega$ on $\mathbb{R}$ converging to $+\infty$  
and a basepoint $\basepoint$ in $\X$, viewed as the projection of the identity
in $\PSL_d(\R)$. 
Let $\rho_t$ be a Hitchin grafting ray determined by a grafting parameter 
$z\in \mathfrak{a}$ and a simple closed geodesic $\gamma^*\subset S$.
Consider the pointed metric spaces $\X_t=(\X,\basepoint,\frac{1}{t}d^{\mf F})$.
We will see that for any $\gamma\in\pi_1(S)$, the distance in $\X_t$ between $\basepoint$ and $\rho_t(\gamma)\cdot\basepoint$ is bounded independently of $t$. 
Thus by passing to an $\omega$-ultralimit,
we obtain an action $\rho_\infty$
of $\pi_1(S)$ on 
the ultralimit $\X_\infty$ 
of the metric spaces $\X_t$ defined by $\omega$. 
This ultralimit is well known to be a
Euclidean building whose apartments correspond to the maximal flats in $\X$, and it is equipped 
with a Finsler metric. 

We first make some choice that will allow us to embed the Bass--Serre tree $\mc T$ of the graph of groups defined by $\gamma^*\subset S$ in $\X_t$.
This embedding will not be $\rho_t$-equivariant, but its limit as $t\to\infty$ will be natural (independent of the choices made) and equivariant.

Let $\tilde S\simeq \H^2$ be the universal cover of $S$, and $\tilde\gamma^*$ the preimage of $\gamma^*$.
Let $\Sigma$ be the surface with boundaries obtained by cutting $\tilde S$ along $\tilde\gamma^*$.
In each component $Y\subset \Sigma$ we pick a point $p(Y)$, and in each boundary component $B\subset \partial Y\subset \partial \Sigma$ we denote by $p(B)\in B$ the shortest distance projection of $p(Y)$.

For each $t$ we have an abstract grafted surface $\tilde S_t$, in which $\Sigma$ embeds isometrically, and we have a quasi-isometric embedding $Q_t:\tilde S_t\to \tilde S_t^\iota\subset\X$ that maps $p(Y_0)$ (where $Y_0$ is our preferred component of $\Sigma$) to the basepoint $\basepoint\in\X$.
For all $t$, $Y\subset\Sigma$ and $B\subset\partial Y$ we denote $p_t(Y)=Q_t(p(Y))$ and $p_t(B)=Q_t(p(B))$.
Let $\mc T_t\subset \tilde S_t^\iota$ be the union of segments of the form $[p_t(Y),p_t(B)]$ if $B\subset\partial Y$ and $[p_t(B_1),p_t(B_2)]$ if $B_1,B_2\subset\partial \Sigma$ project to the same component of $\tilde\gamma^*$ in $\tilde S$.
Note that $\mc T_t$ is a tree isomorphic to the Basse--Serre tree $\mc T$.

The length of  $[p_t(Y),p_t(B)]$ in $\X$ is independent of $t$, hence its length in $\X_t$ tends to zero as $t\to\infty$.
The segment $[p_t(B_1),p_t(B_2)]$ is conjugate to a segment of the form $[0,v+tz]\subset\mf a$ where $v$ depends on $B_1,B_2$, so its length in $\X_t$ is bounded by a constant independent of $t$.
As a consequence, every $p_t(Y)$, resp.\ $p_t(B)$, converge as $t\to\infty$ to a points $p_\infty(Y)$, resp.\ $p_\infty(B)$, in the building $\X_\infty$.
In fact if $B\subset\partial Y$ then $p_\infty(B)=p_\infty(Y)$.
Moreover, the whole tree $\mc T_t$ converges to a tree $\mc T_\infty\subset\X_\infty$ whose vertices are $(p_\infty(Y))_{Y\subset\Sigma}$, whose edges are straight segments conjugate to $[0,z]\subset\mf a$, and which is also isometric to $\mc T$.

Now observe that $\mc T_\infty$ is invariant under the action of any $\gamma\in\pi_1(S)$: for any $Y\subset\Sigma$ and $t$, the distance in $\X$ between $\rho_t(\gamma)p_t(Y)$ and $p_t(\gamma Y)$ is independent of $t$, and hence the distance in $\X_t$ tends to zero, leading to $\rho_\infty(\gamma)p_\infty(Y)=p_\infty(\gamma Y)$ (and this also proves $\rho_\infty$ is well defined).

Finally, note that $\mc T_\infty$ is Finsler-convex.
Indeed for all $Y,Z\subset\Sigma$ and $t$, the path in $\mc T_t$ from $p_t(Y)$ to $p_t(Z)$ and the admissible path connecting the same two points are within bounded distance independent of $t$.
By Proposition~\ref{prop:triangle equality for admissible paths}, the path in $\mc T_t$ is at distance $\leq K$ from an actual Finsler geodesic, where $K$ does not depend on $t$. Hence in $\X_t$ it is at distance $\leq K/t$ from a Finsler geodesic, and at the limit the path in $\mc T_\infty$ is a Finsler geodesic itself. This yields the statement of the second part of Theorem \ref{main2}. 

The first part of Theorem~\ref{main2} follows immediately.

\bigskip\bigskip

\bigskip\bigskip


	\printbibliography[heading=bibintoc]{}


\noindent
Pierre-Louis Blayac \\
Université de Strasbourg, IRMA, 
7 rue René-Descartes,
67084 Strasbourg, France\\
e-mail: blayac@unistra.fr
\medskip

\noindent
Ursula Hamenst\"adt\\
Math. Institut der Univ. Bonn, Endenicher Allee 60, 53115 Bonn, Germany\\
e-mail: ursula@math.uni-bonn.de

\medskip

\noindent
Théo Marty \\
Institut de mathématiques de Bourgogne, 
9 avenue Alain Savary,
21078 Dijon, France\\
e-mail: theo.marty@u-bourgogne.fr

\end{document}